\documentclass[11pt]{article}

\usepackage{fullpage}
\usepackage[utf8]{inputenc}
\usepackage[T1]{fontenc}
\usepackage{amsmath,amsfonts,amssymb,amsthm,bbm}
\usepackage{stackrel}
\usepackage{enumerate}
\usepackage{graphicx}
\usepackage{hyperref}

\usepackage{latexsym} 
\usepackage{color} 

\usepackage{mathtools} 

\allowdisplaybreaks 

\newtheorem{theorem}{Theorem}
\newtheorem{corollary}[theorem]{Corollary}
\newtheorem{proposition}[theorem]{Proposition}

\newtheorem{lemma}[theorem]{Lemma}
\newtheorem{definition}[theorem]{Definition}

\newtheorem{remark}[theorem]{Remark}

\numberwithin{theorem}{section}
\numberwithin{equation}{section}
\numberwithin{figure}{section}

\newcommand{\ve}{\varepsilon}

\newcommand{\ind}{\mathbbm{1}}

\newcommand{\ul}{\underline}
\newcommand{\ol}{\overline}

\DeclareMathOperator{\sgn}{sgn}
\DeclareMathOperator{\diam}{diam}
\DeclareMathOperator{\inter}{int}

\newcommand{\calB}{\mathcal{B}}

\newcommand{\calS}{\mathcal{S}}

\newcommand{\NN}{\mathbb{N}}
\newcommand{\ZZ}{\mathbb{Z}}
\newcommand{\RR}{\mathbb{R}}
\newcommand{\CC}{\mathbb{C}}

\newcommand{\absset}{\mathfrak{S}} 

\newcommand{\PP}{\mathbb{P}}
\newcommand{\EE}{\mathbb{E}}

\newcommand{\TT}{\mathbb{T}} 
\newcommand{\Ball}{B} 
\newcommand{\Ann}{A} 
\newcommand{\din}{\partial^{\textrm{in}}} 
\newcommand{\dout}{\partial^{\textrm{out}}} 
\newcommand{\cluster}{\mathcal{C}}
\newcommand{\Cinf}{\cluster_\infty} 
\newcommand{\lclus}{\cluster^{\textrm{max}}} 
\newcommand{\hole}{\mathcal{H}} 
\newcommand{\holeF}{\bar{\mathcal{H}}} 
\newcommand{\holeT}{\holeF^{(\TT)}} 
\newcommand{\Ch}{\mathcal{C}_H} 
\newcommand{\Cv}{\mathcal{C}_V} 
\newcommand{\circuit}{\mathcal{C}} 
\newcommand{\cout}{\mathcal{C}^{\textrm{out}}} 
\newcommand{\circuitevent}{\mathcal{O}} 
\newcommand{\circuitinside}{\mathcal{I}} 
\newcommand{\colorseq}{\mathfrak{S}} 
\newcommand{\arm}{\mathcal{A}} 
\newcommand{\frozen}{\mathcal{F}}
\newcommand{\net}{\mathcal{N}} 
\newcommand{\Dint}[2]{{#1}^{\textrm{int}(#2)}}
\newcommand{\Di}[1]{{#1}^{(\textrm{int})}}
\newcommand{\Dext}[2]{{#1}^{\textrm{ext}(#2)}}
\newcommand{\chifin}{\chi^{\textrm{fin}}}
\newcommand{\chicov}{\chi^{\textrm{cov}}}
\newcommand{\Cov}{\textrm{Cov}}
\newcommand{\Var}{\textrm{Var}}

\newcommand{\nextN}{\psi_N}
\newcommand{\distvol}{\mu^{\textrm{vol}}}
\newcommand{\grid}{\mathcal{N}}
\newcommand{\gridn}{\mathcal{N}^{\textrm{nice}}}

\newcommand{\Lra}[1]{\stackrel[#1]{}{\longrightarrow}}
\newcommand{\lra}[1]{\stackrel[]{#1}{\leftrightarrow}}
\newcommand{\nlra}[1]{\stackrel[]{#1}{\nleftrightarrow}}

\newcommand{\new}{\textrm{new}}
\newcommand{\calV}{\mathcal{V}}

\DeclareFontFamily{U}{mathx}{\hyphenchar\font45}
\DeclareFontShape{U}{mathx}{m}{n}{
      <5> <6> <7> <8> <9> <10>
      <10.95> <12> <14.4> <17.28> <20.74> <24.88>
      mathx10
      }{}
\DeclareSymbolFont{mathx}{U}{mathx}{m}{n}
\DeclareFontSubstitution{U}{mathx}{m}{n}
\DeclareMathAccent{\wideparen}{0}{mathx}{"75}

\begin{document}

\title{Two-dimensional volume-frozen percolation:\\ deconcentration and prevalence of mesoscopic clusters}

\author{Jacob van den Berg\footnote{CWI and VU University Amsterdam}, Demeter Kiss\footnote{University of Cambridge and AIMR Tohoku University}, Pierre Nolin\footnote{ETH Z\"urich}}

\date{}


\maketitle

\begin{abstract}
Frozen percolation on the binary tree was introduced by Aldous \cite{Aldous2000} around fifteen years ago, inspired by sol-gel transitions. We investigate a version of the model on the triangular lattice, where connected components stop growing (``freeze'') as soon as they contain at least $N$ vertices, for some parameter $N \geq 1$.

This process has a substantially different behavior from the diameter-frozen process, studied in \cite{BLN2012, Kiss2015}: in particular, we show that many (more and more as $N \to \infty$) frozen clusters surrounding the origin appear successively, each new cluster having a diameter much smaller than the previous one. This separation of scales is instrumental, and it helps to approximate the process in sufficiently large (but not too large), as a function of $N$, finite domains by a Markov chain. This allows us to establish a deconcentration property for the sizes of the holes of the frozen clusters around the origin.

For the full-plane process, we then show that it can be compared to the process in large finite domains, so that the deconcentration property also holds in this case. In particular, we obtain that with high probability (as $N \to \infty$), the origin does not belong to a frozen cluster in the final configuration.

This work requires new properties for near-critical percolation, which we develop along the way, and which are interesting in their own right: in particular, an asymptotic formula involving the percolation probability $\theta(p)$ as $p \searrow p_c$, and regularity properties for large holes in the infinite cluster. Volume-frozen percolation also gives insight into forest-fire processes, where lightning hits independently each tree with a small rate, and burns its entire connected component immediately.

\bigskip

\end{abstract}

\newpage

\tableofcontents

\newpage

\section{Introduction}

\subsection{Frozen percolation} \label{sec:intro_fp}

Frozen percolation is a growth process which was first introduced by Aldous \cite{Aldous2000} on the binary tree, motivated by sol-gel transitions \cite{Stockmayer}. Let us first describe it informally, on an infinite simple graph $G = (V, E)$, where the vertices may be interpreted as particles. We start with all edges closed (i.e. all particles are isolated), and we try to turn them open independently of each other: at some random time $\tau_e$ uniformly distributed between $0$ and $1$, the edge $e \in E$ becomes open if and only if it connects two finite open connected components (otherwise it just stays closed). In other words, a connected components grows until it becomes infinite (i.e. it gelates), at which time it just stops growing: we say that it freezes, which explains the name of the process. Apart from sol-gel transitions, one may think of other interpretations, e.g. population dynamics (group formation), and pattern formation in general. There are, somewhat surprisingly at first sight, also interesting connections with (and potential applications to) forest-fire models (at least in the two-dimensional setting, studied in this paper).

The existence of the frozen percolation process is not clear at all. In \cite{Aldous2000}, Aldous studies the case when $G$ is the infinite $3$-regular tree, as well as the case of the planted binary tree (where all vertices have degree $3$, except the root vertex which has degree $1$): using the tree structure, which allows for explicit computations, he shows that the frozen percolation process does exist in these two cases (and that it exhibits a fascinating form of self-organized critical behavior). However, Benjamini and Schramm noticed soon after Aldous' paper that such a process does not exist on the square lattice $\mathbb{Z}^2$ (see also Remark (i) after Theorem 1 in \cite{Berg2001}).

In order to circumvent this non-existence issue, a ``truncated'' process was introduced in \cite{BLN2012} by de Lima and two of the authors, where a connected component stops growing when it reaches a certain ``size''  $N$, where $N \geq 1$ is some parameter of the process. Formally, the original frozen percolation process corresponds to $N = \infty$, and one would like to understand what happens as $N \to \infty$, in view of the non-existence result.

When $N$ is finite, ``size'' can have various meanings, and in \cite{BLN2012}, the size of a cluster is measured by its diameter. This diameter-frozen process was then further studied by the second author in \cite{Kiss2015}, who established a precise description as $N \to \infty$, which, roughly speaking, can be summarized as follows. Let us fix some $K > 1$, and look at a square of side length $KN$ (centered at $0$): only finitely many frozen clusters appear (the probability that there are more than $k$ such clusters decays exponentially in $k$, uniformly in $N$), and they all freeze in a near-critical window around the percolation threshold $p_c$. In particular, it is shown that the frozen clusters all look like near-critical percolation clusters, with total density converging to $0$ as $N \to \infty$, and with high probability the origin does not belong to a frozen cluster: in the final configuration, a typical point is on a macroscopic non-frozen cluster, i.e. a cluster with diameter of order $N$, but smaller than $N$.

The truncated process on a binary tree is studied in \cite{BKN2012}, where it is shown that the final configuration is completely different: a typical point is either on a frozen cluster (i.e. with diameter $\geq N$), or on a microscopic one (with diameter $O(1)$), but one observes neither macroscopic non-frozen clusters, nor mesoscopic ones. Moreover, the way of measuring the size of a cluster does not really matter in this case: under suitable hypotheses, the process converges (in some weak sense) to Aldous' process as $N \to \infty$.

In the present paper, we go back to the case of a two-dimensional lattice, where we now measure the size of a cluster by the number of vertices that it contains. Throughout the paper, we work with a site version of frozen percolation, on the (planar) triangular lattice $\TT$ (we do this because site percolation on $\TT$ is the planar percolation process for which the most precise results are known, as discussed below). This lattice has vertex set
$$V(\TT) = \{ x+y e^{\pi i/3} \in \mathbb{C} \: : \: x, y \in \ZZ \},$$
and edge set $E(\TT)$ obtained by connecting all pairs $u, v \in V(\TT)$ for which $\| u-v \|_2=1$. If $u, v \in V$ are connected by an edge, i.e. $(u,v)\in E(\TT)$, we say that $u$ and $v$ are neighbors, and we write $u \sim v$.

The independent site percolation process on $\TT$ can be described as follows. We consider a family $(\tau_v)_{v \in V(\TT)}$ of i.i.d random variables, with uniform distribution on $[0,1]$. For $p \in [0,1]$, we say that a vertex $v$ is $p$-black (resp. $p$-white) if $\tau_v \leq p$ (resp. $\tau_v > p$). Then, $p$-black and $p$-white vertices are distributed according to independent site percolation with parameter $p$, where vertices are independently black or white, with respective probabilities $p$ and $1-p$: we denote by $\PP_p$ the corresponding probability measure. Vertices can be grouped into maximal connected components (clusters) of $p$-black sites and $p$-white sites, which defines a partition of $V(\TT)$. It is a celebrated result \cite{Kesten1980} that for all $p \leq p_c := 1/2$, there is a.s. no infinite $p$-black cluster, while for $p > p_c$, there exists a.s. a unique infinite $p$-black cluster. We refer the reader to \cite{Grimmett1999} for an introduction to percolation theory.

We can then define the volume-frozen percolation process itself, based on the same collection $(\tau_v)_{v \in V(\TT)}$. For a subset $A \subseteq V(\TT)$, its volume is the number of vertices that it contains, denoted by $|A|$. Let $G=(V,E)$ be a subgraph of $\TT$, and $N \geq 1$ be a fixed parameter. At time $t=0$, we set all the vertices in $V$ to be white, and as time $t$ evolves from $0$ to $1$, each vertex $v \in V$ can become black at time $t = \tau_v$ only: it is allowed to do so if and only if all the black clusters touching $v$ have a volume strictly smaller than $N$ (otherwise, $v$ stays white until the end, i.e time $t=1$). That is, black clusters are allowed to grow until their volume is larger than or equal to $N$, when their growth is stopped: such a cluster is then said to be frozen. We say that a black vertex is frozen (at a given time) if (at that time) it belongs to a frozen cluster. We use the notation $\PP_N^{(G)}$ for the corresponding probability measure, and we omit the graph $G$ used when it is clear from the context. Note that this process is well-defined: it can be seen as a finite range interacting particle system, thus general theory \cite{Liggett2005} provides existence. Also note that the particular choice of the \emph{uniform} distribution for the $\tau_v$'s is immaterial: any other continuous distribution produces the same process, up to a time change. In fact, when the graph $G$ is finite, the process is essentially as follows. Choose uniformly at random a permutation of the vertices of $G$: one by one, in this chosen order, each vertex is turned black, unless at least one of its neighbors is already contained in a black cluster with volume $\geq N$.

One of our main results shows that for the triangular lattice, the fraction of frozen sites vanishes as $N \to \infty$.

\begin{theorem} \label{thm:full_plane}
For the volume-frozen percolation process on $\TT$ with parameter $N \geq 1$,
\begin{equation}
\PP_N^{(\TT)}(\text{$0$ is frozen at time $1$}) \Lra{N \to \infty} 0.
\end{equation}
\end{theorem}

In fact, the proof of Theorem \ref{thm:full_plane} provides a stronger result (namely a deconcentration property), which shows a substantial difference with the diameter-frozen model, as well as with Aldous' model on the tree: consider two independent realizations of the cluster of $0$ at time $1$ in the frozen percolation process, and denote the larger one by $C_1$, and the smaller one by $C_2$. Then $\frac{|C_1|}{|C_2|} \to + \infty$ in probability as $N \to \infty$. Note that this property immediately implies Theorem \ref{thm:full_plane}, since the ratio of the volumes of two frozen clusters is between $\frac{1}{2}$ and $2$. It also shows that we only observe mesoscopic clusters: for every $M > 1$,
$$\PP_N^{(\TT)} \bigg( M < |\cluster_1(0)| < \frac{N}{M} \bigg) \stackrel[N \to \infty]{}{\longrightarrow} 1.$$
We can also see from the proof of Theorem \ref{thm:full_plane} that as $N \to \infty$, the number of frozen clusters surrounding the origin tends to $\infty$ in probability (another important difference with the diameter-frozen model).

Indirectly, our work relies on the conformal invariance property of critical percolation \cite{Smirnov} and the SLE (Schramm-Loewner Evolution) technology \cite{LSW_2001a, LSW_2001b} (see also \cite{Werner_PC}). A key ingredient for the more refined results about the percolation phase transition is the construction of the scaling limit of near-critical percolation \cite{GPS13b}. Note that our site version of frozen percolation (described above) is the exact analogue of the bond version on $\ZZ^2$: if the above-mentioned ingredients were available in the latter case, all our proofs would be applicable as well.


\subsection{Exceptional scales in volume-frozen percolation}


In \cite{BN15}, we showed the existence of a sequence of exceptional scales $m_k = m_k(N)$, $k \geq 1$, with $m_1(N) = \sqrt{N}$ and $m_{k+1}(N) \gg m_k(N)$ (as $N \to \infty$) for all $k \geq 1$.

Let us denote by $\Ball_n := [-n,n]^2$ the ball of radius $n$ around the origin in the $L^\infty$ norm. The scales $(m_k(N))_{k \geq 1}$ are exceptional in the sense that if we consider the volume-frozen percolation process in $\Ball_m$, for some $m = m(N)$, we get two very different behaviors according to whether $m$ stays close to one of these scales or not. More precisely, we proved in \cite{BN15} that the following dichotomy holds.
\begin{itemize}
\item If $m_k \ll m \ll m_{k+1}$ as $N \to \infty$ for some $k \geq 1$ (i.e. we start between two exceptional scales but far from them), then (w.h.p.) $k$ successive frozen clusters appear around $0$, at (random) times $p_k < p_{k-1} < \ldots < p_1$ (all strictly larger than $p_c = \frac{1}{2}$) such that $m_{i-1} \ll L(p_i) \ll m_i$ (where $L(p)$ is the characteristic length at $p$: see \eqref{eq:def_Lp} below for a precise definition). Moreover, the cluster $\cluster_1(0)$ of the origin at time $1$ satisfies $1 \ll |\cluster_1(0)| \ll N$: in other words, we only see mesoscopic clusters.

\item On the other hand, if $m \asymp m_k$ as $N \to \infty$ (for a given $k \geq 1$), then (w.h.p.) one of the following three situations occurs, each having a probability bounded away from $0$: either there are $k-1$ successive freezings, and $|\cluster_1(0)| < N$ but is $\asymp N$, or there are $k$ successive freezings, and either $|\cluster_1(0)| \geq N$, or $|\cluster_1(0)| \asymp 1$. That is, we only observe macroscopic (frozen and non-frozen) and microscopic clusters.
\end{itemize}
Another significant difference is that in the first case all the frozen clusters appear close to $p_c$, while in the second case freezing can occur on the whole time interval $(p_c,1)$ (as on the binary tree, but note that there are no macroscopic non-frozen clusters on the tree).

These exceptional scales clearly highlight the non-monotonicity of the process, which makes it quite challenging to study: we need to develop specific tools and ideas to study its dynamics. The existence of these exceptional scales also constitutes a big difference with diameter-frozen percolation \cite{Kiss2015}. For the diameter process, there is essentially one characteristic scale ($N$), and frozen clusters typically leave holes which are too small for new frozen clusters to emerge, while for the volume process, most frozen clusters leave holes where new clusters can freeze.

Heuristically, we expect the resulting configuration in the full-plane process to correspond to the first case in the dichotomy, i.e. $m_k \ll m \ll m_{k+1}$. However, we proceed in a different way: we first prove that even if we start close to $m_k$, the successive freezings create enough ``deconcentration'' if $k$ is very large, so that with high probability we end up far away from $m_1$. This yields in particular the following result.

\begin{theorem} \label{thm:large_k}
For all $\ve > 0$, there exists $l \geq 1$ such that for all $k \geq l$, the following holds: if $m(N) \in [m_k(N), m_{k+1}(N)]$ for all sufficiently large $N$,
then
$$\limsup_{N \to \infty} \PP_N^{(\Ball_{m(N)})}(\text{$0$ is frozen at time $1$}) \leq \ve.$$
\end{theorem}

This result is interesting in itself, but it is also an intermediate step to prove Theorem \ref{thm:full_plane}: for that, we ``connect'' the full-plane frozen percolation process with the process in large enough (as a function of $N$) domains. We actually need a more uniform result than Theorem \ref{thm:large_k}, where boxes can be replaced by domains which are ``sufficiently regular'' (see Theorem \ref{thm:large_k_strong} in Section \ref{sec:proof_large_k} for a precise statement).

\subsection{Organization of the paper}
In the first three sections (Sections \ref{sec:near_critical} to \ref{sec:volume}), we collect and develop all the tools from independent percolation which are used in our proofs of Theorems \ref{thm:full_plane} and \ref{thm:large_k}. More specifically, we need results about the near-critical regime, close to the percolation threshold $p_c$.

In Section \ref{sec:near_critical}, we first discuss classical results, and we derive some consequences of these results. We then prove more involved properties, for which the scaling limit of near-critical percolation (see \cite{GPS13b}) is needed. In particular, we establish a formula for the asymptotic behavior of the density $\theta(p)$ of the infinite cluster as $p \searrow p_c$, which is a refinement of one of the central results of Kesten's celebrated paper \cite{Kesten1987}. This improved formula is crucial for enabling us to follow the dynamics of the process.

A central object in our reasonings is the hole of the origin in the infinite cluster (in the supercritical regime $p > p_c$), and we study it further in Section \ref{sec:approximable}, proving continuity (with respect to $p$) and regularity properties which are interesting in themselves. In particular, one of the difficulties is to rule out the existence of certain bottlenecks, which could perturb the future evolution of the process.

In Section \ref{sec:volume}, we discuss and extend several estimates (from \cite{BCKS01}) on the volume of the largest connected component in a finite domain. These estimates are used repeatedly in our proofs, to obtain a good control on the successive freezing times.

We then turn to the frozen percolation process itself. We first study it in finite domains, before analyzing the full-plane process in Section \ref{sec:full_plane}.

In Section \ref{sec:deconcentration}, we discuss the exceptional scales, and we introduce several chains associated with the frozen percolation process in a finite box. One of these chains is an exact Markov chain, and we prove a deconcentration property for it, using an abstract lemma obtained in Section \ref{sec:abstract_deconcentration}.

This deconcentration property is then used in Section \ref{sec:finite_box} to prove Theorem \ref{thm:large_k}. Roughly speaking, we need to know that the number of frozen clusters surrounding the origin is sufficiently large: for instance, we can start with a box with side length between $m_k(N)$ and $m_{k+1}(N)$, for $k$ large enough.

We then establish Theorem \ref{thm:full_plane} in Section \ref{sec:full_plane}, i.e. the asymptotic absence of dust in the full-plane process. For that, we explain how to couple the process in $\TT$ with the process in finite, large enough (as a function of $N$) domains, which allows us to use the results from the previous section. Finally, in Section \ref{sec:related_proc}, we briefly discuss the potential connection with two other natural processes.

\section{Near-critical percolation} \label{sec:near_critical}

Our proofs rely heavily on a precise description of independent percolation near criticality, i.e. on how this model behaves through its phase transition. Before turning to frozen percolation itself in later sections, we first collect all the results that are needed. After fixing notations in Section \ref{sec:notations}, we present properties which have by now become classical, in Section \ref{sec:classical_results}, and we derive a few consequences of these properties in Section \ref{sec:add_lemmas}. We then turn to more specific technical results, in Sections \ref{sec:pi_sigma}, \ref{sec:theta_L} and \ref{sec:proofs_theta_L}. Their proofs are more involved, relying on recent breakthroughs by Garban, Pete, and Schramm \cite{Garban2010, GPS13b}, and (so far) they only ``work'' for site percolation on the triangular lattice.

\subsection{Notations} \label{sec:notations}

In what follows, a \emph{path} is a sequence of vertices, where any two consecutive vertices are neighbors. Two vertices $u$ and $v$ are said to be connected, which we denote by $u \lra{} v$, if there exists a path from $u$ to $v$ on $\TT$ that consists of black sites only (we also consider white connections, but in this case, we always mention explicitly the color). Two subsets $A, B \subseteq V(\TT)$ are said to be connected if there exist $u \in A$ and $v \in B$ which are connected, and we write $A \lra{} B$. For $p > p_c$, the unique infinite $p$-black cluster is denoted by $\Cinf = \Cinf(p)$. We also write $v \lra{} \infty$ for the event $v \in \Cinf$, and we use the notation
$$\theta(p) = \PP_p(0 \lra{} \infty)$$
for the density of $\Cinf$.

For $A \subseteq \TT$, we consider its inner boundary $\din A$, which consists of all the sites in $A$ that are neighbor with a site in $A^c$, and its outer boundary $\dout A =\din (A^c)$, which consists of all the sites in $A^c$ neighbor to a site in $A$. Note that if $A$ is a black cluster, then $\din A$ and $\dout A$ consist of black and white sites, respectively.

For a rectangle $R = [x_1,x_2] \times [y_1,y_2]$ ($x_1 < x_2$, $y_1 < y_2$), we denote by $\Ch(R)$ (resp. $\Cv(R)$) the event that there exists a black path in $R$ that connects the two vertical (resp. horizontal) sides of $R$. We write $\Ch^*(R)$ and $\Cv^*(R)$ for the analogous events with white paths.

For $0 < m < n$, we define the annulus
$$\Ann_{m,n} := \Ball_n \setminus \Ball_m.$$
For $z \in \CC$, we use the short-hand notations $\Ball_r(z) = z + \Ball_r$ and $\Ann_{m,n}(z) = z + \Ann_{m,n}$. For notational convenience, we also allow the value $n = \infty$, writing $\Ann_{m,\infty}(z) = \CC \setminus \Ball_m(z)$. For $A = \Ann_{m,n}(z)$, the event that there exists a black (resp. white) circuit in $A$, i.e. surrounding $\Ball_m(z)$, is denoted by $\circuitevent(A)$ (resp. $\circuitevent^*(A)$), and we often use the outermost such black circuit in $A$, which we denote by $\cout_A$ (we take $\cout_A = \emptyset$ when such a circuit does not exist).

As often when studying near-critical percolation, the so-called arm events play a central role in our proofs. For $k \geq 1$, $\sigma \in \colorseq_k := \{b,w\}^k$ (where we write $b$ and $w$ for black and white, respectively), and an annulus $A$ as above, we define the event $\arm_\sigma(A)$ that there exist $k$ disjoint paths $\gamma_i$ ($1 \leq i \leq k$) in $A$, in counter-clockwise order, each connecting $\Ball_m(z)$ to $\partial \Ball_n(z)$, and such that $\gamma_i$ has color $\sigma_i$ for each $i$. We denote
\begin{equation}
\pi_\sigma(m,n) = \PP_{p_c}\left( \arm_\sigma\left( \Ann_{m,n} \right) \right),
\end{equation}
and we simply write $\pi_\sigma(n)$ for $\pi_\sigma(0,n)$ (for paths starting from a neighbor of the origin). We write $\arm_1$, $\arm^*_1$, $\arm_4$, and $\arm_6$ in the cases when $\sigma$ is $(b)$, $(w)$, $(bwbw)$, and $(bwwbww)$, respectively (and similarly for $\pi_1$, $\pi^*_1$, $\pi_4$, and $\pi_6$). Note that for an annulus $A$, $\circuitevent(A) = (\arm_1^*(A))^c$. For notational convenience, we also write
\begin{equation}
\circuitevent(z;m,n) = \circuitevent( \Ann_{m,n}(z) ) \quad \text{and} \quad \arm_\sigma(z;m,n) = \arm_\sigma( \Ann_{m,n}(z) ),
\end{equation}
where $z$ is assumed to be $0$ when it is omitted.  

We define the characteristic length by
\begin{equation} \label{eq:def_Lp}
L(p) = \max \left\{n>0 \: : \: \PP_p \left( \Ch \left([0,n] \times [0,2n]\right) \right) \geq 0.01 \right\}
\end{equation}
for $p<1/2$, and by $L(p) = L(1-p)$ for $p>1/2$. From the definition above, it is clear that $L(p)$ is piecewise constant, so not continuous, and non-decreasing (resp. non-increasing) on $[0,p_c)$ (resp. $(p_c,1]$). We thus use a slightly different function $\tilde L$ defined as follows. For each discontinuity point $p \neq p_c$ of $L$, we set $\tilde L(p) = L(p)$, and then we extend $\tilde L$ to $(0,1)\setminus\{p_c\}$ by linear interpolation. The function $\tilde L$ has similar properties as $L$, with the additional benefit of being continuous and strictly monotone on $[0,p_c)$ and $(p_c,1]$, which will come handy later. With a slight abuse of notation, in the following we write $L$ for $\tilde L$.

\subsection{Classical results} \label{sec:classical_results}

Here we collect some classical results in near-critical percolation which will be used throughout the paper.
\begin{enumerate}
\item[(i)] \emph{Russo-Seymour-Welsh (RSW) bounds.} For each $k \geq 1$, there exists a constant $\delta_k > 0$ such that
\begin{equation} \label{eq:RSW}
\PP_p \left( \Ch \left([0,kn] \times [0,n]\right) \right) \geq \delta_k \quad \text{and} \quad \PP_p \left( \Ch^* \left([0,kn] \times [0,n]\right) \right) \geq \delta_k
\end{equation}
for all $p \in (0,1)$ and $n \leq L(p)$.

\item[(ii)] \emph{Exponential decay with respect to $L(p)$.} There exist universal constants $c_i, c'_i > 0$ ($i = 1, 2$) such that
\begin{equation} \label{eq:exp_decay}
\PP_p\left(\Cv\left([0,2n] \times [0,n]\right)\right) \leq c_1 e^{-c_2 \frac{n}{L(p)}} \quad \text{and} \quad \PP_p\left(\Ch\left([0,2n] \times [0,n]\right)\right) \geq c'_1 e^{-c'_2 \frac{n}{L(p)}}
\end{equation}
for all $p<1/2$ and $n \geq 1$ (see Lemma 39 in \cite{Nolin2008}). Using standard arguments, it follows from \eqref{eq:exp_decay} that for some $c_3 > 0$: for all $p>1/2$ and $n \geq L(p)$,
\begin{equation} \label{eq:connection_expdecay}
\PP_{p}(\partial \Ball_n \lra{} \infty) \geq 1 - e^{-c_3 \frac{n}{L(p)}}.
\end{equation}
In particular, this implies (see Corollary 41 in \cite{Nolin2008}) that for $p>1/2$,
\begin{equation} \label{eq:equiv_expdecay}
\PP_{p}(0 \lra{} \partial \Ball_{L(p)}) \asymp \theta(p).
\end{equation}

\item[(iii)] \emph{Extendability and quasi-multiplicativity of arm events at criticality.} For all $k \geq 1$ and $\sigma \in \colorseq_k$, there are constants $c_1, c_2 > 0$ (depending on $\sigma$ only) such that
\begin{align} \label{eq:ext}
  c_1 \pi_\sigma(2n_1,n_2)\leq \pi_\sigma(n_1,n_2)\leq c_2 \pi_\sigma(n_1,2n_2)
\end{align}
and
\begin{align} \label{eq:qmult}
  c_1 \pi_\sigma(n_1,n_3) \leq\pi_\sigma(n_1,n_2) \pi_\sigma(n_2,n_3) \leq c_2 \pi_\sigma(n_1,n_3)
\end{align}
for all $0 \leq n_1 \leq n_2 \leq n_3$ (see Propositions 16 and 17 in \cite{Nolin2008}, respectively).

\item[(iv)] \emph{Arm events in near-critical regime.} For every $k \geq 1$, $\sigma \in \colorseq_k$, let $\arm^{p,p'}_\sigma$ denote the modification of the event $\arm_\sigma$ where the black arms are $p$-black, and the white ones are $p'$-white. Then for all $C \geq 1$, there exist constants $c_1, c_2 > 0$ (depending on $k$, $\sigma$, and $C$) such that
\begin{equation} \label{eq:Kesten1}
  c_1 \pi_\sigma(m,n) \leq \PP_p(\arm^{p,p'}_\sigma(m,n)) \leq c_2\pi_\sigma(m,n)
\end{equation}
for all $p,p'\in(0,1)$, and all $m, n \leq C (L(p)\wedge L(p'))$ (see Lemma 6.3 in \cite{Damron2009}, or Lemma 8.4 in \cite{GPS13b}).

\item[(v)] \emph{Lower and upper bounds on the $1$-arm exponent.} There exist universal constants $c_1, c_2, \eta > 0$ such that
\begin{equation} \label{eq:1arm_bound}
c_1 \left(\frac{m}{n}\right)^{1/2} \leq \PP_p(\arm_1(m,n)) \leq c_2 \left(\frac{m}{n}\right)^{\eta}
\end{equation}
for all $m, n > 0$ with $m < n < L(p)$. This implies that for all $k \geq 1$, $\sigma \in \colorseq_k$, and $C \geq 1$, there exist universal constants $c_3, \alpha>0$ such that
\begin{equation} \label{eq:gps}
\PP(\arm^{p,p'}_\sigma(m,n)) \leq c_3 \left(\frac{m}{n}\right)^\alpha
\end{equation}
for all $p,p'$, and $m, n \leq C (L(p)\wedge L(p'))$.

\item[(vi)] \emph{Lower bound on the $4$-arm exponent.} There exist universal constants $c, \delta > 0$ such that
\begin{equation} \label{eq:4arm_bound}
\PP_p(\arm_4(m,n)) \geq c \left(\frac{m}{n}\right)^{2-\delta}
\end{equation}
for all $m, n > 0$ with $m < n < L(p)$ (this is a consequence of Theorem 24 (3) in \cite{Nolin2008}). In particular, there is a universal constant $C'$ such that
\begin{equation} \label{eq:sum_4arm_bound}
\sum_{k=1}^n 2^{2k} \pi_4(2^k) \leq C' 2^{2n} \pi_4(2^n)
\end{equation}
for all $n \geq 1$.

\item[(vii)] \emph{Upper bound on the $6$-arm exponent.} There exist universal constants $c, \delta > 0$ such that
\begin{equation} \label{eq:6arm_bound}
\PP_p(\arm_6(m,n)) \leq c \left(\frac{m}{n}\right)^{2+\delta}
\end{equation}
for all $m, n > 0$ with $m < n < L(p)$ (we refer the reader to Theorem 24 (3) in \cite{Nolin2008}).

\item[(viii)] \emph{Asymptotic equivalences.} The following are central results in \cite{Kesten1987}:
\begin{equation} \label{eq:Kesten_theta_pi}
\theta(p) \asymp \pi_1(L(p))
\end{equation}
as $p \searrow p_c$ (see Theorem 2 of \cite{Kesten1987}, or (7.25) in \cite{Nolin2008}), and
\begin{equation} \label{eq:Kesten_L}
|p-p_c| L(p)^2 \pi_4(L(p)) \asymp 1
\end{equation}
as $p \to p_c$ (see (4.5) in \cite{Kesten1987}, or Proposition 34 of \cite{Nolin2008}).
\end{enumerate}

\subsection{Additional properties} \label{sec:add_lemmas}

%

Let us first give a definition.

\begin{definition} \label{def:net}
For $0 < a < b$, we consider all the horizontal and vertical rectangles of the form
$$\Ball_a(2ax) \cup \Ball_a(2ax'), \quad \text{with } x, x' \in \Ball_{\lceil b / 2a \rceil + 1}, \: x \sim x'$$
(covering the ball $\Ball_{b+2a}$), and we denote by $\net_p(a,b)$ the event that in each of these rectangles, there exists a $p$-black crossing in the long direction.
\end{definition}

Note that $\net_p(a,b)$ implies the existence of a $p$-black cluster $\net$ which ensures that all the $p$-black clusters and all the $p$-white clusters that intersect $\Ball_b$, except $\net$ itself, have a diameter at most $4a$. In the following, such a cluster $\net$ is called a \emph{net}.

\begin{lemma} \label{lem:net}
There exist universal constants $c_1, c_2 > 0$ such that: for all $0 < a < b$ and $p > p_c$,
\begin{equation} \label{eq:net}
\PP( \net_p(a,b) ) \geq 1 - c_1 \left( \frac{b}{a} \right)^2 e^{-c_2 \frac{a}{L(p)}}.
\end{equation}
\end{lemma}

\begin{proof}[Proof of Lemma \ref{lem:net}]
This follows immediately from the exponential decay property \eqref{eq:exp_decay}.
\end{proof}

%
%

We also derive the following lower bound, which is used in the proof of Proposition \ref{prop:coupling_full_plane}.

\begin{lemma} \label{lem:diff_theta_like}
For all $\kappa, \kappa' \geq 1$, there exists a constant $c = c(\kappa, \kappa') > 0$ such that: for all $p, p' > p_c$ with $p < p'$ and $L(p') \geq \kappa'^{-1} L(p)$, all $n \geq \kappa^{-1} L(p)$,
\begin{equation} \label{eq:diff_theta_like}
\PP(0 \lra{p'} \infty, \: 0 \nlra{p} \partial \Ball_n) \geq c \frac{|p' - p|}{|p - p_c|} \theta(p).
\end{equation}
\end{lemma}

\begin{proof}[Proof of Lemma \ref{lem:diff_theta_like}] 
Since the left-hand side of \eqref{eq:diff_theta_like} is increasing in $n$, we can assume that $n=\kappa^{-1}L(p)$. We construct a sub-event of $\{0 \lra{p'} \infty, \: 0 \nlra{p} \partial \Ball_n\}$ for which the desired lower bound holds, as follows. We start with the events 
\begin{align*}
E_1 & := \{ \text{there is a } p \text{-black circuit } \circuit_1 \text{ in } \Ann_{\kappa^{-1} L(p)/4, \kappa^{-1} L(p)/2} \text{ s.t. } 0 \lra{p} \circuit_1\},\\[1mm]
\text{and} \quad E_2 & := \{ \text{there is a } p \text{-black circuit } \circuit_2 \text{ in } \Ann_{\kappa^{-1} L(p), 2 \kappa^{-1} L(p)} \text{ s.t. } \circuit_2 \lra{p} \infty\}.
\end{align*}
These two events are independent, and RSW \eqref{eq:RSW} implies that
$$\PP(E_1) \PP(E_2) \geq c_1 \theta(p)$$
for some constant $c_1 = c_1(\kappa) > 0$. 

If we also introduce
$$\mathcal{W}_p := \{ v \in \Ann_{\kappa^{-1} L(p)/2, \kappa^{-1} L(p)} \: : \: v \text{ is } p \text{-white}, \: \partial v \lra{p} \partial \Ball_{\kappa^{-1} L(p)/4}, \: \partial v \lra{p} \partial \Ball_{2 \kappa^{-1} L(p)} \}$$
(where we denote by $\partial v$ the set of neighbors of $v$), then there exist constants $c_2, c_3 > 0$ (depending only on $\kappa$) such that the event
$$E_3 := \{ \text{there is a } p \text{-white circuit in } \Ann_{\kappa^{-1} L(p)/2, \kappa^{-1} L(p)}\} \cap \{ | \mathcal{W}_p | \geq c_2 L(p)^2 \pi_4(L(p)) \}$$
satisfies: for all $p > p_c$, $\PP(E_3) \geq c_3$. This property follows from standard arguments, and we sketch a proof on Figure \ref{fig:many_pivotals}.

\begin{figure}[t]
\begin{center}
\includegraphics[width=9cm]{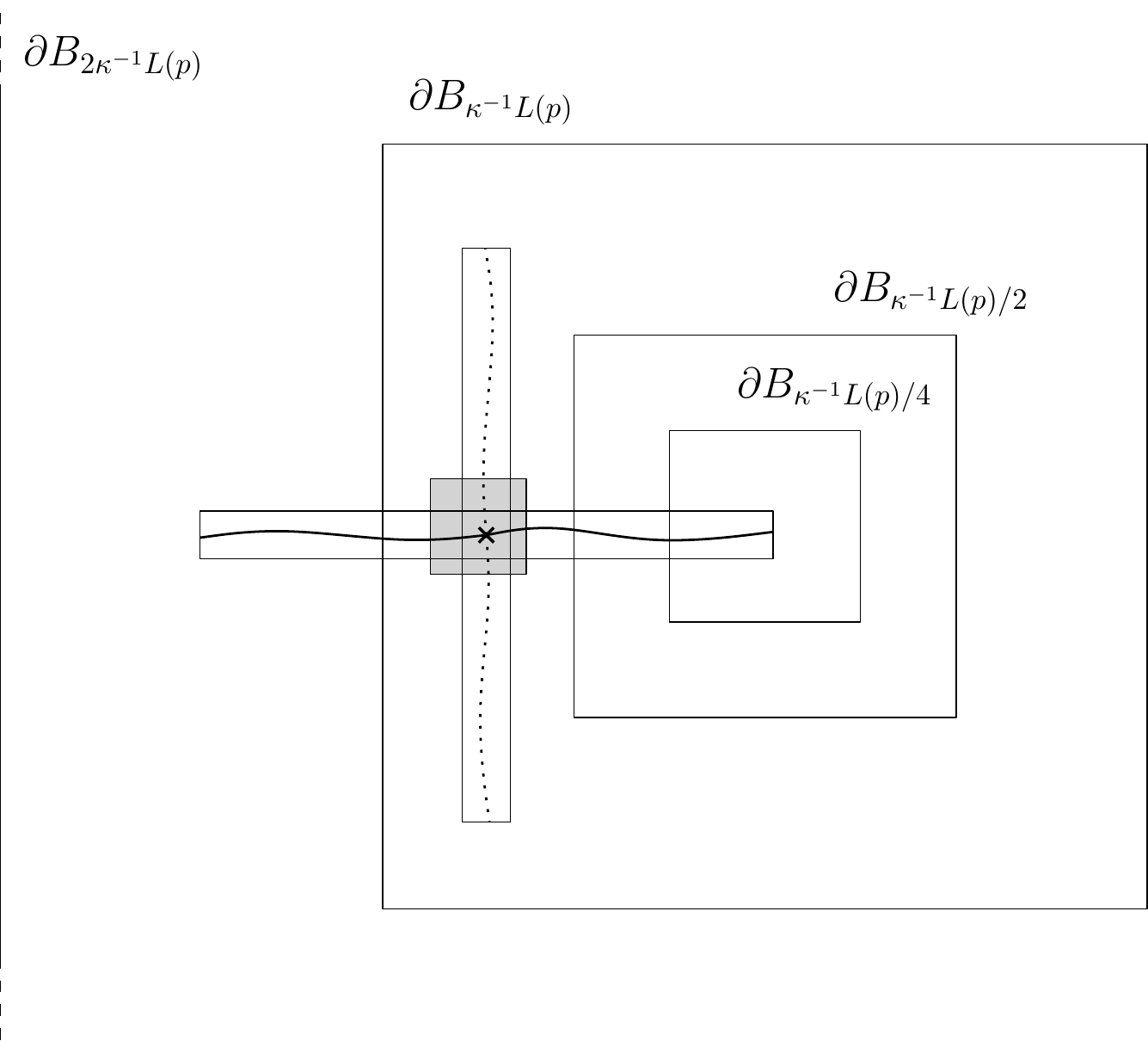}
\caption{\label{fig:many_pivotals} It follows from a second moment argument that with a probability $\geq c'_3 > 0$, there exist $\asymp L(p)^2 \pi_4(L(p))$ vertices in the gray region with four well-separated arms, as depicted. By using RSW \eqref{eq:RSW}, we can then extend the two $p$-white arms into a $p$-white circuit in $\Ann_{\kappa^{-1} L(p)/2, \kappa^{-1} L(p)}$.}
\end{center}
\end{figure}

We now restrict ourselves to the event $E_1 \cap E_2$, we let $\circuit_1$ and $\circuit_2$ be the inner- and outermost circuits appearing in the events $E_1$ and $E_2$, respectively, and we condition on the circuits $\circuit_1$ and $\circuit_2$, as well as on the configuration inside $\circuit_1$ and outside $\circuit_2$. The configuration between $\circuit_1$ and $\circuit_2$ is thus fresh, and we obtain
\begin{equation} \label{eq:lower_bound_E1E2E3}
\PP(E_1 \cap E_2 \cap E_3) = \sum_{C_1, C_2} \PP(E_3 \: | \: \circuit_1 = C_1, \: \circuit_2 = C_2) \PP(\circuit_1 = C_1, \: \circuit_2 = C_2) \geq c_1 \theta(p) \cdot c_3. 
\end{equation}
Using the pivotal vertices produced by the event $E_3$, we deduce that for some $c_4 > 0$,
\begin{align*}
\PP(0 \lra{p'} \infty, \: 0 \nlra{p} \partial \Ball_{\kappa^{-1} L(p)}) & \geq c_4 |p'-p| L(p)^2 \pi_4(L(p)) \PP(E_1 \cap E_2 \cap E_3)\\
& \geq c_1 c_3 c_4 |p'-p| L(p)^2 \pi_4(L(p)) \theta(p)
\end{align*}
(here, the first inequality uses the fact that $|p'-p| L(p)^2 \pi_4(L(p)) \leq c_5 \frac{|p' - p|}{|p - p_c|} \leq c_6$ for some universal $c_5 > 0$, and $c_6 = c_6(\kappa') > 0$, from \eqref{eq:Kesten_L} and the hypothesis on $p$ and $p'$, and the second inequality uses \eqref{eq:lower_bound_E1E2E3}), which completes the proof of Lemma \ref{lem:diff_theta_like} (by applying again \eqref{eq:Kesten_L}).
\end{proof}

Note that Lemma \ref{lem:diff_theta_like} implies in particular the following: there exists a constant $c = c(\kappa') > 0$ such that for all $p, p' > p_c$ with $p < p'$ and $L(p') \geq \kappa'^{-1} L(p)$,
$$\theta(p') - \theta(p) \geq c \frac{|p' - p|}{|p - p_c|} \theta(p)$$
(by fixing one value of $\kappa$, e.g. $\kappa = 1$, and letting $n \to \infty$ in the right-hand side of \eqref{eq:diff_theta_like}). It is also possible to derive a similar upper bound on $\theta(p') - \theta(p)$.

\begin{lemma} \label{lem:diff_theta}
For all $\kappa > 1$, there exists a constant $C = C(\kappa) > 0$ such that: for all $p, p' > p_c$ with $p < p'$ and $L(p') \geq \kappa^{-1} L(p)$, we have
\begin{equation} \label{eq:diff_theta}
\theta(p') - \theta(p) \leq C \frac{|p' - p|}{|p - p_c|} \theta(p).
\end{equation}
\end{lemma}

\begin{proof}[Proof of Lemma \ref{lem:diff_theta}]
Since this result is not used later in the paper, we postpone the proof to Appendix \ref{sec:app_diff_theta}.
\end{proof}

\subsection{Asymptotics of $\pi_\sigma$} \label{sec:pi_sigma}

We now recall some results on the large scale behavior of arm events at criticality. We first remind that their probabilities are described asymptotically by critical exponents, whose values are known (except in the so-called monochromatic case, for $k \geq 2$ arms of the same color). The following result is due to Smirnov and Werner \cite{SmW} (except for the case $k=1$ \cite{LSW_2002}, and for the existence of $\alpha_\sigma$ in the $k \geq 2$ monochromatic case \cite{BeffaraNolin}). Its proof relies on the connection between critical percolation and SLE (Schramm-Loewner Evolution) processes with parameter $6$, which uses the conformal invariance property of critical percolation (in the scaling limit) \cite{Smirnov} and properties of SLE processes \cite{LSW_2001a, LSW_2001b}.

\begin{lemma} \label{lem:arm_exp}
For all $k \geq 1$ and $\sigma \in \colorseq_k$,
$$\pi_\sigma(k,n) = n^{-\alpha_\sigma +o(1)} \quad \text{as $n \to \infty$},$$
for some constant $\alpha_\sigma > 0$. Furthermore,
\begin{itemize}
\item $\alpha_\sigma = \frac{5}{48}$ for $k = 1$,

\item and $\alpha_\sigma = \frac{k^2-1}{12}$ for all $k \geq 2$ and $\sigma \in \colorseq_k$ containing both colors.
\end{itemize}
\end{lemma}

This has the following consequence, known as a \emph{ratio-limit theorem}.

\begin{lemma}[Proposition 4.9 of \cite{Garban2010}] \label{lem:ratio_lim}
For all $k \geq 1$, $\sigma \in \colorseq_k$ and $\lambda > 1$,
$$\lim_{n \to \infty} \frac{\pi_\sigma(k, \lambda n)}{\pi_\sigma(k, n)} = \lambda^{-\alpha_\sigma},$$
where $\alpha_\sigma$ is as in Lemma \ref{lem:arm_exp}.
\end{lemma}

Actually, we make use of a slightly stronger version of this result: the above point-wise convergence holds locally uniformly in $\lambda$.

\begin{lemma} \label{lem:Rob_ratio_lim}
For all $k \geq 1$, $\sigma \in \colorseq_k$, $\bar{\lambda} > 1$ and $\ve > 0$, there exists $K \geq 1$ such that: for all $n' > n \geq K$ with $\frac{n'}{n} \leq \bar{\lambda}$, we have
$$\frac{\pi_\sigma(k, n')}{\pi_\sigma(k, n)} \cdot \left(\frac{n'}{n}\right)^{\alpha_\sigma} \in (1 - \ve, 1 + \ve),$$
where $\alpha_\sigma$ is as in Lemma \ref{lem:arm_exp}.
\end{lemma}

\begin{proof}[Proof of Lemma \ref{lem:Rob_ratio_lim}]
This is a rather immediate consequence of Lemma \ref{lem:ratio_lim}, and the fact that $\pi_\sigma$ is decreasing in its second argument. Indeed, let us write $\bar{\lambda} = (1 + \alpha)^m$, for some $\alpha > 0$ very small, depending on $\ve$ (a precise choice is made later). Lemma \ref{lem:ratio_lim} immediately gives that there exists $K \geq 1$ large enough so that for all $n \geq K$ and all $\lambda \in \mathcal{L}_m := \{ (1 + \alpha)^i \: : \: 1 \leq i \leq m\}$,
\begin{equation} \label{eq:ratio_limit_pw}
\frac{\pi_\sigma(k, \lambda n)}{\pi_\sigma(k, n)} \cdot \lambda^{\alpha_\sigma} \in \bigg(1 - \frac{\ve}{2}, 1 + \frac{\ve}{2} \bigg).
\end{equation}
Now, let us consider any $n \geq K$ and $n' \in (n, \bar{\lambda} n]$: there exists some $i \in \{1, \ldots, m\}$ for which $(1 + \alpha)^i \leq \frac{n'}{n} \leq (1 + \alpha)^{i+1}$, and the monotonicity of $\pi_\sigma$ implies
$$\frac{\pi_\sigma(k, (1 + \alpha)^{i+1} n)}{\pi_\sigma(k, n)} \leq \frac{\pi_\sigma(k, n')}{\pi_\sigma(k, n)} \leq \frac{\pi_\sigma(k, (1 + \alpha)^i n)}{\pi_\sigma(k, n)}.$$
By combining this with \eqref{eq:ratio_limit_pw}, we obtain
$$\bigg(1 - \frac{\ve}{2} \bigg) \big( (1 + \alpha)^{i+1} \big)^{- \alpha_\sigma} \leq \frac{\pi_\sigma(k, n')}{\pi_\sigma(k, n)} \leq \bigg(1 + \frac{\ve}{2} \bigg) \big( (1 + \alpha)^i \big)^{- \alpha_\sigma},$$
and so
$$\bigg(1 - \frac{\ve}{2} \bigg) (1 + \alpha)^{- \alpha_\sigma} \bigg( \frac{n'}{n} \bigg)^{- \alpha_\sigma} \leq \frac{\pi_\sigma(k, n')}{\pi_\sigma(k, n)} \leq \bigg(1 + \frac{\ve}{2} \bigg) (1 + \alpha)^{\alpha_\sigma} \bigg( \frac{n'}{n} \bigg)^{- \alpha_\sigma}.$$
This yields the desired conclusion, by choosing $\alpha = \alpha(\ve)$ small enough.
\end{proof}

\subsection{Near-critical behavior of $\theta(p)$ and $L(p)$} \label{sec:theta_L}

In this section, we state two more specific properties of near-critical percolation. To our knowledge, these results are new, and we believe that they are interesting in themselves. Their proofs are more involved, since they rely on the scaling limit of near-critical percolation \cite{GPS13b} constructed by Garban, Pete and Schramm. However, the results and tools from \cite{GPS13b} are not used elsewhere in the paper, only via the following Proposition \ref{prop:theta_over_pi} and Lemma \ref{lem:str_Kesten}. Therefore, we dedicate a separate section for the proofs of these two results (Section \ref{sec:proofs_theta_L}), which can be skipped at first reading.

Our analysis of volume-frozen percolation relies on locating precisely the successive freezing times, for which we need to closely keep track of the value of $\theta$. It turns out that the classical relation \eqref{eq:Kesten_theta_pi} is not good enough for that purpose, and we make use of the stronger version below.

\begin{proposition} \label{prop:theta_over_pi}
There exists a constant $c_{\theta} \in (0,\infty)$ such that
$$\frac{\theta(p)}{\pi_1(L(p))} \Lra{p \searrow p_c} c_{\theta}.$$
\end{proposition}

We are also interested in the quantity $|p-p_c| L(p)^2 \pi_4(L(p))$ as $p \to p_c$. We already know from \eqref{eq:Kesten_L} that it is $\asymp 1$, but we need that it has actually a limit.

\begin{lemma} \label{lem:str_Kesten}
There exists a constant $c > 0$ such that
\begin{equation} \label{eq:str_Kesten}
|p-p_c| L(p)^2 \pi_4(L(p)) \Lra{p \searrow p_c} c.
\end{equation}
\end{lemma}


\subsection{Proofs of Proposition \ref{prop:theta_over_pi} and Lemma \ref{lem:str_Kesten}} \label{sec:proofs_theta_L}

Before we dive into the proof, we extend our notations to accommodate the triangular lattice at different mesh sizes, as in \cite{Garban2010, GPS13b}. These new notations are used only in this section.

For $\eta>0$, let $\TT^\eta$ be the triangular lattice with mesh size $\eta$ (i.e. the lattice obtained by scaling $\TT$ with a factor $\eta$). For all the quantities defined so far, we add a superscript $\eta$ to indicate the dependence on the mesh size. In particular, $\PP^\eta_p$ refers to site percolation on $\TT^\eta$ with parameter $p$. Note that
$$L^{\eta}(p) = \eta L(p), \quad \text{and} \quad \pi_j^\eta(a,b) = \pi_j(a\eta^{-1}, b\eta^{-1})$$
for all $j \in \{1, 4\}$ and $0 < a < b$.

We make use of the following near-critical parameter scale: for $\lambda \in \RR$, we set 
\begin{equation} \label{eq:near_crit_param}
p_\lambda(\eta) := p_c + \lambda \frac{\eta^2}{\pi^\eta_4(\eta,1)}.
\end{equation}
We use the short-hand $\PP^{\eta,\lambda} := \PP_{p_\lambda(\eta)}^\eta$, and we extend the notation $\pi_j^{\eta}(a,b)$ by
$$\pi_j^{\eta,\lambda}(a,b) := \PP^{\eta,\lambda}(\arm_j(a,b))$$
for $j \in \{1,4\}$ and $0 < a < b$. Finally, we set (with a slight abuse of notation)
\begin{equation} \label{eq:def_Llambda}
L^\eta(\lambda) := L^\eta(p_\lambda(\eta)).
\end{equation}

First, let us recall some results from \cite{Garban2010} and \cite{GPS13b} (where we restrict to the cases that we need, namely $j = 1$ or $4$, although we believe it to hold for other cases as well).

\begin{theorem} \label{thm:ncrit_scaling_lim}
The quantities $L^\eta(\lambda)$ and $\pi_j^{\eta,\lambda}(a,b)$ converge to some continuum analogs as $\eta \to 0$:
\begin{itemize}
\item[(i)] for any $\lambda \in \RR$,
\begin{equation}
L^\eta(p_{\lambda}(\eta)) = L^\eta(\lambda) \stackrel[\eta \to 0]{}{\longrightarrow} L^0(\lambda),
\end{equation}

\item[(ii)] and for all $j \in \{1,4\}$ and $0 < a < b$,
\begin{equation} \label{eq:cv_pi_eta_lambda}
\pi_j^{\eta,\lambda}(a,b) \stackrel[\eta \to 0]{}{\longrightarrow} \pi_j^{0,\lambda}(a,b).
\end{equation}
\end{itemize}
\end{theorem}

This result comes from the fact that for any fixed $\lambda \in \RR$, the percolation model on $\TT^\eta$ with parameter $p_\lambda(\eta)$ converges in distribution to the continuum near-critical percolation model as $\eta \to 0$ (in the quad-crossing topology). More precisely, (i) follows from Theorem 9.4 of \cite{GPS13b}, while (ii) follows from the same theorem combined with Lemma 2.9 of \cite{Garban2010}.


\begin{theorem}[Theorem 10.3 and Corollary 10.5 of \cite{GPS13b}] \label{thm:scale_invar}
For all $\lambda \in \RR \setminus \{0\}$, $j \in \{1,4\}$ and $0 < a < b$,
\begin{align*}
L^0(\lambda) & = |\lambda|^{-4/3} L^0(1) \in (0,\infty),\\[1mm]
\text{and} \quad \pi^{0,\lambda}_j(a,b) & = \pi^{0, \sgn(\lambda)}_j(|\lambda|^{4/3} a, |\lambda|^{4/3} b).
\end{align*}
In particular,
$$\pi^{0,1}_j(a L^0(1), b L^0(1)) = \pi^{0,\lambda}_j(a L^0(\lambda), b L^0(\lambda))$$
for all $\lambda > 0$ and $0 < a < b$.
\end{theorem}

Let us also remind that conformal invariance of critical percolation in the scaling limit (see Theorem 7 of \cite{Camia2006}) implies that
\begin{equation} \label{eq:crit_cinvar}
\pi^{0}_j(a, b) = \pi^{0}_j(s a, s b)  
\end{equation}
for all $j \in \{1,4\}$, $0 < a < b$ and $s > 0$ (where we write $\pi^{0}_j$ for $\pi^{0,0}_j$).

The results above rely on the following ratio-limit theorem.

\begin{proposition} \label{prop:ratio_lim_gen}
For any fixed $a, r > 0$ and $\lambda \in \RR$, there exists a constant $c_\lambda(r, a) > 0$ such that
$$\lim_{\eta \to 0} \frac{\pi^{\eta, \lambda}_j (\eta, r)}{\pi^{\eta, \lambda}_j (\eta, a)} = \lim_{\ve \to 0} \frac{\pi^{0, \lambda}_j (\ve, r)}{\pi^{0, \lambda}_j (\ve, a)} = c_\lambda(r, a)$$
for all $j \in \{1, 4\}$. In the case where $\lambda = 0$, one has $c_0(r, a) = (\frac{r}{a})^{-\alpha_j}$, with $\alpha_1 = \frac{5}{48}$ and $\alpha_4 = \frac{5}{4}$.
\end{proposition}

\begin{proof}[Proof of Proposition \ref{prop:ratio_lim_gen}] The case $\lambda = 0$ coincides with Proposition 4.9 of \cite{Garban2010}. The case $\lambda \neq 0$ follows from a combination of the proof of that proposition, and \eqref{eq:cv_pi_eta_lambda}.
\end{proof}

The following result is a key lemma in \cite{GPS13b}.

\begin{lemma}[Lemma 8.4 of \cite{GPS13b}] \label{lem:ncrit_stab}
For all $\lambda \in \RR$ and $j \in \{1,4\}$, there exist constants $0 < c < C < \infty$ and $\eta_0$ (depending on $\lambda$ and $j$) such that: for all $\eta \leq \eta_0$,
$$c \leq \frac{\pi^{\eta, \lambda}_j (\eta, 1)}{\pi^{\eta}_j (\eta,1 )} \leq C.$$
\end{lemma}

Before we proceed to the proof of Proposition \ref{prop:theta_over_pi}, we need a few more results.

\begin{lemma} \label{lem:crit_ncrit_arm}
For all $\lambda \in \RR$ and $j \in \{1, 4\}$, there exists a constant $C = C(\lambda, j) > 0$ such that for all $a \in (0, 1)$, we have:
\begin{equation} \label{eq:crit_ncrit_arm}
\Bigg| \frac{\pi^{\eta, \lambda}_j (\eta, a)}{\pi^{\eta}_j (\eta, a)} - 1 \Bigg| \leq C a^2 \big( \pi^{\eta}_4 (a, 1) \big)^{-1}
\end{equation}
for all $\eta \leq \eta_0 = \eta_0(\lambda, j, a)$.
\end{lemma}

\begin{proof}[Proof of Lemma \ref{lem:crit_ncrit_arm}] We suppose that $\lambda > 0$ and $j = 1$, since the cases when $\lambda < 0$ or $j = 4$ can be treated in a similar way. We note that
$$\pi^{\eta, \lambda}_j (\eta, a) - \pi^{\eta}_j (\eta, a) = \PP^\eta ( \calB ),$$
where $\calB$ is the event that there exists a $p_\lambda(\eta)$-black arm in $\Ann_{\eta,a}$, but no $p_c$-black arm. If $\calB$ occurs, there exists a vertex which lies on a $p_c$-white circuit in $\Ann_{\eta,a}$, as well as on a $p_\lambda(\eta)$-black arm crossing that annulus. Among the vertices having this property, let $v$ be the one which is closest to the origin (if there are multiple choices, we pick one by using some deterministic procedure). This vertex $v$ is then $p_c$-white and $p_\lambda(\eta)$-black, and we see four disjoint arms around $v$: two $p_\lambda(\eta)$-black and two $p_c$-white arms, starting from $v$ and reaching a distance $d = d(v, \{0\} \cup \partial \Ball_a)$.

In order to obtain an upper bound on $\PP^\eta ( \calB )$, we distinguish two cases, depending on the distance from $v$ to the origin: we introduce the two sub-events
$$\calB_1 := \Big\{ d(0,v) \leq \frac{a}{2} \Big\} \subseteq \calB \quad \text{and} \quad \calB_2:=\calB \setminus \calB_1.$$

We start by bounding the probability of $\calB_1$. Let $i_{\textrm{max}} := \big\lceil \log_2 \big( \frac{a}{2\eta} \big) \big\rceil$. By dividing the annulus $\Ann_{\eta,a}$ into the annuli $A_i = \Ann_{2^{i-1} \eta, 2^i \eta}$ ($1 \leq i \leq i_{\textrm{max}}$), we obtain
\begin{align*}
\PP^\eta (\calB_1) & \leq \sum_{i=1}^{i_{\textrm{max}}} \PP^\eta(v \in A_i)\\
& \leq | p_\lambda(\eta) - p_c | \sum_{i=1}^{i_{\textrm{max}}} |A_i| \cdot \PP^\eta_{p_\lambda(\eta)} \big( \arm_1(\eta, 2^{i-2} \eta) \big) \cdot \PP^\eta \big( \arm_4^{p_\lambda(\eta), p_c}(\eta, 2^{i-1} \eta) \big) \cdot \PP^\eta_{p_\lambda(\eta)} \big( \arm_1(2^{i+1} \eta, a) \big)\\
& \leq C_1 | p_\lambda(\eta) - p_c | \pi_1^{\eta}(\eta, a) \sum_{i=1}^{i_{\textrm{max}}} 2^{2i+2} \pi_4^{\eta}(\eta, 2^{i-1} \eta),
\end{align*}
for some constant $C_1 = C_1(\lambda) > 0$ (using \eqref{eq:Kesten1}, \eqref{eq:ext} and \eqref{eq:qmult}). Hence,
$$\PP^\eta (\calB_1) \leq C_2 | p_\lambda(\eta) - p_c | \pi_1^{\eta}(\eta, a) \Big( \frac{a}{\eta} \Big)^2 \pi_4^{\eta}(\eta,a) = C_2 \pi_1^{\eta}(\eta,a) a^2 \frac{\pi_4^{\eta}(\eta,a)}{\pi_4^{\eta}(\eta,1)} \lambda$$
for some $C_2 > 0$, where the inequality uses \eqref{eq:sum_4arm_bound}, and the equality uses the definition of $p_\lambda(\eta)$ \eqref{eq:near_crit_param}. Using finally \eqref{eq:qmult}, we obtain that for some $C_3 = C_3(\lambda) > 0$,
\begin{equation} \label{eq:bound_calB_1}
\PP^\eta (\calB_1) \leq C_3 \pi_1^{\eta}(\eta,a) a^2 \big( \pi_4^{\eta}(a,1) \big)^{-1}.
\end{equation}

We can bound the probability of $\calB_2$ in a similar way, and obtain that
\begin{equation} \label{eq:bound_calB_2} 
\PP^\eta (\calB_2) \leq C_4 \pi_1^{\eta}(\eta,a) a^2 \big( \pi_4^{\eta}(a,1) \big)^{-1},
\end{equation}
where $C_4 = C_4(\lambda) > 0$. The computation in this case is slightly more complicated: when $v$ is close to $\partial \Ball_a$, we only have $4$ short arms, but we get $3$ long arms in a half plane, unless $v$ is close to a corner, in which case we have $2$ long arms in a quarter plane. Since the corresponding exponents $\alpha^{\textrm{hp}}_3 = 2$ ($3$ arms in a half plane) and $\alpha^{\textrm{qp}}_2 = 2$ ($2$ arms in a quarter plane) are larger than the $4$-arm exponent $\alpha_4$, the summations above can be adapted to this case. A combination of \eqref{eq:bound_calB_1} and \eqref{eq:bound_calB_2} then finishes the proof of Lemma \ref{lem:crit_ncrit_arm}.
\end{proof}

We can now obtain the following result.

\begin{lemma} \label{lem:ncrit_div_crit_arm}
For all $\lambda \in \RR$ and $r > 0$, there exists a constant $c = c(r, \lambda) > 0$ such that
\begin{equation} \label{eq:ncrit_div_crit_arm}
\lim_{\eta \to 0} \frac{\pi^{\eta, \lambda}_1 (\eta, r)}{\pi^{\eta}_1 (\eta, r)} = c.
\end{equation}
\end{lemma}

\begin{proof}[Proof of Lemma \ref{lem:ncrit_div_crit_arm}]
By Lemma \ref{lem:crit_ncrit_arm}, the ratio in the left-hand side of \eqref{eq:ncrit_div_crit_arm} is bounded. We can rewrite it as
$$\frac{\pi^{\eta, \lambda}_1 (\eta, r)}{\pi^{\eta, \lambda}_1 (\eta, \delta)} \cdot \frac{\pi^{\eta, \lambda}_1 (\eta, \delta)}{\pi^{\eta}_1 (\eta, \delta)} \cdot \bigg( \frac{\pi^{\eta}_1 (\eta, r)}{\pi^{\eta}_1 (\eta, \delta)} \bigg)^{-1}$$
for some $\delta \in (0, r \wedge 1)$. By Proposition \ref{prop:ratio_lim_gen}, the first and third terms above converge to $c_\lambda (r, \delta)$ and $c_0 (r, \delta)$ as $\eta \to 0$, respectively. Furthermore, it follows from Lemma \ref{lem:crit_ncrit_arm} that the middle term is $1 + O(\delta^{\alpha})$, uniformly in $\eta$. Hence, the ratio in the left-hand side of \eqref{eq:ncrit_div_crit_arm} is a Cauchy sequence in $\eta$, which finishes the proof of the lemma.
\end{proof}

\begin{lemma} \label{lem:cont_S}
The function
\begin{equation} \label{def_S}
S_1(r) := \lim_{\eta \to 0} \frac{\pi^{\eta, 1}_1 (\eta, r)}{\pi^{\eta}_1 (\eta, r)}
\end{equation} 
is continuous on $(0,\infty)$.
\end{lemma}

\begin{proof}[Proof of Lemma \ref{lem:cont_S}]
Note that it is sufficient to prove the following property: for any fixed $r > 0$,
\begin{equation} \label{eq:continuity_pi1}
\pi^{\eta}_1 \big( \eta, r(1+ \delta) \big) = \pi^{\eta}_1 (\eta, r) \big( 1 + O(\delta) \big) \quad \text{as $\delta \to 0$},
\end{equation}
uniformly in $\eta > 0$, and similarly with $\pi^{\eta, 1}_1$. We now prove this property. To fix ideas, we consider $\pi^{\eta}_1$ and $\delta > 0$. We have
$$\pi^{\eta}_1 (\eta, r) - \pi^{\eta}_1 \big( \eta, r(1+ \delta) \big) = \PP^\eta ( \calB ),$$
where $\calB$ is the event that there exist a $p_c$-black arm in $\Ann_{\eta,r}$ and a $p_c$-white circuit in $\Ann_{\eta,r(1+\delta)}$. Note that $\calB \subseteq \calB_1 \cap \calB_2$, where
\begin{itemize}
\item $\calB_1 := \{$there exists a $p_c$-black arm in $\Ann_{\eta,\frac{r}{4}} \}$,

\item and $\calB_2 := \{$there is a box of the form $z + \Ball_{\frac{\delta r}{2}} \subseteq \Ann_{r, r(1+ \delta)}$ which is connected to distance $\frac{r}{2}$ by three arms (two $p_c$-white, and one $p_c$-black) staying in $\Ann_{\eta,r(1+\delta)} \}$ (in particular, these three arms are contained in a ``half plane'').
\end{itemize}
Since $\calB_1$ and $\calB_2$ are independent, we can then obtain \eqref{eq:continuity_pi1} by using the exponent $\alpha^{\textrm{hp}}_3 = 2$ for the existence of $3$ arms in a half plane (we only need to consider of order $\delta^{-1}$ boxes in the event $\calB_2$).
\end{proof}

\begin{corollary} \label{cor:conv_S}
For all $\lambda, r > 0$,
$$\lim_{\eta \to 0} \frac{\pi^{\eta, \lambda}_1 (\eta, r)}{\pi^{\eta}_1 (\eta, r)} = S_1 \bigg( \frac{L^0(1)}{L^0(\lambda)} r \bigg) = S_1 \big( \lambda^{4/3} r \big).$$
\end{corollary}

\begin{proof}[Proof of Corollary \ref{cor:conv_S}]
This follows by combining Theorem \ref{thm:scale_invar}, \eqref{eq:crit_cinvar} and Lemma \ref{lem:ncrit_div_crit_arm}.
\end{proof}

\begin{lemma} \label{lem:ncrit_theta}
There exists a universal constant $c_2 > 0$ such that: for all $\lambda > 0$,
\begin{equation} \label{eq:ncrit_theta}
\lim_{\eta \to 0} \frac{\theta (p_\lambda(\eta))}{\pi^{\eta, \lambda}_1(\eta, L^{\eta}(\lambda))} = c_2.
\end{equation}
\end{lemma}

\begin{proof}[Proof of Lemma \ref{lem:ncrit_theta}]
First, it follows from \eqref{eq:equiv_expdecay} and \eqref{eq:def_Llambda} that the ratio in the left-hand side of \eqref{eq:ncrit_theta} is bounded away from $0$ and $\infty$. We rewrite it as
$$\frac{\theta (p_\lambda(\eta))}{\pi^{\eta, \lambda}_1 (\eta, L^{\eta}(\lambda))} = \frac{\theta (p_\lambda(\eta))}{\pi^{\eta, \lambda}_1 (\eta, K L^{\eta}(\lambda))} \cdot \frac{\pi^{\eta, \lambda}_1 (\eta, K L^{\eta}(\lambda))}{\pi^{\eta, \lambda}_1 (\eta, L^{\eta}(\lambda))}$$
for some $K > 1$. It then follows from \eqref{eq:exp_decay} that the first term above is $1 + o(1)$ as $K \to \infty$, uniformly for small $\eta$, while for every fixed $K$, the second term converges to $c(\lambda, K)$ as $\eta \to 0$ (by Proposition \ref{prop:ratio_lim_gen}). This implies the existence of the limit in the left-hand side of \eqref{eq:ncrit_theta}.

By combining Corollary \ref{cor:conv_S}, Lemma \ref{lem:cont_S}, Proposition \ref{prop:ratio_lim_gen} and Theorem \ref{thm:ncrit_scaling_lim}, we then obtain
$$\lim_{\eta \to 0} \frac{\pi^{\eta, \lambda}_1 (\eta, K L^{\eta} (\lambda))}{\pi^{\eta, \lambda}_1 (\eta, L^{\eta} (\lambda))} = \frac{S_1 (K L^0 (1))}{S_1 (L^0(1))} \cdot K^{-5/48}.$$
Hence, the limit above is independent of $\lambda$, which completes the proof of Lemma \ref{lem:ncrit_theta}.
\end{proof}

We are now ready to prove Proposition \ref{prop:theta_over_pi}.

\begin{proof}[Proof of Proposition \ref{prop:theta_over_pi}]
Let $p > p_c$, and set $\eta = \eta(p) = \frac{L^0(1)}{L(p)}$. Note that $\eta(p) \to 0$ as $p\searrow p_c$. Further, we set $\lambda = \lambda(p)$ such that $p_{\lambda(p)}(\eta(p)) = p$. Hence $L^\eta(\lambda(p)) = L^0(1)$. Theorems \ref{thm:ncrit_scaling_lim} and \ref{thm:scale_invar} imply that $\lambda(p) \to 1$ as $p \searrow p_c$. Let $\ve>0$ and $p_0 > p_c$ such that $|\lambda(p) - 1| < \ve$ for all $p \in (p_c, p_0)$. Then 
$$\frac{\theta(p_{1-\ve} (\eta))}{\pi^{\eta,1+\ve}_1 (\eta, L^{\eta}(1+\ve))} \leq \frac{\theta(p)}{\pi^{\eta, \lambda(p)}_1 (\eta, L^{\eta}(\lambda(p)))} \leq \frac{\theta(p_{1+\ve}(\eta))}{\pi^{\eta, 1-\ve}_1 (\eta, L^{\eta}(1-\ve))}.$$
If we take the limit as $\eta \to 0$, Lemma \ref{lem:ncrit_theta} (combined with Corollary \ref{cor:conv_S}, Lemma \ref{lem:cont_S}, Theorem \ref{thm:ncrit_scaling_lim} and Lemma \ref{lem:ratio_lim}) implies that the lower and the upper bounds become $c_2 \big( \frac{1-\ve}{1+\ve} \big)^{5/36}$ and $c_2 \big( \frac{1+\ve}{1-\ve} \big)^{5/36}$, respectively, where $c_2$ is as in Lemma \ref{lem:ncrit_theta}.
Since $\ve > 0$ was arbitrary, we obtain
\begin{equation} \label{eq:theta_over_pi_1}
\lim_{p \searrow p_c}\frac{\theta (p)}{\pi^{\eta, \lambda(p)}_1(\eta, L^{\eta} (\lambda(p)))} = c_2.
\end{equation}
In a similar way, we have 
\begin{equation} \label{eq:theta_over_pi_2}
\lim_{p \searrow p_c} \frac{\pi^{\eta, \lambda(p)}_1(\eta, L^0 (1))}{\pi^{\eta,1 }_1 (\eta, L^0(1))} = 1.
\end{equation}
Now, recall that $L^\eta (\lambda(p)) = L^0 (1)$, $L(p) = L^1 (p) = \eta^{-1} L^\eta (p)$ and $\pi(a, b) = \pi^{\eta}(\eta a, \eta b)$. We can thus complete the proof of Proposition \ref{prop:theta_over_pi} by combining \eqref{eq:theta_over_pi_1}, \eqref{eq:theta_over_pi_2} and Lemma \ref{lem:ncrit_div_crit_arm}.
\end{proof}

We now turn to the proof of Lemma \ref{lem:str_Kesten}.

\begin{proof}[Proof of Lemma \ref{lem:str_Kesten}]
We still use the same notations, in particular we consider $\lambda, \eta > 0$ and $p_\lambda (\eta)$ as in \eqref{eq:near_crit_param}. For $p = p_\lambda (\eta)$, the left-hand side of \eqref{eq:str_Kesten} is equal to
$$|p_\lambda(\eta) - p_c| L(p_\lambda(\eta))^2 \pi_4( L(p_\lambda(\eta))) = \bigg( \frac{L(p_\lambda (\eta))}{\eta^{-1}} \bigg)^2 \cdot \frac{\pi_4 (L (p_\lambda(\eta)))}{\pi_4 (\eta^{-1})} \cdot \lambda.$$
Let us now take the limits as $\eta \to 0$: the first factor in the right-hand side converges to $(L^0(\lambda))^2$ by Theorem \ref{thm:ncrit_scaling_lim}, and the second factor converges to $L^0(\lambda)^{-5/4}$, by the same theorem combined with Lemma \ref{lem:ratio_lim} (ratio-limit theorem for $\pi_4$). Hence, 
$$\lim_{\eta \searrow 0} |p_\lambda(\eta) - p_c| L(p_\lambda(\eta))^2 \pi_4( L(p_\lambda(\eta))) = (L^0(\lambda))^{3/4} \lambda,$$
which is constant in $\lambda > 0$ by Theorem \ref{thm:scale_invar}. From this, Lemma \ref{lem:str_Kesten} follows easily.
\end{proof}

In particular, we can compare precisely $L(p)$ and $L(p')$ when $p$ and $p'$ are close to $p_c$.

\begin{lemma} \label{lem:p-L(p)}
For all $\ve>0$ and $\bar{\lambda} > 1$, there exists $p_0 > p_c$ such that: for all $p, p' \in (p_c, p_0)$ with $\frac{L(p)}{L(p')} \in (\bar{\lambda}^{-1}, \bar{\lambda})$, we have
$$\left(\frac{L(p)}{L(p')}\right)^{3/4} \cdot \frac{p-p_c}{p'-p_c} \in (1 - \ve, 1 + \ve).$$
\end{lemma}

\begin{proof}[Proof of Lemma \ref{lem:p-L(p)}]
This follows by combining Lemma \ref{lem:str_Kesten} with Lemma \ref{lem:Rob_ratio_lim}.
\end{proof}

\section{Holes in supercritical percolation} \label{sec:approximable}

\subsection{Definition and a-priori estimates} \label{sec:holes_estimates}

In the supercritical regime $p > p_c$, the unique infinite black cluster $\Cinf(p)$ either contains the origin, or surrounds it (a.s.). In the latter case, the origin lies in a ``hole'': this geometric object plays an important role to study the successive freezings. In this section, we prove estimates on this hole, which is defined formally as follows.

\begin{definition} \label{def:hole_cinf}
We call \emph{hole of the origin} at time $p > p_c$, denoted by $\hole(p)$, the connected component of $0$ in $\TT \setminus (\Cinf(p) \cup \dout \Cinf(p))$, i.e. when we remove the vertices which are neither in the infinite black cluster, nor neighbor of it. By convention, we take $\hole(p) = \emptyset$ if $0$ belongs to $\Cinf(p) \cup \dout \Cinf(p)$.
\end{definition}

This is a natural object to define, in the following sense. Given $\Cinf(p)$, the region outside this cluster and its boundary is exactly the region where we have no information on the colors of the vertices. So in other words, $\hole(p)$ is the connected component of $0$ in this ``terra incognita''. Note that clearly, $\hole(p) \supseteq \hole(p')$ for $p_c < p < p'$. The reason why we remove an extra layer of white sites along the boundary of $\Cinf(p)$ also comes from the connection with frozen percolation, where white vertices along the boundary of a frozen cluster are not allowed to become black at a later time (see Definition \ref{def:hole_fp} below).

\begin{figure}
\begin{center}
\includegraphics[width=12cm]{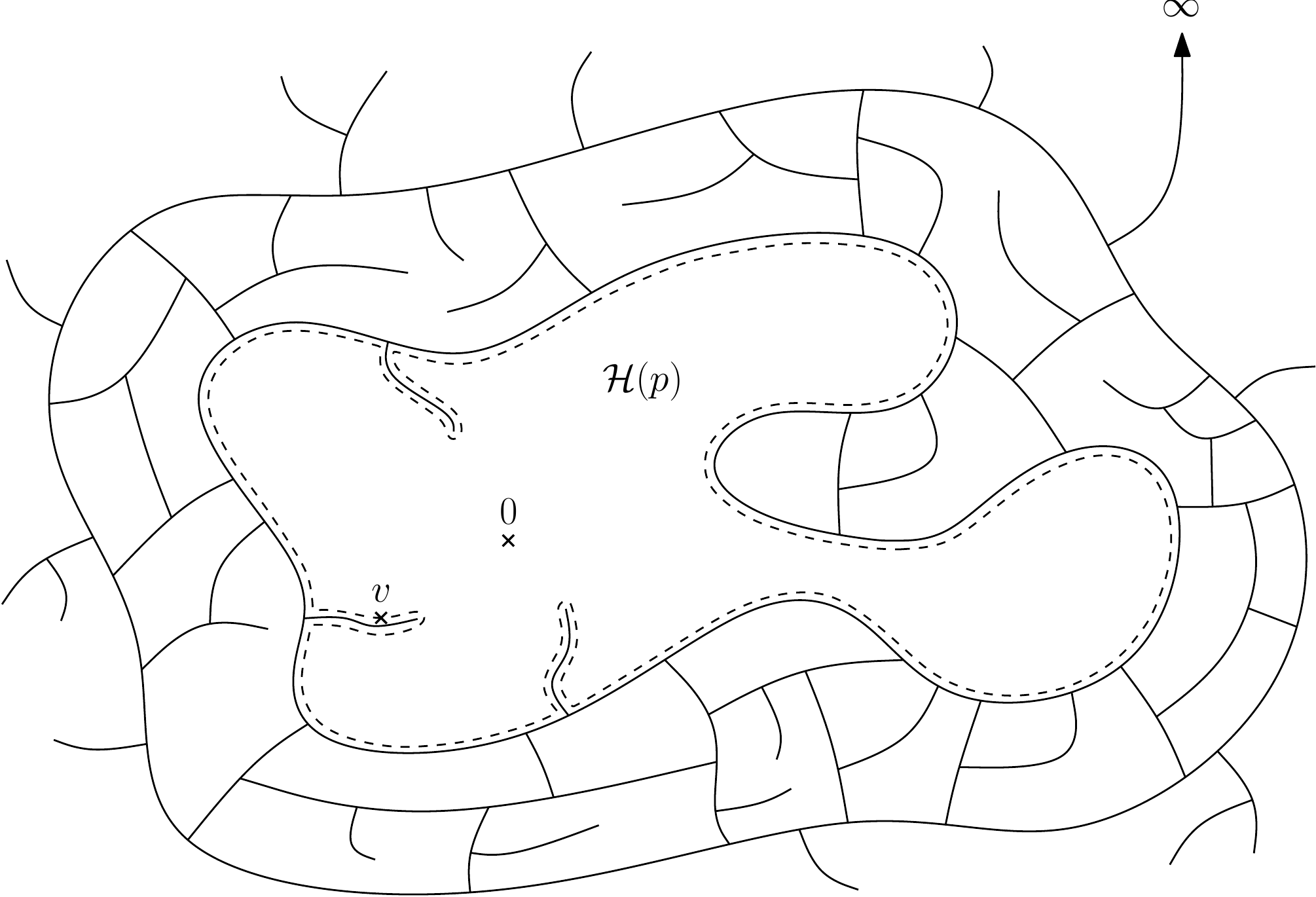}
\caption{\label{fig:hole_cinf} The hole of the origin $\hole(p)$ at time $p$. Its outer boundary $\dout \hole(p)$ is a circuit of white sites (in dashed lines), and each vertex $v$ on it has a neighbor in $\Cinf(p)$, i.e. is connected to $\infty$ by a black path (in solid lines).}
\end{center}
\end{figure}

\begin{remark} \label{rem:three_arms}
For future use, let us observe that if $v \in \dout \hole(p)$ and $\hole(p)$ has diameter $l$, then we can find three disjoint arms from neighbors of $v$ to $\partial \Ball_{l/2}(v)$: two white arms, by following $\dout \hole(p)$ in two directions, and one black arm, since by definition, each vertex of $\dout \hole(p)$ has a neighbor in $\Cinf(p)$ (we refer the reader to Figure \ref{fig:hole_cinf} for an illustration). Also, note that $\hole(p)$ is simply connected.
\end{remark}

We start with some easy a-priori estimates on $\hole(p)$.

\begin{lemma} \label{lem:apriori_hole}
The following lower bounds hold for all $p > p_c$, where $\alpha, c > 0$ are some universal constants.
\begin{itemize}
\item[(i)] For all $\lambda \geq 1$,
\begin{equation} \label{eq:apriori_hole1}
\PP \left( \frac{\hole(p)}{L(p)} \subseteq \Ball_\lambda \right) \geq 1 - e^{-\alpha \lambda}.
\end{equation}

\item[(ii)] For all $\lambda \leq 1$,
\begin{equation} \label{eq:apriori_hole2}
\PP \left( \frac{\hole(p)}{L(p)} \supseteq \Ball_\lambda \right) \geq 1 - c \pi_1 \left( \lambda L(p), L(p) \right).
\end{equation}
\end{itemize}
\end{lemma}

\begin{proof}[Proof of Lemma \ref{lem:apriori_hole}]
(i) We have
$$\PP \left( \frac{\hole(p)}{L(p)} \subseteq \Ball_\lambda \right) \geq \PP_p  \bigg( \circuitevent \bigg( \frac{\lambda}{2} L(p), \lambda L(p) \bigg) \cap \arm_1 \bigg( \frac{\lambda}{2} L(p), \infty \bigg) \bigg) \geq c' \PP_p  \bigg( \arm_1 \bigg( \frac{\lambda}{2} L(p), \infty \bigg) \bigg)$$
for some universal constant $c' > 0$ (using the FKG inequality, and then RSW \eqref{eq:RSW}). The desired lower bound then follows from \eqref{eq:connection_expdecay}.

(ii) We can write
\begin{align*}
\PP \left( \frac{\hole(p)}{L(p)} \supseteq \Ball_\lambda \right) \geq \PP_p  \big( \circuitevent^* ( \lambda L(p), 2 L(p) ) \big) = 1 - \PP_p  \big( \arm_1 ( \lambda L(p), 2 L(p) ) \big).
\end{align*}
Finally, we use that $\PP_p  \big( \arm_1 ( \lambda L(p), 2 L(p) ) \big) \asymp \pi_1(\lambda L(p), L(p))$ (from \eqref{eq:Kesten1} and \eqref{eq:ext}).
\end{proof}

Based on the scaling limit of near-critical percolation, it is natural to expect that $\frac{\hole(p)}{L(p)}$ converges in distribution as $p \searrow p_c$ (in a suitable topology), and so its volume $\frac{|\hole(p)|}{L(p)^2}$ as well. However, the precise knowledge of the scaling limit is not needed for the proofs of Theorems \ref{thm:full_plane} and \ref{thm:large_k}: we only need to know that $\frac{|\hole(p)|}{L(p)^2}$ does not fluctuate too much as $p \searrow p_c$. We use the following estimates, which are weaker and can be proved using classical results.

\begin{lemma} \label{lem:bounds_hole}
There exist universal constants $\alpha_1, \alpha_2, c_1 > 0$ such that the following bounds hold for all $p > p_c$.
\begin{itemize}
\item[(i)] For all $\lambda \geq 1$,
\begin{align}
  e^{-\alpha_1 \sqrt{\lambda}} & \leq \PP \left( \frac{|\hole(p)|}{L(p)^2} \geq \lambda \right) \leq e^{-\alpha_2 \sqrt{\lambda}},  \label{eq:bounds_hole1}\\
\text{and} \quad e^{-\alpha_1 \lambda} & \leq \PP \left( \frac{\hole(p)}{L(p)} \supseteq \Ball_\lambda \right) \leq e^{-\alpha_2 \lambda}.\label{eq:bounds_hole_diam1}
\end{align}

\item[(ii)] For all $\lambda \leq 1$,
\begin{align} 
  c_1 \pi_1 \left( \sqrt{\lambda} L(p), 2 L(p) \right) & \leq \PP \left( \frac{|\hole(p)|}{L(p)^2} \leq \lambda \right) \leq \pi_1 \left( \sqrt{\lambda} L(p), 2 L(p) \right), \label{eq:bounds_hole2}\\
\text{and} \quad c_1 \pi_1 \left( \lambda L(p), 2 L(p) \right) &\leq \PP \left( \frac{\hole(p)}{L(p)} \subseteq \Ball_\lambda \right) \leq \pi_1 \left( \lambda L(p), 2 L(p) \right). \label{eq:bounds_hole_diam2}
 \end{align}
\end{itemize}
\end{lemma}

\begin{proposition} \label{prop:vol_hole}
There exists $\beta > 1$ such that: for all $\bar{\lambda}>1$, for all $p, p' > p_c$ sufficiently close to $p_c$ (depending on $\bar{\lambda}$), and all $\lambda \in [\bar{\lambda}^{-1},\bar{\lambda}]$, one has
\begin{equation} \label{eq:vol_hole}
  \PP \left( \frac{|\hole(p')|}{L(p')^2} \geq \lambda \right) \geq \PP \left( \frac{|\hole(p)|}{L(p)^2} \geq \beta \lambda \right).
\end{equation}
\end{proposition}

We first prove Lemma \ref{lem:bounds_hole}.
\begin{proof}[Proof of Lemma \ref{lem:bounds_hole}] We only prove \eqref{eq:bounds_hole1} and \eqref{eq:bounds_hole2}, since \eqref{eq:bounds_hole_diam1} and \eqref{eq:bounds_hole_diam2} follow in similar ways.

(i) Let us consider $\lambda \geq 1$. For the lower bound in \eqref{eq:bounds_hole1}, we note that
$$\PP \big( |\hole(p)| \geq \PP_p \bigg( \circuitevent^* \bigg( \frac{\sqrt{\lambda}}{2} L(p), \sqrt{\lambda} L(p) \bigg) \bigg),$$
which is at least $e^{-\alpha_1 \sqrt{\lambda}}$, from the lower bound in \eqref{eq:exp_decay} (and the FKG inequality).

Let us now turn to the upper bound in \eqref{eq:bounds_hole1}. If $\hole(p)$ has a volume at least $\lambda L(p)^2$, then $\dout \hole(p)$, which is a white circuit surrounding $0$, must contain one site at a distance at least $\frac{\sqrt{\lambda}}{4} L(p)$ from $0$. This implies
$$\PP \big( |\hole(p)| \geq \lambda L(p)^2 \big) \leq \PP_p \bigg( \arm^*_1 \bigg( \frac{\sqrt{\lambda}}{8} L(p), \frac{\sqrt{\lambda}}{4} L(p) \bigg) \bigg) + \PP_p \bigg( \circuitevent^* \bigg( \frac{\sqrt{\lambda}}{8} L(p),\infty \bigg) \bigg),$$
which is at most $e^{-\alpha_1 \sqrt{\lambda}}$, using once again \eqref{eq:exp_decay}.

(ii) We now consider $\lambda \leq 1$. The upper bound in \eqref{eq:bounds_hole2} follows from the observation that if $|\hole(p)| \leq \lambda L(p)^2$, then the infinite black cluster intersects $\Ball_{\sqrt{\lambda} L(p)}$, so
$$\PP \big( |\hole(p)| \leq \lambda L(p)^2 \big) \leq \PP_p \big( \arm_1 \big( \sqrt{\lambda} L(p), 2 L(p) \big) \big).$$
For the lower bound in \eqref{eq:bounds_hole2}, we note that
\begin{align*}
\PP \big( |\hole & (p)| \leq \lambda L(p)^2 \big)\\
& \geq \PP_p \bigg( \circuitevent \bigg( \frac{\sqrt{\lambda}}{4} L(p), \frac{\sqrt{\lambda}}{2} L(p) \bigg) \cap \arm_1 \bigg( \frac{\sqrt{\lambda}}{4} L(p), 2 L(p) \bigg) \cap \circuitevent \big( L(p), 2 L(p) \big) \cap \arm_1 \big( L(p), \infty \big) \bigg)\\
& \geq \tilde{c}_1 \PP_p \bigg( \arm_1 \bigg( \frac{\sqrt{\lambda}}{4} L(p), 2 L(p) \bigg) \bigg) \PP_p \big( \arm_1 \big( L(p), \infty \big) \big)
\end{align*}
for some universal constant $\tilde{c}_1 > 0$ (using the FKG inequality, and then RSW \eqref{eq:RSW}). We know from \eqref{eq:connection_expdecay} that there exists $\tilde{c}_2 > 0$ such that $\PP_p \left( \arm_1 \left( L(p), \infty \right) \right) \geq \tilde{c}_2$, which allows us to conclude (with \eqref{eq:ext}).
\end{proof}

These bounds can now be used to prove Proposition \ref{prop:vol_hole}.

\begin{proof}[Proof of Proposition \ref{prop:vol_hole}]
We first prove the following claim: there exist $\tilde{\beta} > 1$, and $0 < \bar{\lambda}_1 < \bar{\lambda}_2$, such that for all $\lambda > 0$ with $\lambda \notin (\bar{\lambda}_1,\bar{\lambda}_2)$, for all $p, p' > p_c$ sufficiently close to $p_c$ (depending on $\lambda$),
\begin{equation}
\PP \left( \frac{|\hole(p')|}{L(p')^2} \geq \lambda \right) \geq \PP \left( \frac{|\hole(p)|}{L(p)^2} \geq \tilde{\beta} \lambda \right).
\end{equation}
For $\lambda \geq 1$, applying successively the two bounds of \eqref{eq:bounds_hole1} yields
\begin{equation}
\PP \left( \frac{|\hole(p')|}{L(p')^2} \geq \lambda \right) \geq e^{-\alpha_1 \sqrt{\lambda}} = e^{-\alpha_2 \frac{\alpha_1}{\alpha_2} \sqrt{\lambda}} \geq \PP\left( \frac{|\hole(p)|}{L(p)^2} \geq \frac{\alpha_1^2}{\alpha_2^2} \lambda \right),
\end{equation}
which proves the claim for $\lambda \geq \bar{\lambda}_2$, with $\bar{\lambda}_2 = 1$ and $\tilde{\beta} = \frac{\alpha_1^2}{\alpha_2^2}$. Let us now turn to small values of $\lambda$. Using the lower bound provided by \eqref{eq:bounds_hole2}, we get: for $x \leq 1$,
\begin{equation} \label{eq:hole_small_lambda}
\PP \left( \frac{|\hole(p)|}{L(p)^2} < x \right) \geq c_1 \pi_1 \left( \sqrt{\frac{x}{2}} L(p), 2 L(p) \right) \geq c_1 \tilde{c}_1 \pi_1 \left( \sqrt{\frac{x}{2}} L(p'), 2 L(p') \right).
\end{equation}
Here, we used \eqref{eq:qmult} and the ratio-limit theorem (Lemma \ref{lem:ratio_lim}): we need $L(p)$ and $L(p')$ to be sufficiently large, i.e. $p$ and $p'$ to be close enough to $p_c$, depending on $x$. Using \eqref{eq:Kesten1}, the right-hand side of \eqref{eq:hole_small_lambda} is at least $c_1 \tilde{c}_1 \tilde{c}_2 \PP_{p'} \left( \arm_1 \left( \sqrt{\frac{x}{2}} L(p'), 2 L(p') \right) \right)$. Moreover, for $\ve > 0$ small enough,
$$c_1 \tilde{c}_1 \tilde{c}_2 \PP_{p'} \left( \arm_1 \left( \sqrt{\frac{x}{2}} L(p'), 2 L(p') \right) \right) \geq \PP_{p'} \left( \arm_1 \left( \sqrt{\ve x} L(p'), 2 L(p') \right) \right) \geq \PP \left( \frac{|\hole(p')|}{L(p')^2} < \ve x \right),$$
where the first inequality uses \eqref{eq:qmult}, \eqref{eq:Kesten1} and \eqref{eq:1arm_bound}, and the second one uses the upper bound in \eqref{eq:bounds_hole2}. This gives the claim for $\lambda \leq \bar{\lambda}_1 = \ve$, with $\tilde{\beta} = \frac{1}{\ve}$.

\bigskip

We now use the claim to prove the proposition itself. Let us define
$$\lambda_i = \left( \frac{\bar{\lambda}_2}{\bar{\lambda}_1} \right)^{i-1} \bar{\lambda}_1, \quad i \in \ZZ$$
(so that $\lambda_i = \bar{\lambda}_i$ for $i =1, 2$). Noting that for every $i \in \ZZ$, $\lambda_i \notin (\bar{\lambda}_1,\bar{\lambda}_2)$, we deduce from the claim that for all $p, p' > p_c$ close enough to $p_c$ (depending on $i$),
\begin{equation} \label{eq:pf_vol_hole}
\PP \left( \frac{|\hole(p')|}{L(p')^2} \geq \lambda_i \right) \geq \PP \left( \frac{|\hole(p)|}{L(p)^2} \geq \tilde{\beta} \lambda_i \right).
\end{equation}
Moreover, the same conclusion holds for all $p, p' > p_c$ close enough to $p_c$ (depending on $\bar{\lambda}$) and all $i \in \ZZ$ such that $\lambda_i \in \left[ \bar{\lambda}^{-1}, \frac{\bar{\lambda}_2}{\bar{\lambda}_1} \bar{\lambda} \right]$ simultaneously. Indeed, there are only finitely many such values of $i$, since $\lambda_i \to 0$ and $+ \infty$ as $i \to + \infty$ and $- \infty$, respectively. Now, let us consider $\lambda \in [\bar{\lambda}^{-1},\bar{\lambda}]$: there exists $i \in \ZZ$ such that $\lambda \in [\lambda_i, \lambda_{i+1}]$, and for all $p, p' > p_c$ close enough to $p_c$ (depending on $\bar{\lambda}$),
$$\PP \left( \frac{|\hole(p')|}{L(p')^2} \geq \lambda \right) \geq \PP \left( \frac{|\hole(p')|}{L(p')^2} \geq \lambda_{i+1} \right) \geq \PP \left( \frac{|\hole(p)|}{L(p)^2} \geq \tilde{\beta} \lambda_{i+1} \right),$$
using $\lambda \leq \lambda_{i+1}$, and then \eqref{eq:pf_vol_hole} (note that $\lambda_{i+1} \in \left[ \bar{\lambda}^{-1}, \frac{\bar{\lambda}_2}{\bar{\lambda}_1} \bar{\lambda} \right]$). We can now write
$$\PP \left( \frac{|\hole(p)|}{L(p)^2} \geq \tilde{\beta} \lambda_{i+1} \right) = \PP \left( \frac{|\hole(p)|}{L(p)^2} \geq \tilde{\beta} \frac{\lambda_{i+1}}{\lambda} \lambda \right) \geq \PP \left( \frac{|\hole(p)|}{L(p)^2} \geq \tilde{\beta} \frac{\bar{\lambda}_2}{\bar{\lambda}_1} \lambda \right),$$
which completes the proof of Proposition \ref{prop:vol_hole}, with $\beta = \tilde{\beta} \frac{\bar{\lambda}_2}{\bar{\lambda}_1}$.
\end{proof}

\subsection{Approximable domains}

In this section, we introduce a regularity property for domains, which plays an important role when studying successive (nested) frozen clusters, allowing one to describe the frozen percolation process in an iterative manner. Roughly speaking, this property says that the domain can be approximated by a union of small squares, and we first prove in Section \ref{sec:approx_hole} that it is satisfied by percolation holes. We then use it to establish a continuity property for $|\hole(p)|$ (Section \ref{sec:approx_cont}). We explain later, in Section \ref{sec:approx_vol}, that it helps to predict the volume of the largest connected component.

Let us now give a formal definition. First, we need to introduce some notation. For every $l > 0$, we consider a partition of the plane into squares of side length $l$, such that $0$ is the center of one of these squares:
$$\CC = \bigsqcup_{k_1,k_2 \in \ZZ} \big( (k_1 l, k_2 l) + b_l \big),$$
with $b_l = \big[ - \frac{l}{2}, \frac{l}{2} \big)^2$. These squares are called \emph{$l$-blocks}. Each $l$-block has four neighbors, and this notion of adjacency gives rise to connected components of $l$-blocks.

\begin{figure}
\begin{center}
\includegraphics[width=12cm]{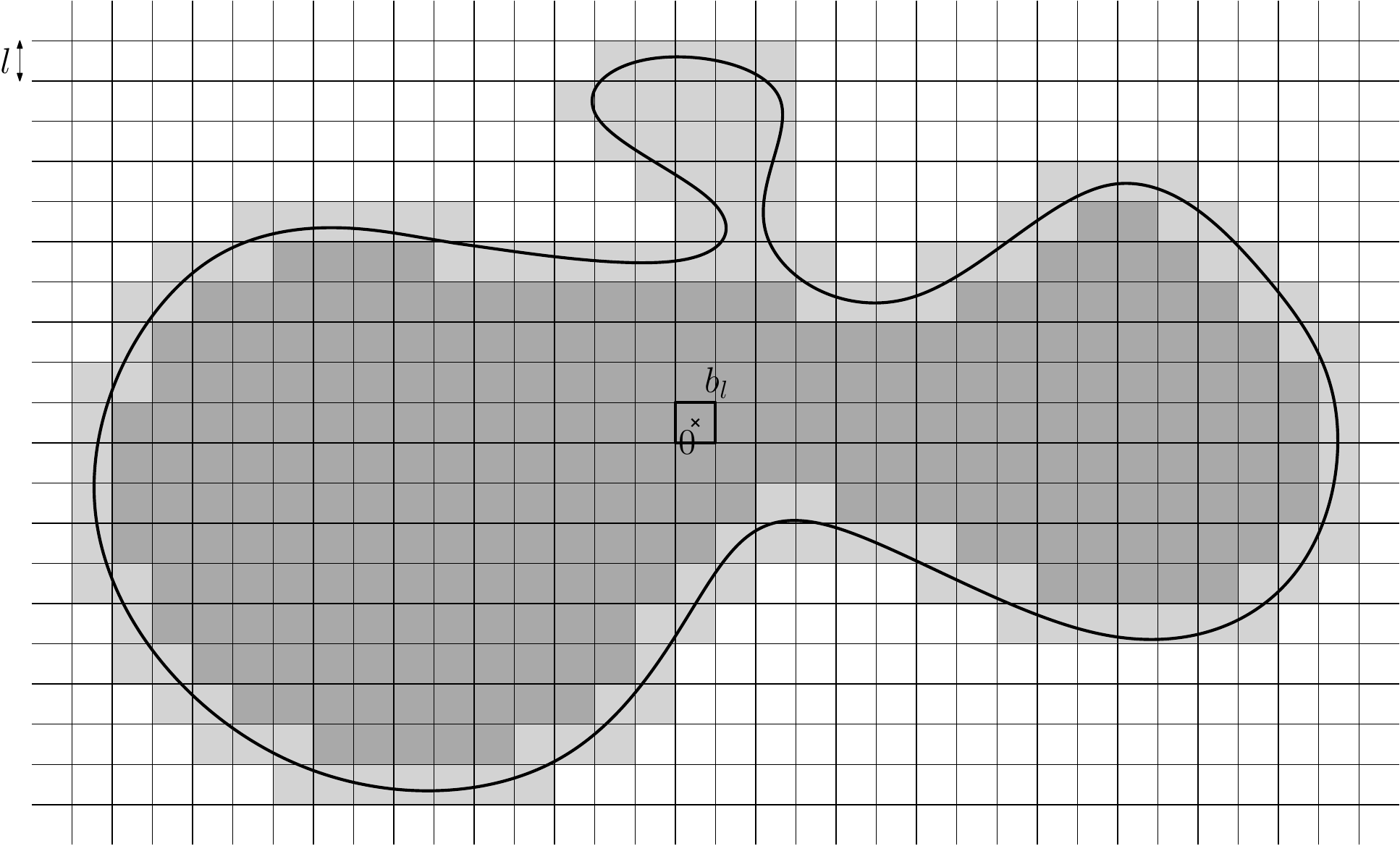}
\caption{\label{fig:approx} The inner and outer approximations of a domain $\Lambda$: $\Dint{\Lambda}{l}$ and $\Dext{\Lambda}{l} \setminus \Dint{\Lambda}{l}$ are in dark and light gray, respectively.}
\end{center}
\end{figure}

For a connected domain $\Lambda \subseteq \CC$, we introduce the following inner and outer approximations by $l$-blocks (see Figure \ref{fig:approx}).
\begin{itemize}
\item We consider the collection of $l$-blocks which are entirely contained in $\Lambda$. These $l$-blocks can be grouped into connected components, and we denote by $\Dint{\Lambda}{l}$ the union of all $l$-blocks in the connected component of $b_l$. By convention, we take $\Dint{\Lambda}{l} = \emptyset$ if $b_l$ is not contained in $\Lambda$.

\item We denote by $\Dext{\Lambda}{l}$ the union of all $l$-blocks that intersect $\Lambda$.
\end{itemize}

\begin{definition}
Let $\Lambda \subseteq \CC$ be a bounded and simply connected domain. For $l > 0$ and $\eta \in (0,1)$, we say that $\Lambda$ is \emph{$(l,\eta)$-approximable} if
\begin{itemize}
\item[(i)] $b_l \subseteq \Lambda$,

\item[(ii)] and $\big| \Dext{\Lambda}{l} \setminus \Dint{\Lambda}{l} \big| < \eta | \Lambda |$.
\end{itemize}
\end{definition}

Clearly, $\Dint{\Lambda}{l} \subseteq \Lambda \subseteq \Dext{\Lambda}{l}$, so this property implies in particular that
$$|\Dint{\Lambda}{l}| > (1 - \eta) |\Lambda| \quad \text{and} \quad |\Dext{\Lambda}{l}| < (1 + \eta) |\Lambda|.$$

We also define the \emph{$t$-shrinking} of a domain $\Lambda$ (for $t > 0$) as
\begin{equation}
\big \{ z : d(z, \CC \setminus \Lambda) \geq t \big \},
\end{equation}
where $d$ is the distance induced by the $\infty$ norm on $\CC$. In other words, it is the complement of the $t$-neighborhood (for the distance $d$) of $\CC \setminus \Lambda$. This notion is used in the particular case when $\Lambda$ is a union of $l$-blocks, as depicted on Figure \ref{fig:shrinking}. In such a situation, for $\ve \in (0,1)$, we denote by $\Lambda^l_{(\ve)}$ the $(\ve l)$-shrinking of $\Lambda$. The value of $l$ is most often clear from the context, in which case we drop the superscript for notational convenience. For future use, let us note that for all $\ve \in (0, \frac{1}{4})$ (and uniformly in $l$),
\begin{equation} \label{eq:shrinking_vol}
(1 - 4\ve) |\Lambda| \leq |\Lambda_{(\ve)}^l| \leq |\Lambda|.
\end{equation}

\begin{figure}[t]
\begin{center}
\includegraphics[width=12cm]{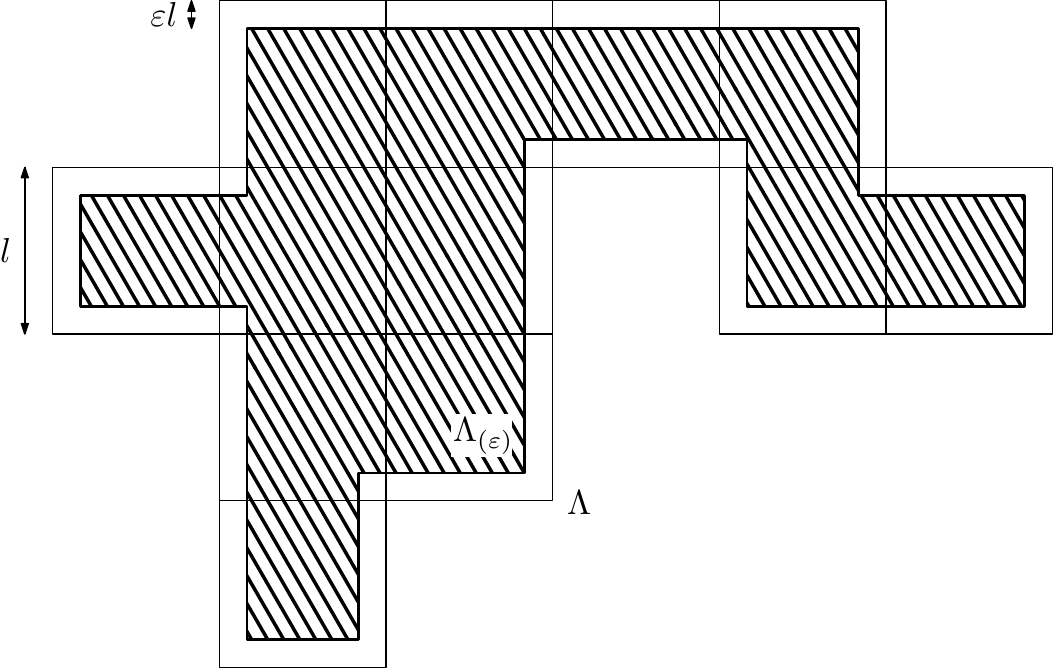}
\caption{\label{fig:shrinking} This figure depicts a union of $l$-blocks $\Lambda$, and its $(\ve l)$-shrinking $\Lambda^l_{(\ve)} = \Lambda_{(\ve)}$.}
\end{center}
\end{figure}

\subsection{Approximability of $\hole(p)$} \label{sec:approx_hole}

We now prove that for $\alpha$ small enough, with high probability, $\hole(p)$ can be approximated by using squares of side length $\alpha L(p)$.

\begin{lemma} \label{lem:approx_H}
For all $\ve, \eta > 0$, there exist $\delta = \delta(\ve) > 0$ and $\alpha = \alpha(\ve, \eta) > 0$ such that: for all $p \in (p_c, p_c + \delta)$,
\begin{equation}
\PP \big( \text{$\hole(p)$ is $(\alpha L(p), \eta)$-approximable} \big) > 1 - \ve.
\end{equation}
\end{lemma}

\begin{proof}[Proof of Lemma \ref{lem:approx_H}]
Let us fix $\ve > 0$. First, we can deduce from Lemma \ref{lem:apriori_hole} the existence of $\tilde{\delta}>0$ and $0 < c_1 < c_2$ such that: for all $p \in (p_c, p_c + \tilde{\delta})$,
\begin{equation} \label{eq:apriori_H}
\PP \left( \Ball_{c_1 L(p)} \subseteq \hole(p) \subseteq \Ball_{c_2 L(p)} \right) \geq 1 - \frac{\ve}{4}.
\end{equation}
In what follows, we consider $p \in (p_c, p_c + \tilde{\delta})$, and we assume that the event in \eqref{eq:apriori_H}, that we denote by $A$, holds. We also take some small $\alpha > 0$, explaining later how to choose it appropriately. We want to derive upper bounds on the volume of $\Dext{\hole(p)}{\alpha L(p)} \setminus \Dint{\hole(p)}{\alpha L(p)}$, which is a union of $(\alpha L(p))$-blocks. By definition of the inner and outer approximations, it can be decomposed as
$$\Dext{\hole(p)}{\alpha L(p)} \setminus \Dint{\hole(p)}{\alpha L(p)} = \Lambda_1 \cup \Lambda_2,$$
where $\Lambda_1$ is the union of blocks that intersect $\dout \hole(p)$, and $\Lambda_2$ is the union of blocks which are entirely contained in $\hole(p)$, but not connected to $b_{\alpha L(p)}$ inside $\hole(p)$.

To handle $\Lambda_1$, note that each block $b = z + b_{\alpha L(p)}$ in $\Lambda_1$ is connected to $z + \partial \Ball_{c_1 L(p)}$ by three arms: two white arms and one black arm (see Remark \ref{rem:three_arms}, and use that $\Ball_{c_1 L(p)} \subseteq \hole(p)$). Denoting $\sigma = (wwb)$, this has a probability
$$\PP_p \left( \arm_{\sigma} \left( \alpha L(p), c_1 L(p) \right) \right) \leq \left( \frac{\alpha L(p)}{c_1 L(p)} \right)^{\mu} = \left( \frac{\alpha}{c_1} \right)^{\mu},$$
for some $\mu > 0$ (using \eqref{eq:gps}). Using that on the event $A$, there are at most $\Big( \frac{2 c_2 L(p)}{\alpha L(p)} \Big)^2$ possible choices for a block in $\Lambda_1$, we deduce
\begin{equation} \label{eq:upper_Lambda1}
\EE_p \big[ |\Lambda_1| \ind_A \big] \leq c_3 \left( \frac{2 c_2 L(p)}{\alpha L(p)} \right)^2 \left( \frac{\alpha}{c_1} \right)^{\mu} (\alpha L(p))^2 = c_4 \alpha^{\mu} L(p)^2.
\end{equation}

As to $\Lambda_2$, it follows from a max-flow min-cut argument, and the fact that $\hole(p)$ is simply connected, that for every $v \in \Lambda_2$, there is a (not necessarily unique) $(\alpha L(p))$-block $B$ such that every path in $\hole(p)$ from $v$ to $0$ intersects $\tilde{B} := \Ball_{\alpha L(p)}(z_B)$, where $z_B$ denotes the center of $B$. In other words, $v$ is separated from $0$ (and so, on the event $A$, from the whole set $\Ball_{c_1 L(p)}$) by the box $\tilde{B}$. With some abuse of terminology, we say that ``$B$ separates $v$ from $0$''.

\begin{figure}
\begin{center}
\includegraphics[width=13cm]{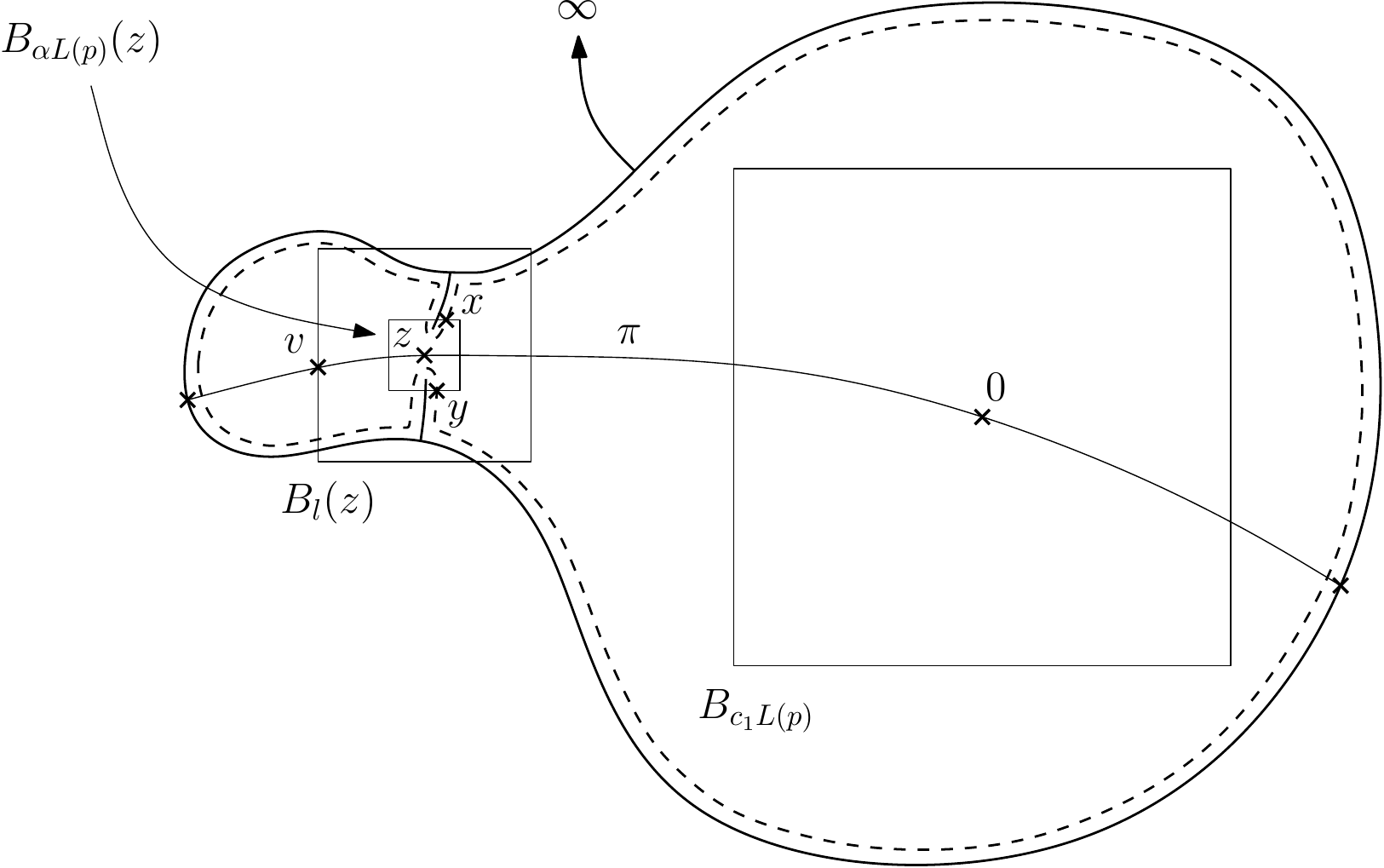}
\caption{\label{fig:bubble} In order to obtain an upper bound on $\EE_p \big[ |\Lambda_2| \ind_A \big]$, we use that around a well-chosen $(\alpha L(p))$-block, there exist locally six arms: four white arms, and two black ones.}
\end{center}
\end{figure}

We now point out some geometric consequences of the fact that $B$ separates $v$ from $0$. Let us denote by $l$ the distance from $z$ to $v$. First, it is clear that $\tilde{B}$ must intersect $\dout \hole(p)$. Also, we can construct a path as follows, starting with a path $\pi_0$ from $v$ to $0$ in $\hole(p)$. We know that if we remove $\tilde{B}$, $v$ and $0$ end up in different connected components of $\hole(p)$. In the component of $0$, we extend $\pi_0$ from $0$ to some point which is on the boundary of that component at a distance at least $c_1 L(p)$ from $z$, and in the component of $v$, we extend $\pi_0$ from $v$ to some point on the boundary of that component  at a distance at least $l$ from $z$. Let $\pi$ denote the resulting path (see Figure \ref{fig:bubble}). We can then fix two vertices $x$ and $y$ in $\partial \tilde{B} \cap \dout \hole(p)$ which are on different sides of $\pi$.

By Remark \ref{rem:three_arms}, there exist three arms (with colors $\sigma = (wwb)$) starting from neighbors of $x$, and reaching distance $\min(l, c_1 L(p))$ from $z$. Note that these arms lie on the same side of $\pi$. Similarly, there are also three such arms from neighbors of $y$, and these arms lie on the opposite side of $\pi$ (compared to the previous arms): we have thus six arms (with colors $\tilde{\sigma} = (wwbwwb)$) in the annulus $\Ann_{\alpha L(p), \min(l, c_1 L(p))}(z)$. Using again Remark \ref{rem:three_arms}, we also know that there exist three arms (with colors $\sigma$) from neighbors of $x$ to distance $\max(l, c_1 L(p))$ from $z$, and the last part of these arms, from distance $\min(l, c_1 L(p))$ on, is clearly disjoint from the six arms mentioned earlier. Hence, we get that for given vertex $v$ and $(\alpha L(p))$-block $B$,
\begin{align}
\PP(A, \: v \in \hole(p), & \: B \text{ separates } v \text{ from } 0) \nonumber\\
& \leq \PP_p \left( \arm_{\tilde{\sigma}} \left( \alpha L(p), \min(l, c_1 L(p)) \right) \right) \PP_p \left( \arm_{\sigma} \left( \min(l, c_1 L(p)), \max(l, c_1 L(p)) \right) \right), \label{eq:six_three_arms}
\end{align}
where $l = d(v, z_B)$.

We then use \eqref{eq:six_three_arms} to derive an upper bound on $\EE_p \big[ |\Lambda_2| \ind_A \big]$. It is convenient to first rule out the case when $v$ and $B$ are far apart: we introduce $\tilde{A} := \{$there exist an $(\alpha L(p))$-block $B$, and a vertex $v$ with $d(v, z_B) \geq c_1 L(p)$, such that $B$ separates $v$ from $0\}^c$. By \eqref{eq:six_three_arms}, we have
$$\PP(A \cap \tilde{A}^c) \leq c_5 \left( \frac{2 c_2 L(p)}{\alpha L(p)} \right)^2 \PP_p \left( \arm_{\tilde{\sigma}} \left( \alpha L(p), c_1 L(p) \right) \right),$$
where $\Big( \frac{2 c_2 L(p)}{\alpha L(p)} \Big)^2$ counts the possible choices for $B$. Using \eqref{eq:6arm_bound}, we obtain
$$\PP(A \cap \tilde{A}^c) \leq c_6 \left( \frac{c_2}{\alpha} \right)^2  \left( \frac{\alpha}{c_1} \right)^{2+\delta} = c_7 \alpha^{\delta},$$
which can be made $\leq \frac{\ve}{4}$ by taking $\alpha$ sufficiently small.

Now, we can write
$$\EE_p \big[ |\Lambda_2| \ind_{A \cap \tilde{A}} \big] \leq \sum_B \sum_{\substack{v \text{ s.t.}\\ d(v, z_B) \leq c_1 L(p)}} \PP_p(B \text{ separates } v \text{ from } 0),$$
where the first sum is over all $(\alpha L(p))$-blocks $B \subseteq \Ball_{c_2 L(p)}$. Using again \eqref{eq:six_three_arms}, we obtain
\begin{align*}
\EE_p \big[ |\Lambda_2| \ind_{A \cap \tilde{A}} \big] & \leq c_8 \left( \frac{c_2}{\alpha} \right)^2 \sum_{r=0}^{c_1/\alpha} \sum_{\substack{v \text{ s.t. } d(v, z_B) \text{ in}\\ [r \alpha L(p), (r+1) \alpha L(p))}} \PP_p(B \text{ separates } v \text{ from } 0)\\
& \leq c_9 \left( \frac{c_2}{\alpha} \right)^2 \left( \alpha^2 L(p)^2 \left( \frac{\alpha}{c_1} \right)^{\mu} + \sum_{r=1}^{c_1/\alpha} r \alpha^2 L(p)^2 \left( \frac{\alpha}{r \alpha} \right)^{2 + \delta} \left( \frac{r \alpha}{c_1} \right)^{\mu} \right)\\
& \leq c_{10} \alpha^{\mu} L(p)^2 \left( 1 + \sum_{r=1}^{\infty} r^{-1-\delta+\mu} \right).
\end{align*}
Since we may choose $\mu < \delta$ in the beginning, without loss of generality, we obtain
\begin{equation} \label{eq:upper_Lambda2}
\EE_p \big[ |\Lambda_2| \ind_{A \cap \tilde{A}} \big] \leq c_{11} \alpha^{\mu} L(p)^2.
\end{equation}

Finally, we can write
\begin{align}
\PP_p \big( |\Dext{\hole(p)}{\alpha L(p)} & \setminus \Dint{\hole(p)}{\alpha L(p)}| \geq \eta |\hole(p)| \big) \nonumber \\
& \leq \PP_p \big( A \cap \tilde{A} \cap \{|\Lambda_1| + |\Lambda_2| \geq \eta |\hole(p)|\} \big) + \PP_p \big( (A \cap \tilde{A})^c \big) \nonumber\\
& \leq \PP_p \big( A \cap \tilde{A} \cap \{|\Lambda_1| + |\Lambda_2| \geq \eta (2 c_1)^2 L(p)^2\} \big) + \frac{\ve}{2}. \label{eq:upper_bd_lambda12}
\end{align}
Applying Markov's inequality to the probability in the right-hand side of \eqref{eq:upper_bd_lambda12} (using \eqref{eq:upper_Lambda1} and \eqref{eq:upper_Lambda2}) shows that this probability is at most $(c_4 + c_{11}) \alpha^{\mu} (4 \eta c_1^2)^{-1}$. By taking $\alpha$ sufficiently small, this last quantity is smaller than $\frac{\ve}{2}$, which completes the proof of Lemma \ref{lem:approx_H}.
\end{proof}

\subsection{Continuity property for $\hole(p)$} \label{sec:approx_cont}

We now establish a continuity property for the volume of $\hole(p)$, based on the approximability property.

\begin{lemma} \label{lem:cont_vol}
For all $\ve > 0$, there exist $\alpha, \delta > 0$ such that: for all $p, p' \in (p_c, p_c + \delta)$ with $p < p'$ and $\frac{L(p)}{L(p')} < 1 + \delta$, one has with probability $> 1 - \ve$,
\begin{itemize}
\item[(i)] $\hole(p)$ is $(\alpha L(p), \ve)$-approximable

\item[(ii)] and $(\Dint{\hole(p)}{\alpha L(p)})_{(\ve)} \subseteq \hole(p') \subseteq \hole(p)$.
\end{itemize}
\end{lemma}

This lemma implies directly the following continuity result for $|\hole(p)|$.

\begin{corollary} \label{cor:cont_vol}
For all $\ve > 0$, there exists $\delta > 0$ such that: for all $p, p' \in (p_c, p_c + \delta)$ with $p < p'$ and $\frac{L(p)}{L(p')} < 1 + \delta$, one has
\begin{equation}
\PP \bigg( \frac{|\hole(p)|}{|\hole(p')|} < 1 + \ve \bigg) > 1 - \ve.
\end{equation}
\end{corollary}

\begin{proof}[Proof of Corollary \ref{cor:cont_vol} from Lemma \ref{lem:cont_vol}]
For any $\ve>0$, let us assume that properties (i) and (ii) from Lemma \ref{lem:cont_vol} are satisfied for some $\alpha, p, p'$. Using successively (ii) and \eqref{eq:shrinking_vol}, we obtain
$$|\hole(p')| \geq \big| (\Dint{\hole(p)}{\alpha L(p)})_{(\ve)} \big| \geq (1 - 4 \ve) \big| \Dint{\hole(p)}{\alpha L(p)} \big|.$$
From the $(\alpha L(p), \ve)$-approximability of $\hole (p)$ (property (i)), we can then deduce
$$|\hole(p')| \geq (1 - 4 \ve) (1 - \ve) |\hole(p)|,$$
which concludes the proof.
\end{proof}

\begin{proof}[Proof of Lemma \ref{lem:cont_vol}]
\textbf{Step 1.} We first show that every vertex on $\dout \hole(p')$ is close to $\dout \hole(p)$. We establish the following claim: for all $\ve, \eta > 0$, there exists $\delta > 0$ such that: for all $p, p' \in (p_c, p_c + \delta)$ with $p < p'$ and $\frac{L(p)}{L(p')} < 1 + \delta$, one has
\begin{equation} \label{eq:cont_H}
\PP \left( \exists v \in \dout \hole(p') \text{ with } d(v, \din \hole(p)) \geq \eta L(p) \right) < \ve.
\end{equation} 
Let us fix $\ve > 0$. It is enough to prove the claim for all $\eta$ smaller than some value (depending on $\ve$). Hence, we may assume (using the a-priori bounds provided by Lemma \ref{lem:apriori_hole}) that $\eta$ satisfies: for all $p$ close enough to $p_c$,
\begin{equation} \label{eq:apriori_H2}
\PP \left( \Ball_{\eta L(p)} \subseteq \hole(p) \subseteq \Ball_{\eta^{-1} L(p)} \right) > 1 - \frac{\ve}{2}.
\end{equation}
Now, suppose there is a vertex $v \in \dout \hole(p')$ with $d(v, \din \hole(p)) \geq \eta L(p)$. It follows from the definition of $\hole(p')$ that there exists an infinite $p'$-black path starting from a neighbor of $v$. Since $\hole(p') \subseteq \hole(p)$, this path intersects the circuit $\dout \hole(p)$, and we call $w$ the first such intersection point, which is thus $p'$-black and $p$-white.

\begin{figure}
\begin{center}
\includegraphics[width=12cm]{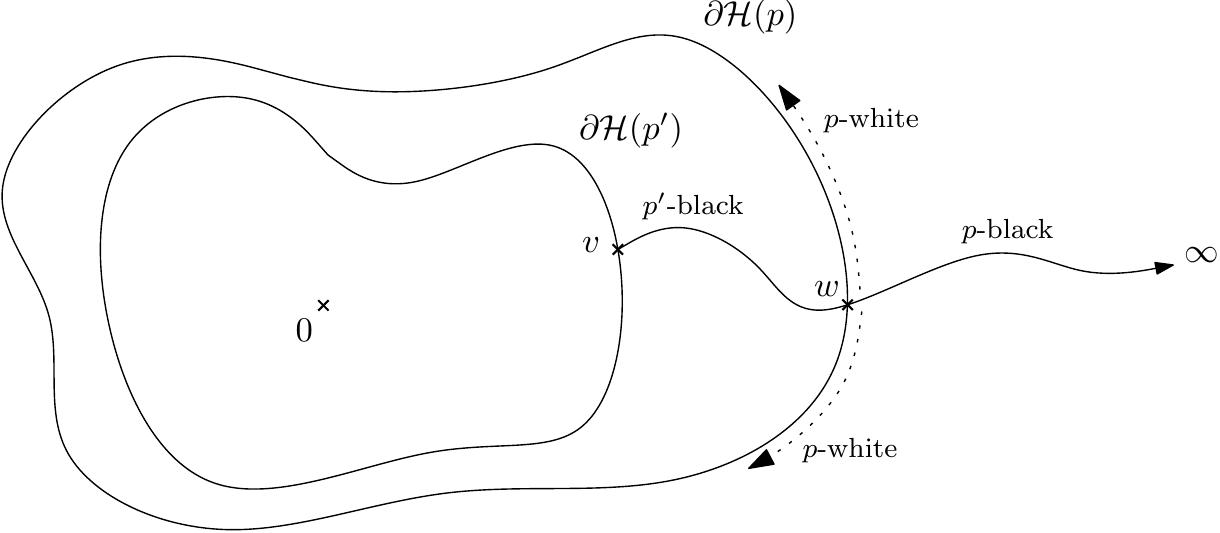}
\caption{\label{fig:cont_hole} The four arm configuration appearing in the proof of Lemma \ref{lem:cont_vol}.}
\end{center}
\end{figure}

If we now assume that the event in \eqref{eq:apriori_H2} holds, we can find four arms starting from neighbors of $w$ (see Figure \ref{fig:cont_hole}) to $w + \partial \Ball_{\eta L(p)}$: two $p'$-black arms (using the infinite path starting from $v$), and two $p$-white ones (following $\dout \hole(p)$ in two directions). For each vertex $w$, these two properties (being $p'$-black and $p$-white, and having four arms to distance $\eta L(p)$) have a probability $\leq c_1 (p'-p) \pi_4(\eta L(p))$ (using \eqref{eq:Kesten1}). Since there are at most $c_2 (\eta^{-1} L(p))^2$ choices for $w$, we deduce (combined with \eqref{eq:apriori_H2}) that
\begin{align*}
\PP \big( \exists v \in \dout & \hole(p') \text{ with } d(v, \din \hole(p)) \geq \eta L(p) \big)\\
& < \frac{\ve}{2} + c_3 \eta^{-2} L(p)^2 (p'-p) \pi_4(\eta L(p))\\
& < \frac{\ve}{2} + c_4 \eta^{-2} (\pi_4(\eta L(p), L(p)))^{-1} \frac{p'-p}{p-p_c} \left[ L(p)^2 (p-p_c) \pi_4(L(p)) \right]
\end{align*}
(using \eqref{eq:qmult}). We know that $\eta^{-2} (\pi_4(\eta L(p), L(p)))^{-1} \leq c_5 \eta^{-4}$ (using \eqref{eq:4arm_bound}), and $L(p)^2 (p-p_c) \pi_4(L(p)) \asymp 1$ (see \eqref{eq:Kesten_L}), so the desired probability is $< \ve$ for $\frac{p'-p}{p-p_c}$ small enough (depending on $\eta$), i.e. $\frac{L(p)}{L(p')}$ sufficiently close to $1$ (using now Lemma \ref{lem:p-L(p)}).

\textbf{Step 2.} We now complete the proof. Let us fix $\ve > 0$. It follows from Step 1 and Lemma \ref{lem:approx_H} that we can find $\alpha, \delta > 0$ small enough such that: for all $p, p' \in (p_c, p_c + \delta)$ with $p < p'$ and $\frac{L(p)}{L(p')} < 1 + \delta$, one has
\begin{itemize}
\item[(i)] $\Ball_{\frac{\alpha}{2} L(p)} \subseteq \hole(p')$,

\item[(ii)] for all $v \in \dout \hole(p')$, $d(v, \din \hole(p)) < \ve \alpha L(p)$,

\item[(iii)] and $\hole(p)$ is $(\alpha L(p), \ve)$-approximable
\end{itemize}
with probability $> 1 - \ve$. It then suffices to observe that
$$(\Dint{\hole(p)}{\alpha L(p)})_{(\ve)} \subseteq \hole(p')$$
follows from properties (i) and (ii). First, $\Dint{\hole(p)}{\alpha L(p)}$ is a connected component of $(\alpha L(p))$-blocks, and it is easy to convince oneself that its $(\ve \alpha L(p))$-shrinking is connected as well (as in the example of Figure \ref{fig:shrinking}), e.g. by induction (we may assume that $\ve < \frac{1}{2}$). We know from (i) that the block $b_{\alpha L(p)}$ is in $\Dint{\hole(p)}{\alpha L(p)}$, and that it is contained in $\hole(p')$, so it is surrounded by the circuit $\dout \hole(p')$. Hence, this circuit either completely surrounds the connected set $(\Dint{\hole(p)}{\alpha L(p)})_{(\ve)}$, in which case the desired conclusion follows, or it intersects $(\Dint{\hole(p)}{\alpha L(p)})_{(\ve)}$. But this second possibility cannot occur, since a vertex $v \in \dout \hole(p') \cap (\Dint{\hole(p)}{\alpha L(p)})_{(\ve)}$ would satisfy
$$d(v,\din \hole(p)) \geq d(v,(\Dint{\hole(p)}{\alpha L(p)})^c) \geq \ve \alpha L(p)$$
(by definition of the shrinking), which contradicts (ii). This completes the proof of Lemma \ref{lem:cont_vol}.
\end{proof}

\section{Volume estimates for independent percolation} \label{sec:volume}

\subsection{Largest clusters in an approximable domain} \label{sec:approx_vol}

In order to study volume-frozen percolation, we need estimates on the volume of the largest black cluster inside a connected subset $\Lambda \subseteq \TT$ (typically, these estimates are used in the case when $\Lambda$ is a hole left by earlier freezing events). More precisely, we can look at the configuration for independent percolation in such a $\Lambda$, at time $p$: among all the black connected components, we denote by $\lclus_{\Lambda}(p)$ the one with largest volume. There may be several such components, but we can just choose one of them according to some deterministic rule.

Several properties of $\lclus_{\Ball_n}(p)$ were established in \cite{BCKS01}, in particular that it has a volume $\approx |\Ball_n| \theta(p)$ if $n \gg L(p)$. We now explain how to extend this property to more general domains. The approximability property turns out to play an important role here.

\begin{lemma} \label{lem:largest_vol}
For all $\ve > 0$ and $C \geq 1$, there exists $\mu > 0$ such that: for all $p > p_c$, all $n$ with $\frac{L(p)}{n} < \mu$, and all sets $\Lambda$ of one of the two types
\begin{itemize}
\item $\Lambda = (\tilde{\Lambda})_{(\beta)}$, where $\beta \in [0,\frac{1}{3}]$ and $\tilde{\Lambda}$ is a connected component of $\leq C$ $n$-blocks containing $b_n$,
\item or $\Lambda$ is an $(n,\frac{\ve}{2})$-approximable set with $\Ball_n \subseteq \Lambda \subseteq \Ball_{C n}$,
\end{itemize}
the following three properties are satisfied, with probability $> 1 - \ve$:
\begin{itemize}
\item[(i)] the largest $p$-black cluster in $\Lambda$, i.e. $\lclus_{\Lambda}(p)$, satisfies
$$\bigg| \frac{|\lclus_{\Lambda}(p)|}{\theta(p) |\Lambda|} - 1 \bigg| < \ve,$$

\item[(ii)] this cluster contains a circuit in $\Ann_{\frac{n}{8}, \frac{n}{4}}$ which is connected to $\infty$ by a $p$-black path,

\item[(iii)] and all other $p$-black clusters in $\Lambda$ have a volume at most $\ve \theta(p) |\Lambda|$.
\end{itemize}
\end{lemma}

Note that property (ii) ensures that the hole around the origin in $\lclus_{\Lambda}(p)$ coincides with $\hole(p)$.

\begin{proof}[Proof of Lemma \ref{lem:largest_vol}]
We can follow essentially the proofs of \cite{BCKS01}. For the convenience of the reader, we explain in Appendix \ref{sec:app_BCKS} which adaptations are needed in our particular setting.
\end{proof}

%
%
%

We now define stopping sets. They play the role of stopping times in our situation, allowing us to study iteratively the frozen percolation process.

\begin{definition} \label{def:stopping_set}
Consider a countable set $\absset$, and a process $X = (X(s))_{s\in \absset}$ indexed by $\absset$. A random subset $\calS \subseteq \absset$ is called a \emph{stopping set} for $X$ if it satisfies the following property:
$$\text{for all } S \subseteq \absset, \quad \{\calS = S\} \in \sigma(X(s) \: : \: s \in \absset \setminus S).$$
\end{definition}

Stopping sets can be seen as a generalization of stopping times. When studying percolation, the following stopping set is often used. If $A$ is an annulus, and $\circuit$ is the outermost black circuit in $A$, then the set $\inter(\circuit)$ of vertices inside it, i.e. in the finite connected component of $\TT \setminus \circuit$, is a stopping set (if such a circuit does not exist, we simply take $\inter(\circuit) = \emptyset$).

For a simply connected domain $\Lambda$, the following stopping set turns out to be very useful in our proofs. We consider the percolation process with parameter $p$ in $\Lambda$, and we look at the set $\cluster_{\Lambda}(p)$ of all black vertices inside $\Lambda$ which are connected to $\din \Lambda$: if we remove this set, together with its outer boundary (which consists of white sites) and $\din \Lambda$, we obtain as a ``hole of the origin'' the set
\begin{equation} \label{eq:def_hole_domain}
\hole^{(\Lambda)}(p) := \text{connected component of $0$ in $\TT \setminus (\din \Lambda \cup \cluster_{\Lambda}(p) \cup \dout \cluster_{\Lambda}(p))$},
\end{equation}
which we take $ = \emptyset$ if $0$ belongs to $\cluster_{\Lambda}(p) \cup \dout \cluster_{\Lambda}(p)$. Note that $\hole^{(\Lambda)}(p)$ is a stopping set if $\Lambda$ is given, or if $\Lambda$ itself is a stopping set. This property of ``explorability from outside'' makes it a useful substitute of $\hole(p)$.

\begin{remark} \label{rem:H_HLambda}
Let us observe that there exists $\alpha > 0$ with the following property: for all $p > p_c$ and $\lambda \geq 1$, for all simply connected domains $\Lambda$,
\begin{equation} \label{eq:H_HLambda}
\Ball_{\lambda L(p)} \subseteq \Lambda \: \Longrightarrow \: \PP \big( \hole^{(\Lambda)}(p) = \hole(p) \big) \geq 1 - e^{-\alpha \lambda}.
\end{equation}
Indeed, we note that $\hole(p) \subseteq \Ball_{\lambda L(p)}$ implies the existence of a $p$-black circuit in $\Ball_{\lambda L(p)}$ which surrounds $0$, and which is connected to $\infty$. In particular, if $\Ball_{\lambda L(p)} \subseteq \Lambda$, this circuit is contained in $\Lambda$, and it is connected to its boundary, so that $\hole^{(\Lambda)}(p)$ and $\hole(p)$ coincide in this case. Finally, Lemma \ref{lem:apriori_hole} (i) implies that for some $\alpha > 0$,
$$\PP \big( \hole(p) \subseteq \Ball_{\lambda L(p)} \big) \geq 1 - e^{-\alpha \lambda}.$$
\end{remark}


We use this remark later to apply Lemma \ref{lem:largest_vol} in holes $\hole(p)$ created by the infinite cluster, which allows us to analyze successive frozen clusters.

\subsection{Tail estimates and moment bounds}

We first mention a tail estimate on $\big| \lclus_{\Lambda} \big|$, which follows easily from a result in \cite{BCKS01}.

\begin{lemma} \label{lem:BCKS}
There exist universal constants $c_1, c_2, X > 0$ such that for all $p > p_c$, $n \geq L(p)$, and $x \geq X$,
\begin{equation} \label{eq:BCKS}
\PP_p \left( \big| \lclus_{\Ball_n} \big| \geq x n^2 \theta(p) \right) \leq c_1 e^{- c_2 x \frac{n^2}{L(p)^2}}.
\end{equation}
\end{lemma}


\begin{proof}[Proof of Lemma \ref{lem:BCKS}]
It follows from Proposition 4.3 (iii) in \cite{BCKS01} that for all $p > p_c$, $n \geq L(p)$ and $x \geq 0$,
$$\PP_p \left( \big| \lclus_{\Ball_n} \big| \geq x n^2 \theta(p) \right) \leq c'_1 \frac{n^2}{L(p)^2} e^{- c'_2 x \frac{n^2}{L(p)^2} + c'_3 \frac{n^2}{L(p)^2}}$$
(in that result, we have $s(L(p)) = (2 L(p))^2 \pi_1(L(p)) \asymp L(p)^2 \theta(p)$, using \eqref{eq:Kesten_theta_pi}). Hence, if we take $X > 0$ so that $- c'_2 X + c'_3 = - \frac{c'_2}{2} X$, we obtain: for all $x \geq X$,
$$\PP_p \left( \big| \lclus_{\Ball_n} \big| \geq x n^2 \theta(p) \right) \leq c'_1 \frac{n^2}{L(p)^2} e^{- \frac{c'_2}{2} x \frac{n^2}{L(p)^2}},$$
which is $\leq c'_4 e^{- \frac{c'_2}{4} x \frac{n^2}{L(p)^2}}$ for some universal constant $c'_4$ large enough.
\end{proof}

%

We now derive moment bounds for the random variables
\begin{equation} \label{eq:def_Vn}
\calV_n(z) := \big| \{v \in \Ball_n(z) \: : \: v \lra{} \partial \Ball_{2n}(z) \} \big|
\end{equation}
(when $z=0$, we simply write $\calV_n$).

\begin{lemma} \label{lem:moment_bd}
There exists a universal constant $C > 0$ such that for all $p > p_c$ and $n \geq L(p)$,
\begin{equation} \label{eq:moment_bd}
\text{for all } m \geq 1, \quad \EE_p \big[ ( \calV_n )^m \big] \leq m! \big( C n^2 \theta(p) \big)^m.
\end{equation}
\end{lemma}

\begin{proof}[Proof of Lemma \ref{lem:moment_bd}]
It follows from Lemma 4.2 in \cite{BCKS01} (with $d=2$, and using also that $\pi_1(L(p)) \asymp \theta(p)$, from \eqref{eq:Kesten_theta_pi}) that for all $p > p_c$,
\begin{equation} \label{eq:moment_bd_BCKS}
\text{for all } m \geq 1, \quad \EE_p \big[ ( \calV_{L(p)} )^m \big] \leq m! \big( C' L(p)^2 \theta(p) \big)^m,
\end{equation}
where $C' > 0$ is a universal constant. For $n \geq L(p)$, we can cover $\Ball_n$ with $k = \big\lceil \frac{n}{L(p)} \big\rceil^2 \leq 4 \frac{n^2}{L(p)^2}$ (possibly overlapping) boxes of the form $z_i + \Ball_{L(p)} \subseteq \Ball_n$ ($1 \leq i \leq k$). We clearly have $\calV_n \leq \sum_{i=1}^k \calV_{L(p)}(z_i)$, and Minkowski's inequality implies that for all $m \geq 1$,
$$\EE_p \big[ ( \calV_n )^m \big] \leq \EE_p \Bigg[ \bigg( \sum_{i=1}^k \calV_{L(p)}(z_i) \bigg)^m \Bigg] \leq k^m \EE_p \big[ ( \calV_{L(p)} )^m \big],$$
since each $\calV_{L(p)}(z_i) \stackrel{(d)}{=} \calV_{L(p)}$. Combined with \eqref{eq:moment_bd_BCKS}, this yields the desired result.
\end{proof}

These bounds are used in Section \ref{sec:full_plane}, in combination with the following form of Bernstein's inequality.

\begin{lemma} \label{lem:Bernstein}
Let $(X_i)_{1 \leq i \leq n}$ ($n \geq 1$) be independent real-valued random variables, satisfying: for all $1 \leq i \leq n$,
$$\EE[ X_i ] = 0 \quad \text{and} \quad \EE \big[ |X_i|^m \big] \leq m! M^{m-2} \frac{\sigma_i^2}{2} \quad \text{for all } m \geq 2,$$
for some $M > 0$ and $(\sigma_i)_{1 \leq i \leq n}$. Then for all $y \geq 0$,
\begin{equation}
\PP \bigg( \bigg| \sum_{i=1}^n X_i \bigg| \geq y \bigg) \leq 2 e^{-\frac{1}{2} \frac{y^2}{\sigma^2 + My}}, \quad \text{where } \sigma^2 = \sum_{i=1}^n \sigma_i^2.
\end{equation}
\end{lemma}

\begin{proof}[Proof of Lemma \ref{lem:Bernstein}]
This follows from an application of Markov's inequality to the random variable $e^{\lambda \sum_{i=1}^n X_i}$, for a well-chosen value of the parameter $\lambda > 0$ (here, $\lambda = \frac{y}{\sigma^2 + My}$). We refer the reader to the proof of (7) in \cite{Benett1962} for more details.
\end{proof}

Finally, we state a consequence of Lemma \ref{lem:moment_bd}, which is needed in Section \ref{sec:nice_circuits}.

\begin{lemma} \label{lem:calV_quant}
There exists a constant $c_0 > 0$ satisfying the following property. For all $\ve > 0$, there exists $X \geq 1$ (depending only on $\ve$) such that for all $p > p_c$ and $n \geq X L(p)$,
$$\PP_p \big( \calV_n \geq c_0 n^2 \theta(p) \big) \geq 1 - \ve.$$
\end{lemma}

\begin{proof}[Proof of Lemma \ref{lem:calV_quant}]
Let us consider $n = x L(p)$, with $x \geq 1$. If we denote $E_1 := \net \big(\frac{\sqrt{x} L(p)}{8}, 2n \big)$ (recall Definition \ref{def:net} for nets), it follows from Lemma \ref{lem:net} that
$$\PP_p ( E_1 ) \geq 1 - c_1 \bigg( \frac{2n \cdot 8}{\sqrt{x} L(p)} \bigg)^2 e^{- c_2 \frac{\sqrt{x} L(p)}{8 L(p)} } = 1 - c_3 x e^{-c_4 \sqrt{x}},$$
which is $\geq 1 - \frac{\ve}{2}$ for all $x \geq X_1 = X_1(\ve)$.

Now, let us consider $k = \big\lfloor \frac{n}{3 \sqrt{x} L(p)} \big \rfloor^2 \geq c_1 x$ disjoint boxes of the form $z_i + \Ball_{2 \sqrt{x} L(p)} \subseteq \Ball_n$ ($1 \leq i \leq k$). In each of them, we have (since $\PP_{p}(0 \lra{} \partial \Ball_n) \asymp \theta(p)$ for $n \geq L(p)$, by \eqref{eq:equiv_expdecay})
$$\EE_p [ \calV_{\sqrt{x} L(p)}(z_i) ] \asymp x L(p)^2 \theta(p),$$
and using Lemma \ref{lem:moment_bd} with $m=2$,
$$\EE_p [ (\calV_{\sqrt{x} L(p)}(z_i))^2 ] \leq c_2 (x L(p)^2 \theta(p))^2.$$
We can thus deduce from a second-moment argument that
$$\PP_p ( \calV_{\sqrt{x} L(p)}(z_i) \geq c_3 x L(p)^2 \theta(p)) \geq c_3,$$
for some universal constant $c_3 > 0$ small enough. If we call
$$E_2 := \Big\{ \big| \big\{ 1 \leq i \leq k \: : \: \calV_{\sqrt{x} L(p)}(z_i) \geq c_3 x L(p)^2 \theta(p) \big\} \big| \geq \frac{c_3}{2} k \Big\},$$
then $\PP_p(E_2) \geq 1 - \frac{\ve}{2}$ for all $x \geq X_2 = X_2(\ve)$ (using that the $k \geq c_1 x$ boxes are disjoint). We observe
$$\calV_n \geq \ind_{E_1} \bigg( \sum_{1 \leq i \leq k} \calV_{\sqrt{x} L(p)}(z_i) \bigg),$$
so on the event $E_1 \cap E_2$ (which occurs with probability $\geq 1 - \ve$ if $x \geq \max(X_1,X_2)$), we have
$$\calV_n \geq \bigg(\frac{c_3}{2} c_1 x \bigg) \cdot \big( c_3 x L(p)^2 \theta(p) \big) = c_7 n^2 \theta(p).$$
\end{proof}

\subsection{Nice circuits} \label{sec:nice_circuits}

In Section \ref{sec:full_plane}, when we explain how to relate the full-plane process to the process in finite domains, the following quantity plays an important role. Recall that for a circuit $\circuit$, we denote by $\inter(\circuit)$ the set of vertices inside it. We introduce
$$X^{\circuit}_p := \big| \circuitinside^{\circuit}(p) \big|, \quad \text{where } \circuitinside^{\circuit}(p) := \{ v \in \inter(\circuit) \: : \: v \lra{p} \circuit \}.$$

We can obtain good estimates on this quantity when $\circuit$ is well-behaved, which occurs with high probability if $\circuit$ is obtained as the outermost black circuit $\cout_A$ in an annulus $A$. Before stating precise results, Lemmas \ref{lem:nice_circ} and \ref{lem:nice_circ2} below, we introduce a notation for quantiles.

\begin{definition} \label{def:quantile}
For a real-valued random variable $X$ and $\ve \in (0,1)$, we denote by $\ul Q_\ve(X)$ and $\ol Q_\ve(X)$ the (resp.) \emph{lower} and \emph{upper $\ve$-quantiles} of $X$, defined as
\begin{equation} \label{eq:lower_quant}
\ul Q_\ve(X) := \inf \{ x \in \RR \: : \: \PP(X \leq x) \geq \ve \}
\end{equation}
\begin{equation} \label{eq:upper_quant}
\hspace{-0.9cm} \text{and} \quad \ol Q_\ve(X) := \sup \{ x \in \RR \: : \: \PP(X \geq x) \geq \ve \}
\end{equation}
(so that $\PP(X < \ul Q_\ve(X)) \le \ve$ and $\PP(X > \ol Q_\ve(X)) \le \ve$).
\end{definition}

\begin{definition} \label{def:nice_circ}
For $p > p_c$ and $C > 0$, we say that a circuit $\circuit$ is \emph{$(p, C)$-nice} if $f_p(\circuit) \leq C m^2 \theta(p)$, where $m = \diam(\circuit)$ and
$$f_p(\circuit) := \sum_{i=1}^{\lceil \log_2 L(p) \rceil - 1} \big| \{ v \in \inter(\circuit) \: : \: 2^{i-1} \leq d(v,\circuit) < 2^i \} \big| \cdot \pi_1(2^{i-1}).$$
\end{definition}

\begin{lemma} \label{lem:nice_circ}
For all $\ve > 0$, there exists a constant $C > 0$ (depending only on $\ve$) such that: for all $p > p_c$ and $n \geq L(p)$,
$$\PP_p \Big( \cout_{\Ann_{n/2, n}} \text{ exists and is not } (p, C) \text{-nice} \Big) \leq \ve.$$
\end{lemma}

\begin{proof}[Proof of Lemma \ref{lem:nice_circ}]
We denote by $E_0$ the event that there exists a $p$-black circuit in $A = \Ann_{n/2, n}$, and by $\circuit = \cout_A$ the outermost such circuit when it exists (otherwise, $\circuit = \emptyset$ by convention). We have
$$\EE \big[ f_p(\circuit) \big] = \sum_{i=1}^{\lceil \log_2 L(p) \rceil - 1} \EE \Big[ \big| \{ v \in \inter(\circuit) \: : \: 2^{i-1} \leq d(v,\circuit) < 2^i \} \big| \Big] \cdot \pi_1(2^{i-1}),$$
and if $v$ is within a distance $2^i < L(p)$ from $\circuit$, then there exist two black arms in the annulus $\Ann_{2^i,L(p)}(v)$ (coming from the black circuit $\circuit$). Hence, if we denote $\pi_2 = \pi_{bw}$,
\begin{equation}
\EE \big[ f_p(\circuit) \big] \leq C'_1 n^2 \sum_{i=1}^{\lceil \log_2 L(p) \rceil - 1} \pi_2(2^i, L(p)) \pi_1(2^{i-1})
\end{equation}
(using \eqref{eq:Kesten1}). We have $\pi_1(L(p)) \asymp \pi_1(2^i) \pi_1(2^i, L(p))$ (from quasi-multiplicativity \eqref{eq:qmult}), and $\pi_1(2^i) \asymp \pi_1(2^{i-1})$ (from \eqref{eq:ext}), so
\begin{equation}
\EE \big[ f_p(\circuit) \big] \leq C'_2 n^2 \pi_1(L(p)) \sum_{i=1}^{\lceil \log_2 L(p) \rceil - 1} \frac{\pi_2(2^i, L(p))}{\pi_1(2^i, L(p))}.
\end{equation}
Since $\pi_1(L(p)) \asymp \theta(p)$ (from \eqref{eq:Kesten_theta_pi}), and
$$\frac{\pi_2(2^i, L(p))}{\pi_1(2^i, L(p))} \leq \bigg( \frac{2^i}{L(p)} \bigg)^{\eta}$$
for some $\eta > 0$ (this follows from the BK inequality, and the a-priori bound \eqref{eq:1arm_bound}), we finally obtain
\begin{equation}
\EE \big[ f_p(\circuit) \big] \leq C'_3 n^2 \theta(p).
\end{equation}
It then follows from Markov's inequality that there exists a constant $C'_4 = C'_4(\ve)$ such that for all $p > p_c$ and $n \geq L(p)$,
$$\PP \big( f_p(\circuit) \geq C'_4 n^2 \theta(p) \big) \leq \ve.$$
\end{proof}

\begin{lemma} \label{lem:nice_circ2}
For all $\ve > 0$ and $C > 0$, there exist constants $\ul c, \ol c > 0$ and $Y \geq 1$ (depending only on $\ve$ and $C$) such that the following property holds. For all $p > p_c$ and $n \geq Y L(p)$, every finite $Z \subseteq V(\TT)$, every collection $(\circuit^z)_{z \in Z}$ of $(p, C)$-nice circuits with disjoint interiors ($\inter(\circuit^z) \cap \inter(\circuit^{z'}) = \emptyset$ for all $z \neq z'$), and such that $\circuit^z \subseteq \Ann_{n/2, n}(z)$, we have: for all $p' \geq p$,
$$\ul c |Z| n^2 \theta(p') \leq \ul Q_\ve \Bigg( \sum_{z \in Z} X_{p'}^{\circuit^z} \Bigg) \leq \ol Q_\ve \Bigg( \sum_{z \in Z} X_{p'}^{\circuit^z} \Bigg) \leq \ol c |Z| n^2 \theta(p').$$
\end{lemma}

\begin{proof}[Proof of Lemma \ref{lem:nice_circ2}]
By Lemma \ref{lem:calV_quant}, there exists a universal constant $c_0 > 0$ and $Y = Y(\ve) \geq 1$ such that for all $p' > p_c$ and $n \geq Y L(p')$,
$$\PP_{p'} \big( \calV_n \geq c_0 n^2 \theta(p') \big) \geq 1 - \frac{\ve}{2}.$$
In the remainder of this proof, we consider $p' \geq p$, so that in particular, $L(p') \leq L(p)$. Since $X^{\circuit^z} \geq \calV_{n/2}(z)$ for every $z \in Z$ (from the definition \eqref{eq:def_Vn} of $\calV$), we have: for all $n \geq 2 Y L(p)$,
$$\PP \big( X_{p'}^{\circuit^z} \geq c_0 n^2 \theta(p') \big) \geq 1 - \frac{\ve}{2}.$$
Since the random variables $(X_{p'}^{\circuit^z})_{z \in Z}$ are independent (the circuits have disjoint interiors), we deduce the existence of $\ul c = \ul c (\ve)$ such that
$$\PP \Bigg( \sum_{z \in Z} X_{p'}^{\circuit^z} \geq \ul c |Z| n^2 \theta(p') \Bigg) > 1 - \ve$$
(by distinguishing the two cases $|Z|$ small, and $|Z|$ large enough). This finally implies
$$\ul Q_{\ve} \Bigg( \sum_{z \in Z} X_{p'}^{\circuit^z} \Bigg) \geq \ul c |Z| n^2 \theta(p').$$



In order to estimate the upper quantile of $\sum_{z \in Z} X_{p'}^{\circuit^z}$, let us fix some $z \in Z$, and write $\circuit = \circuit^z$. We subdivide the vertices in $\inter(\circuit)$ according to their distance to $\circuit$: if we denote $i_{\textrm{max}} := \lceil \log_2 L(p) \rceil$ and $i'_{\textrm{max}} := \lceil \log_2 L(p') \rceil$, we have
\begin{align*}
\EE \big[ X_{p'}^\circuit \big]  & \leq \sum_{i=1}^{i'_{\textrm{max}}-1} \big| \{v \in \inter(\circuit) \: : \: 2^{i-1} \leq d(v,\circuit) < 2^i \} \big| \cdot \PP_{p'}(0 \lra{} \partial \Ball_{2^{i-1}})\\[-1mm]
& \hspace{4cm} + C_2 n^2 \PP_{p'}(0 \lra{} \partial \Ball_{2^{i'_{\textrm{max}}-1}}),
\end{align*}
for some universal constant $C_2 > 0$. Using that $\PP_{p'}(0 \lra{} \partial \Ball_{2^{i-1}}) \asymp \pi_1(2^{i-1})$ (by \eqref{eq:Kesten1}) and $\PP_{p'}(0 \lra{} \partial \Ball_{2^{i'_{\textrm{max}}-1}}) \asymp \PP_{p'}(0 \lra{} \partial \Ball_{L(p')}) \asymp \theta(p')$ (by \eqref{eq:ext} and \eqref{eq:equiv_expdecay}), we obtain
\begin{align}
\EE \big[ X_{p'}^\circuit \big]  & \leq C_3 \sum_{i=1}^{i'_{\textrm{max}}-1} \big| \{v \in \inter(\circuit) \: : \: 2^{i-1} \leq d(v,\circuit) < 2^i \} \big| \cdot \pi_1(2^{i-1}) + C_4 n^2 \theta(p') \nonumber\\
& \leq C_3 f_p(\circuit) + C_4 n^2 \theta(p'). \label{eq:upper_bd_X_p'}
\end{align}
For the last inequality, we replaced $i'_{\textrm{max}}$ by $i_{\textrm{max}}$ in the summation, which we can do since $L(p') \leq L(p)$. Using that $f_p(\circuit) \leq C (2n)^2 \theta(p)$ (since $\circuit$ is $(p, C)$-nice), we deduce from \eqref{eq:upper_bd_X_p'} that $\EE \big[ X_{p'}^\circuit \big] \leq C_5 n^2 (\theta(p) + \theta(p')) \leq C_6 n^2 \theta(p')$. Hence,
$$\EE \Bigg[ \sum_{z \in Z} X_{p'}^{\circuit^z} \Bigg] \leq C_6 |Z| n^2 \theta(p').$$
An application of Markov's inequality now completes the proof of Lemma \ref{lem:nice_circ2}.
\end{proof}

\section{Deconcentration argument} \label{sec:deconcentration}

\subsection{Frozen percolation: notations}

We now go back to frozen percolation. Recall from Section \ref{sec:intro_fp} that $\PP_N^{(G)}$ refers to volume-frozen percolation with parameter $N \geq 1$ on a graph $G=(V,E)$. The set of frozen sites at time $p$ is denoted by $\frozen^{(G)}(p)$, and we simply write $\frozen(p)$ when $G$ is clear from the context. Let us also stress that $(\tau_v)_{v \in V(\TT)}$ provides a natural coupling of the processes on various subgraphs of $\TT$.

In a similar way as for the hole $\hole(p)$ in $\Cinf(p)$ (Definition \ref{def:hole_cinf}), we define the hole of the origin in the frozen percolation process, replacing $\Cinf(p)$ by the set of frozen sites at time $p$:

\begin{definition} \label{def:hole_fp}
For a subgraph $G$ of $\TT$, we denote by $\holeF^{(G)}(p)$ the connected component of the origin in $G \setminus (\frozen(p) \cup \dout \frozen(p))$ (and we take $\holeF^{(G)}(p) = \emptyset$ if $0$ belongs to $\frozen(p) \cup \dout \frozen(p)$).
\end{definition}

By analogy, $\holeF^{(G)}(p)$ is also called hole of the origin, in the frozen percolation process.


\subsection{Exceptional scales} \label{sec:exceptional_scales}

Heuristically, for ordinary percolation in a box with volume $\simeq K^2$, Lemma \ref{lem:largest_vol} implies that for $p > p_c$, a giant connected component arises, with volume $\simeq \theta(p) K^2$ (and all the other components are tiny). Hence, for volume-frozen percolation in this box, we expect the first freezing event to occur at a time $p$ such that $\theta(p) K^2 \simeq N$, i.e. (with $c_{\theta}$ as in Proposition \ref{prop:theta_over_pi})
$$c_{\theta} \pi_1(L(p)) K^2 \simeq N.$$
This freezing event then leaves holes with diameter of order $L(p)$, so that $L(p)$ can be seen as the next scale in the process.

This informal explanation leads us to define $\nextN(K) := K'$ via the equation $c_{\theta} \pi_1(K') K^2 \simeq N$. More precisely, for all $N \geq 1$ and $K$ large enough (so that $c_{\theta} K^2 > N$), we introduce
\begin{equation} \label{eq:def_next}
\nextN(K) := \sup \{ K' \, : \, c_{\theta} \pi_1(K') K^2 \geq N \}.
\end{equation}
We also use $\nextN^{-1}(K') := \inf \{ K \, : \, \nextN(K) \geq K' \}$.

We can now define inductively the sequence of exceptional scales $(m_k(N))_{k \geq 0}$ by: $m_0 = 1$, and for all $k \geq 0$,
\begin{equation} \label{def_mk}
m_{k+1}(N) = \nextN^{-1}(m_k(N)).
\end{equation}
It follows easily from the definitions and the monotonicity of $\pi_1$ that
\begin{equation} 
m_{k+1} = \left\lceil \left(\frac{N}{c_{\theta} \pi_1(m_k)}\right)^{1/2} \right\rceil,
\end{equation}
and that $(m_k(N))_{k \geq 0}$ is non-decreasing for every fixed $N \geq 1$. Note also that
$m_1(N) \sim c_0 \sqrt{N}$ as $N \to \infty$, for some constant $c_0>0$.

By using Lemma \ref{lem:arm_exp}, we can see that each $m_k$ follows a power law: $m_k(N) = N^{\delta_k + o(1)}$ as $N \to \infty$, where the sequence of exponents $(\delta_k)_{k \geq 0}$ satisfies
\begin{equation} \label{eq:rec_delta}
\delta_0 = 0, \quad \text{and } \: \delta_{k+1} = \frac{1}{2} + \frac{5}{96} \delta_k \:\:\: (k \geq 0).
\end{equation}
Note that this sequence is strictly increasing, and that it converges to $\delta_{\infty} = \frac{48}{91}$.

It is natural to introduce the (approximate) fixed point of $\nextN$:
\begin{equation} \label{eq:def_minf}
m_{\infty}(N) := \sup \{ m \, : \, c_{\theta} \pi_1(m) m^2 \leq N \}
\end{equation}
(note that if we consider critical percolation in a box of volume $m^2$, the quantity $\pi_1(m) m^2$ gives the order of magnitude for the volume of the largest connected components). Lemma \ref{lem:arm_exp} implies that $m_{\infty}(N) = N^{\delta_{\infty} + o(1)}$ as $N \to \infty$, where $\delta_{\infty} = \frac{48}{91}$ is the exponent found below \eqref{eq:rec_delta}. The following observation is useful later.

\begin{lemma} \label{lem:nextN}
There exist universal constants $c, \eta > 0$ such that: for all $N \geq 1$, all $K \leq m_{\infty}(N)$,
\begin{equation}
\frac{\nextN(K)}{K} \leq c \left( \frac{K}{m_{\infty}} \right)^{\eta}.
\end{equation}
\end{lemma}

\begin{proof}[Proof of Lemma \ref{lem:nextN}]
We know from the definitions of $\nextN(K)$ \eqref{eq:def_next} and $m_{\infty}$ \eqref{eq:def_minf} that $c_{\theta} \pi_1(\nextN(K)) K^2 \geq N \geq c_{\theta} \pi_1(m_{\infty}) m_{\infty}^2$, so
\begin{equation} \label{eq:nextN_pf1}
\frac{\pi_1(K)}{\pi_1(\nextN(K))} = \frac{c_{\theta} \pi_1(K) K^2}{c_{\theta} \pi_1(\nextN(K)) K^2} \leq \frac{c_{\theta} \pi_1(K) K^2}{c_{\theta} \pi_1(m_{\infty}) m_{\infty}^2}.
\end{equation}
It follows from \eqref{eq:1arm_bound} that
\begin{equation} \label{eq:nextN_pf2}
\frac{\pi_1(K)}{\pi_1(\nextN(K))} \geq c_1 \left( \frac{\nextN(K)}{K} \right)^{1/2}
\end{equation}
and
\begin{equation} \label{eq:nextN_pf3}
\frac{c_{\theta} \pi_1(K) K^2}{c_{\theta} \pi_1(m_{\infty}) m_{\infty}^2} \leq c_2 \left( \frac{K}{m_{\infty}} \right)^{3/2}
\end{equation}
for some $c_1, c_2 > 0$. The desired result then follows by combining \eqref{eq:nextN_pf1}, \eqref{eq:nextN_pf2}, and \eqref{eq:nextN_pf3}.
\end{proof}

This lemma implies that if $\tilde{m}(N) \ll m_{\infty}(N)$ as $N \to \infty$, then $\nextN(\tilde{m}) \ll \tilde{m}$. It holds in particular for $\tilde{m} = m_k$ ($k \geq 0$), since $m_k(N) = N^{\delta_k + o(1)}$, with $\delta_k < \delta_{\infty}$.

\begin{remark}
Although our definition of exceptional scales and the one in \cite{BN15} (call it $(m'_k(N))_{k \geq 0}$) differ slightly, they are equivalent in the following sense: for every $k \geq 1$, $m_k(N) \asymp m'_k(N)$ as $N \to \infty$. In particular, the results from \cite{BN15} mentioned below also apply with this modified definition.
\end{remark}

Finally, we define the corresponding times $q_k = q_k(N)$ by:
\begin{equation} \label{eq:def_qk}
L(q_k) = m_k, \quad \text{with } \: q_k \in (p_c,1) \:\:\: (k \in \NN \cup \{\infty\})
\end{equation}
(recall that $L$ is continuous and strictly decreasing on $(p_c,1)$, see the end of Section \ref{sec:notations}). Our analysis focuses on the time window $[q_{\infty},q_1]$, $q_1$ being roughly the time when the last frozen clusters may appear.

\bigskip

Let us now recall the main results from \cite{BN15} about the scales $(m_k)_{k \geq 1}$, showing that they indeed play a particular role. The first theorem corresponds to the case when one starts with a box of side length of order $m_k$, for some fixed $k \geq 1$.

\begin{theorem}[\cite{BN15}, Theorem 1] \label{thm:exceptional_scales1}
Let $k \geq 2$ be fixed. For every $C \geq 1$, every function $\tilde{m}(N)$ that satisfies
\begin{equation}
C^{-1} m_k(N) \leq \tilde{m}(N) \leq C m_k(N)
\end{equation}
for $N$ large enough, we have
\begin{equation}
\liminf_{N \to \infty} \PP_N^{(\Ball_{\tilde{m}(N)})}(\text{$0$ is frozen at time $1$}) > 0.
\end{equation}
\end{theorem}

The second theorem deals with the case when one starts far from the exceptional scales.

\begin{theorem}[\cite{BN15}, Theorem 2] \label{thm:exceptional_scales2}
For every integer $k \geq 0$ and every $\ve>0$, there exists a constant $C=C(k,\ve) \geq 1$ such that: for every function $\tilde{m}(N)$ that satisfies
\begin{equation}
C m_k(N) \leq \tilde{m}(N) \leq C^{-1} m_{k+1}(N)
\end{equation}
for $N$ large enough, we have
\begin{equation}
\limsup_{N \to \infty} \PP_N^{(\Ball_{\tilde{m}(N)})}(\text{$0$ is frozen at time $1$}) \leq \ve.
\end{equation}
\end{theorem}

These two results were proved by induction, and for that, we established in \cite{BN15} some slightly stronger versions that we now state. For a circuit $\gamma$, we denote by $\mathcal{D}(\gamma) \subseteq \TT$ the domain that it encloses. For any $0 < n_1 < n_2$, we introduce
\begin{itemize}
\item $\Gamma_N(n_1, n_2) = \{$for every circuit $\gamma$ in $\Ann_{n_1, n_2}$, for the process in $\mathcal{D}(\gamma)$ with parameter $N$, $0$ is frozen at time $1\}$,
\item and $\tilde{\Gamma}_N(n_1, n_2) = \{$there exists a circuit $\gamma$ in $\Ann_{n_1, n_2}$ such that for the process in $\mathcal{D}(\gamma)$ with parameter $N$, $0$ is frozen at time $1\}$.
\end{itemize}
Here, we use the natural coupling for the frozen percolation processes in various subgraphs of $\TT$.

\begin{proposition}[\cite{BN15}, Proposition 2] \label{prop:exceptional_scales1}
For any $k \geq 2$, and $0 < C_1 < C_2$, we have
\begin{equation}
\liminf_{N \to \infty} \mathbb{P}(\Gamma_N(C_1 m_k(N), C_2 m_k(N))) > 0.
\end{equation}
This result also holds for $k=1$ under the extra condition that $C_1 > C_0$, where $C_0 > 0$ is a universal constant.
\end{proposition}

\begin{proposition}[\cite{BN15}, Proposition 3] \label{prop:exceptional_scales2}
Let $k \geq 0$, $\ve>0$, and $0 < C_1 < C_2$. Then there exists a constant $C=C(k,\ve,C_1,C_2)$ such that: for every function $\tilde{m}(N)$ that satisfies
\begin{equation}
C m_k(N) \leq C_1 \tilde{m}(N) \leq C_2 \tilde{m}(N) \leq C^{-1} m_{k+1}(N)
\end{equation}
for $N$ large enough, we have
\begin{equation}
\limsup_{N \to \infty} \mathbb{P}(\tilde{\Gamma}_N(C_1 \tilde{m}(N), C_2 \tilde{m}(N))) \leq \ve.
\end{equation}
\end{proposition}

For future use, let us note that Proposition \ref{prop:exceptional_scales2} can be formulated in the following way, which may look stronger at first sight. For all $k \geq 0$, $\ve>0$, and $0 < C_1 < C_2$, there exist $C$ and $N_0$ such that: for all $N \geq N_0$ and all $\tilde{m}$ with $C m_k(N) \leq C_1 \tilde{m} \leq C_2 \tilde{m} \leq C^{-1} m_{k+1}(N)$, we have
$$\mathbb{P}(\tilde{\Gamma}_N(C_1 \tilde{m}, C_2 \tilde{m})) \leq \ve.$$

\begin{remark}
Although we are not using it later, we mention that with small adjustments to the proofs of Propositions \ref{prop:exceptional_scales1} and \ref{prop:exceptional_scales2}, we can also get some information on the size of the final cluster $\cluster_1(0)$ of the origin. For any $0 < n_1 < n_2$ and $M \geq 1$, let us introduce
\begin{itemize}
\item $\Gamma_N^{(M)}(n_1, n_2) = \{$for every circuit $\gamma$ in $\Ann_{n_1, n_2}$, for the process in $\mathcal{D}(\gamma)$ with parameter $N$, $|\cluster_1(0)| \notin \big( M, \frac{N}{M} \big) \}$,

\item and $\tilde{\Gamma}_N^{(M)}(n_1, n_2) = \{$there exists a circuit $\gamma$ in $\Ann_{n_1, n_2}$ such that for the process in $\mathcal{D}(\gamma)$ with parameter $N$, $|\cluster_1(0)| \notin \big( M, \frac{N}{M} \big) \}$.
\end{itemize}
We can then distinguish the same two cases as before.
\begin{itemize}
\item For all $k \geq 2$, $\ve>0$, and $0 < C_1 < C_2$, there exists $M > 1$ such that:
$$\liminf_{N \to \infty} \mathbb{P}(\Gamma_N^{(M)}(C_1 m_k(N), C_2 m_k(N))) \geq 1 - \ve.$$
Moreover, we can also show that each of the three cases $|\cluster_1(0)| \leq M$ ($\cluster_1(0)$ is microscopic), $|\cluster_1(0)| \in \big[ \frac{N}{M}, N)$ (macroscopic and non-frozen), and $|\cluster_1(0)| \geq N$ (macroscopic and frozen) has a probability bounded away from $0$ as $N \to \infty$. 

\item For all $k \geq 0$, $\ve>0$, $0 < C_1 < C_2$, and $M > 1$, there exists a constant $C = C(k, \ve, C_1, C_2, M)$ such that: for every function $\tilde{m}(N)$ that satisfies
\begin{equation}
C m_k(N) \leq C_1 \tilde{m}(N) \leq C_2 \tilde{m}(N) \leq C^{-1} m_{k+1}(N)
\end{equation}
for $N$ large enough, we have
\begin{equation} \label{eq:limsup_Gamma_tilde}
\limsup_{N \to \infty} \mathbb{P}(\tilde{\Gamma}_N^{(M)}(C_1 \tilde{m}(N), C_2 \tilde{m}(N))) \leq \ve.
\end{equation}
\end{itemize}
\end{remark}

\subsection{Some processes associated with frozen percolation} \label{sec:associated_chains}

We now present several random sequences related to frozen percolation in a simply connected, bounded domain $\Lambda \subseteq \TT$. For all $p > p_c$, we denote by $\distvol_p$ the distribution of $\frac{|\hole(p)|}{L(p)^2}$. In the following definitions, some value of the parameter $N \geq 1$ is fixed. Recall that $c_{\theta}$ is the constant in Proposition \ref{prop:theta_over_pi}, and that $\nextN$ is defined in \eqref{eq:def_next}.

\begin{itemize}
\item[(i)] First, we can consider the sequence of successive holes around $0$ for the frozen percolation process in $\Lambda$.
\begin{itemize}
\item We start with $\Lambda_0 = \Lambda$.

\item Given $\Lambda_i$ ($i = 0, \ldots, k-1$), $p_{i+1}$ is the time of the first freezing event for the frozen percolation process in $\Lambda_i$,

\item and $\Lambda_{i+1} = \holeF^{(\Lambda_i)}(p_{i+1})$ (see Definition \ref{def:hole_fp}).
\end{itemize}

\item[(ii)] Given an initial value $K > 0$, we define the (deterministic) sequence $(K_i)_{0 \leq i \leq k}$ by:
$$K_0 = K, \quad \text{and } \: K_{i+1} = \nextN(K_i) \:\:\: (i = 0, \ldots, k-1).$$
Later, $K$ typically depends on $N$. We think of the $K_i$'s as reference scales, at which the successive freezing events occur, as explained in Section \ref{sec:exceptional_scales}.

\item[(iii)] Following the same heuristic explanation as in the beginning of Section \ref{sec:exceptional_scales}, we expect the first freezing event in a domain $\Lambda^*$ to occur at a time $p^*$ such that $c_{\theta} \pi_1(L(p^*)) |\Lambda^*| \simeq N \simeq c_{\theta} \pi_1(\nextN(K)) K^2$, and so (using \eqref{lem:Rob_ratio_lim}) $\frac{L(p^*)}{\nextN(K)} \simeq \big( \frac{K^2}{|\Lambda^*|} \big)^{-48/5}$. Moreover, we expect the frozen percolation hole created in this way to look like $\hole^{(\Lambda^*)}(p^*) \simeq \hole(p^*)$ (see \eqref{eq:def_hole_domain} and Remark \ref{rem:H_HLambda}). This leads us to introduce the following sequences of (random) sets $(\Lambda^*_i)_{0 \leq i \leq k}$ and (random) times $(p^*_i)_{1 \leq i \leq k}$. We expect them to approximate the real process in $\Lambda$, which we prove rigorously in Section \ref{sec:finite_box}.
\begin{itemize}
\item We start with $\Lambda^*_0 = \Lambda$, and some value $K_0 > 0$ (typically, $K_0$ is chosen later of the same order of magnitude as $\diam(\Lambda)$).

\item Given $\Lambda^*_i$ ($i = 0, \ldots, k-1$), $p^*_{i+1}$ is defined by
\begin{equation} \label{eq:def_chain_p*}
\frac{L(p^*_{i+1})}{K_{i+1}} = \left( \frac{|\Lambda^*_i|}{K_i^2} \right)^{48/5},
\end{equation}

\item and then, we take $\Lambda^*_{i+1} = \hole(p^*_{i+1})$.
\end{itemize}

\item[(iv)] We also introduce the chain $(p^{**}_i)_{1 \leq i \leq k}$, defined by taking $\Lambda^{**}_0 = \Lambda$ and 
$$\frac{L(p^{**}_{i+1})}{K_{i+1}} = \left( \frac{|\Lambda^{**}_i|}{K_i^2} \right)^{48/5}$$
(as for $(p^*_i)_{1 \leq i \leq k}$), but $\Lambda^{**}_{i+1} = \hole^{(\Lambda^{**}_i)}(p^{**}_{i+1})$.

\item[(v)] Finally, we use a slight modification of the chain $(p^*_i)_{1 \leq i \leq k}$. The right-hand side of \eqref{eq:def_chain_p*} is clearly equal to $\big( \frac{|\Lambda^*_i|}{L(p^*_i)^2} \big)^{48/5} \big( \frac{L(p^*_i)}{K_i} \big)^{96/5}$, which suggests to define a chain $(\tilde{p}_i)_{0 \leq i \leq k}$ by
\begin{equation} \label{eq:def_ptilde}
\frac{L(\tilde{p}_{i+1})}{K_{i+1}} = \tilde{\alpha}_i^{48/5} \left( \frac{L(\tilde{p}_i)}{K_i} \right)^{96/5},
\end{equation}
where $\tilde{\alpha}_i$ is a fresh random variable with distribution $\distvol_{\tilde{p}_i}$ (in other words, given $\tilde{p}_i$, $\tilde{\alpha}_i$ is independent of the ``past'').
\end{itemize}

This last chain $(\tilde{p}_i)_{0 \leq i \leq k}$ is exactly a Markov chain, which makes it more convenient to work with. In particular, we start by proving deconcentration for this chain, in Section \ref{sec:deconcentration_chain}, based on an abstract result established in Section \ref{sec:abstract_deconcentration}. Moreover, it is easy to see that it behaves, essentially, in the same way as $(p^*_i)_{0 \leq i \leq k}$ and $(p^{**}_i)_{0 \leq i \leq k}$, as we explain now.

We denote by $d_{TV}$ the total variation distance between two distributions, and with a slight abuse of notation, we also talk about the total variation distance between two random variables $X$ and $Y$ (defined as the distance between their respective distributions).

In the following lemma, the chain $(L(\tilde{p}_i))$ starts at step $i=1$, with the value $L(\tilde{p}_1) = L(p^*_1)$ ($= L(p^{**}_1)$), which by definition \eqref{eq:def_chain_p*} is a deterministic function of $\Lambda^*_0$.

\begin{lemma} \label{lem:ptilde_p*}
For all $k \geq 1$, $\ve > 0$, and $c_2 > c_1 > 0$, there exist $M_0, N_0 \geq 1$ such that: for all $N \geq N_0$, $K_0 \in [m_{k+1}(N), m_{\infty}(N) / M_0]$, and $\Lambda^*_0 = \Lambda^{**}_0$ with $\Ball_{c_1 K_0} \subseteq \Lambda^*_0 \subseteq \Ball_{c_2 K_0}$, we have
$$d_{TV} \big( L(p^*_k), L(\tilde{p}_k) \big) \leq \ve \quad \text{and} \quad d_{TV} \big( L(p^{**}_k), L(\tilde{p}_k) \big) \leq \ve.$$
\end{lemma}

\begin{proof}[Proof of Lemma \ref{lem:ptilde_p*}]
Lemma \ref{lem:nextN} ensures that by choosing $M_0$ large enough, we are in a position to use Remark \ref{rem:H_HLambda} repeatedly: for each $i = 0, \ldots, k-1$, if $\Lambda^*_i = \Lambda^{**}_i$, then $p^*_{i+1} = p^{**}_{i+1}$ (from the definition) and Remark \ref{rem:H_HLambda} implies that for $N$ large enough, $\Lambda^*_{i+1} = \Lambda^{**}_{i+1}$ with probability at least $1 - \frac{\ve}{2 k}$. We deduce
\begin{equation} \label{eq:p*_p**}
\PP \big( \forall i \in \{1, \ldots, k\}, \: \: p^*_i = p^{**}_i \text{ and } \Lambda^*_i = \Lambda^{**}_i \big) \geq 1 - \frac{\ve}{2}.
\end{equation}
We can then compare $(p^{**}_i)_{1 \leq i \leq k}$ and $(\tilde{p}_i)_{1 \leq i \leq k}$ by using that the $(\Lambda^{**}_i)$ are stopping sets, which allows us to successively ``refresh'' the configuration inside them. Using again Remark \ref{rem:H_HLambda}, we get that for $N$ large enough,
$$d_{TV}(p^{**}_{i+1},\tilde{p}_{i+1}) \leq d_{TV}(p^{**}_i,\tilde{p}_i) + \frac{\ve}{2 k}$$
for all $0 \leq i \leq k-1$. Combined with \eqref{eq:p*_p**}, this yields the desired result.
\end{proof}

The chain $(p^*_i)_{1 \leq i \leq k}$ was defined as a natural introduction to $(p^{**}_i)_{1 \leq i \leq k}$ and $(\tilde{p}_i)_{1 \leq i \leq k}$, but we are not using it in our proofs, and it will not be mentioned in the rest of the paper.

\subsection{Abstract deconcentration result} \label{sec:abstract_deconcentration}

We now establish a general result that provides deconcentration for functions of independent random variables: we give a simple sufficient condition on such functions to ensure that they are spread out, i.e. that they cannot be concentrated on small intervals. This lemma is instrumental in the proofs of our main results, Theorems \ref{thm:full_plane} and \ref{thm:large_k}.

Let us mention that for sums of independent random variables, a result due to Le Cam can be applied (see \cite{LeCam}, and (B) in \cite{Esseen}). This deconcentration result is used in \cite{BC13}, to show that for two-dimensional critical percolation in a box, there exist macroscopic gaps between the sizes of the largest clusters. In some cases, it is even possible to obtain CLT-type results by uncovering a renewal structure. In particular, McLeish's CLT for martingale differences \cite{McLeish} is used in \cite{KestenZhangCLT} (for ``critical'' first-passage percolation in two dimensions -- the proofs also apply for the maximal number of disjoint open circuits surrounding the origin in 2D percolation at criticality), \cite{ZhangCLT} (for the number of open clusters in a box with side length $n$, for bond percolation on $\ZZ^d$) and \cite{YaoCLT} (for winding angles of arms in 2D critical percolation). CLT-type results are also obtained in \cite{DS12} for 2D invasion percolation, based on mixing properties. However, none of these techniques seems to be directly applicable in our setting.

We first introduce some notations. We use $\Omega_n=\{0,1\}^n$, and for $\omega \in \Omega_n$ and $1 \leq i \leq n$, we denote by $\omega^{(i)}$ (resp. $\omega_{(i)}$) the configuration that coincides with $\omega$ except possibly at index $i$, where it is equal to $1$ (resp. $0$). We also write $|\omega| = |\{ i\in\{1,\ldots,n\}   \: : \: \omega_i=1 \}|$.

Let us consider a family of independent Bernoulli-distributed random variables $(Y_i)_{i \geq 1}$, with corresponding parameters $r_i \in (0,1)$ (i.e. for each $i$, $\PP(Y_i = 1) = r_i$ and $\PP(Y_i = 0) = 1-r_i$), and a sequence of functions $f_n: \Omega_n \to \RR$. For $1 \leq i \leq n$ and $\omega \in \Omega_n$, we denote $\nabla_i f_n(\omega) = f_n(\omega^{(i)}) - f_n(\omega_{(i)})$.

\begin{lemma} \label{lem:deconcentration}
Assume that there exists $\ve>0$ such that $r_i \in (\ve,1-\ve)$ for all $i \geq 1$, and that for all $n \geq 1$, $f_n$ satisfies
\begin{equation}
\text{for all $1 \leq i \leq n$ and $\omega \in \Omega_n$,} \quad \nabla_i f_n(\omega) \geq 1.
\end{equation}
Then there exists a constant $c = c(\ve) \in (0,\infty)$ such that, for all $n \geq 1$ and every interval $I \subseteq \RR$,
$$\PP(f_n(Y_1,\ldots,Y_n) \in I) \leq \frac{c}{n^{1/2}} (|I|+1),$$
where we denote by $|I|$ the length of $I$.
\end{lemma}

\begin{remark}
Note that this result provides an upper bound on the \emph{L\'evy concentration function} of the random variable $X = f_n(Y_1,\ldots,Y_n)$, defined by $Q_X(\lambda) := \sup_{x \in \RR} \PP(X \in [x,x+\lambda])$.
\end{remark}

The proof of Lemma \ref{lem:deconcentration} is based on the following construction.

\begin{lemma} \label{lem:coupling}
Let $n \geq 1$, and denote $\omega = (Y_1,\ldots,Y_n)$. One can construct a sequence $\tilde{\omega}_{[0]} = (0, \ldots, 0), \tilde{\omega}_{[1]}, \ldots, \tilde{\omega}_{[n]} = (1, \ldots, 1)$ such that
\begin{itemize}
\item for every $i \in \{0,\ldots,n-1\}$, $\tilde{\omega}_{[i+1]}$ can be obtained from $\tilde{\omega}_{[i]}$ by switching exactly one coordinate from $0$ to $1$ (so that for each $i$, $|\tilde{\omega}_{[i]}| = i$),

\item for every $i \in \{0,\ldots,n\}$, $\tilde{\omega}_{[i]}$ has the same distribution as $\omega$ conditioned on $|\omega| = i$.
\end{itemize}
\end{lemma}

\begin{proof}[Proof of Lemma \ref{lem:coupling}]
As we explained, Lemma \ref{lem:deconcentration} is used in Section \ref{sec:deconcentration_chain} to show deconcentration for the chain $(L(\tilde{p}_i))_{0 \leq i \leq k}$, and for this application, we only need the case where $r_i = \frac{1}{2}$ for all $i \in \{1, \ldots, n\}$.

When all the parameters $(r_i)_{1 \leq i \leq n}$ are equal, the construction of $(\tilde{\omega}_{[i]})_{0 \leq i \leq n}$ is straightforward: indeed, given $\tilde{\omega}_{[i]}$, we can produce $\tilde{\omega}_{[i+1]}$ by considering the $n-i$ coordinates which are equal to $0$, choose one of them uniformly at random, and switch it to $1$. In this paper, we do not need the general case, which seems to be trickier. Nevertheless, since we find Lemma \ref{lem:deconcentration} interesting in itself, we provide a proof of Lemma \ref{lem:coupling} in Appendix \ref{sec:app_coupling}.
\end{proof}

\begin{proof}[Proof of Lemma \ref{lem:deconcentration}]
We use the coupling provided by Lemma \ref{lem:coupling}: for every interval $I \subseteq \RR$,
$$\PP(f_n(\omega) \in I) = \sum_{i=0}^n \PP(f_n(\omega) \in I \: | \: |\omega|=i) \PP(|\omega| = i) = \sum_{i=0}^n \PP(f_n(\tilde{\omega}_{[i]}) \in I) \PP(|\omega| = i).$$
We now use that $\PP(|\omega| = i) \leq \frac{c}{n^{1/2}}$, where $c = c(\ve) < \infty$ depends only on $\ve$, and obtain
$$\PP(f_n(\omega) \in I) \leq \frac{c}{n^{1/2}} \sum_{i=0}^n \tilde{\EE}\left[\mathbbm{1}_{f_n(\tilde{\omega}_{[i]}) \in I}\right] = \frac{c}{n^{1/2}} \tilde{\EE}\left[ \sum_{i=0}^n \mathbbm{1}_{f_n(\tilde{\omega}_{[i]}) \in I} \right].$$
Finally, we use that $\sum_{i=0}^n \mathbbm{1}_{f_n(\tilde{\omega}_{[i]}) \in I} \leq |I|+ 1$, from our assumption on $f_n$.
\end{proof}

\subsection{Deconcentration for $L(\tilde{p}_k)$ and $L(p^{**}_k)$} \label{sec:deconcentration_chain}

We now obtain deconcentration for the Markov chain $(L(\tilde{p}_i))_{0 \leq i \leq k}$ by applying the abstract result from the previous section, Lemma \ref{lem:deconcentration}.

\begin{proposition} \label{prop:dec_ptilde}
For all $\ve > 0$, $\lambda > 1$, and $c_2 > c_1 > 0$, there exists $k_0 \geq 1$ such that the following holds. For all $k \geq k_0$, there exists $N_0 \geq 1$ such that: for all $N \geq N_0$, $K_0 \geq m_{k+1}(N)$, and $\tilde{p}_0$ with $c_1 K_0 \leq L(\tilde{p}_0) \leq c_2 K_0$,
\begin{equation} \label{eq:dec_ptilde}
\sup_{y > 0} \PP \big( L(\tilde{p}_k) \in (y, \lambda y) \big) < \ve.
\end{equation}
\end{proposition}

\begin{proof}[Proof of Proposition \ref{prop:dec_ptilde}]
Let us consider $\ve > 0$ and $\lambda > 1$, and take $k \geq 1$ (we explain later how to choose it). First, note that by iterating the definition \eqref{eq:def_ptilde} of $(\tilde{p}_i)_{0 \leq i \leq k}$, we obtain
\begin{equation} \label{eq:iteration_ptilde}
\frac{L(\tilde{p}_k)}{K_k} = \left( \frac{L(\tilde{p}_0)}{K_0} \right)^{2 \delta_0} \cdot \prod_{i=0}^{k-1} \tilde{\alpha}_i^{\delta_i}, \quad \text{where } \delta_i = \frac{1}{2}\Big( \frac{96}{5} \Big)^{k-i}.
\end{equation}

In order to use Lemma \ref{lem:deconcentration}, we describe the process $(L(\tilde{p}_i))_{0 \leq i \leq k}$ in terms of i.i.d. random variables $(U_i)_{0 \leq i \leq k-1}$ uniformly distributed on the interval $(0,1)$, as we explain now. For $p > p_c$ and $u \in (0,1)$, we introduce the lower $u$-quantile
$$q(p,u) := \ul Q_u \bigg( \frac{|\hole(p)|}{L(p)^2} \bigg)$$
(recall Definition \ref{def:quantile}). It follows from \eqref{eq:lower_quant} that if $U$ is a random variable uniform on $(0,1)$, then $q(p,U)$ has distribution $\distvol_p$, so that in the definition \eqref{eq:def_ptilde} of $(\tilde{p}_i)_{0 \leq i \leq k}$, we can use the representation
\begin{equation} \label{eq:rep_alphai}
\tilde{\alpha}_i = q(\tilde{p}_i,U_i) \quad (0 \leq i \leq k-1).
\end{equation}
Note that $\tilde{\alpha}_i$ is thus a function of $U_0, \ldots, U_i$.

From the upper bounds in \eqref{eq:bounds_hole1} and \eqref{eq:bounds_hole2}, we can find $C$ large enough (depending on $k$) such that if $U$ is uniformly distributed on $(0,1)$, we have: for all $p > p_c$,
\begin{equation}
\PP \left( q(p,U) \notin [C^{-1},C] \right) \leq \frac{\ve}{10 k}.
\end{equation}
In particular, for all $p > p_c$ and $u \in \big( \frac{\ve}{10 k}, 1- \frac{\ve}{10 k} \big)$, $q(p,u) \in [C^{-1},C]$. We thus introduce the event
\begin{equation} \label{eq:def_G1}
G_1 := \Big\{ \forall i \in \{0, \ldots, k-1\}, \: \: U_i \in \Big( \frac{\ve}{10 k}, 1- \frac{\ve}{10 k} \Big) \Big\},
\end{equation}
which satisfies $\PP(G_1) \geq 1 - \frac{\ve}{5}$. We can also take $N$ large enough so that if the event $G_1$ holds (which we assume from now on), then the $(\tilde{p}_i)_{0 \leq i \leq k}$ are sufficiently close to $p_c$ to allow us to apply Proposition \ref{prop:vol_hole} (with $\bar{\lambda} = C$) to each of them (here, we are using \eqref{eq:iteration_ptilde}, the hypothesis $K_0 \geq m_{k+1}(N)$, and $K_k = \nextN^{(k)}(K_0)$). In particular, this implies that in \eqref{eq:rep_alphai}, the combined effect of $U_0, \ldots, U_{i-1}$ on $\tilde{\alpha}_i$ (through $\tilde{p}_i$) is not very large: a multiplicative factor between $\beta^{-1}$ and $\beta$, where $\beta$ is as in Proposition \ref{prop:vol_hole}.

We now make the following key observation. For some given $U_0, \ldots, U_{k-1}$, let us assume that we change exactly one of them, say $U_i$ (in such a way that $G_1$ still holds), so that
\begin{itemize}
\item[(i)] $\tilde{\alpha}_i$ is multiplied by a factor at least $2 \beta$.
\end{itemize}
Note that
\begin{itemize}
\item[(ii)] $\tilde{\alpha}_0, \ldots, \tilde{\alpha}_{i-1}$ are not affected (indeed, they depend only on $U_0, \ldots, U_{i-1}$),

\item[(iii)] and $\tilde{\alpha}_{i+1}, \ldots, \tilde{\alpha}_{k-1}$ are each changed by a factor between $\beta^{-1}$ and $\beta$ (using the property mentioned in the paragraph below \eqref{eq:def_G1}).
\end{itemize}
Since the exponent $\delta_i$ of $\tilde{\alpha}_i$ in \eqref{eq:iteration_ptilde} satisfies (for $i \leq k-2$)
\begin{equation} \label{eq:sum_deltai}
\frac{\delta_i}{\sum_{j=i+1}^{k-1} \delta_j} \geq \frac{1}{\sum_{j=1}^{\infty} (\frac{5}{96})^j} = \frac{91}{5},
\end{equation}
these three properties together imply that $L(\tilde{p}_k)$ gets multiplied by a factor at least $2^{\delta_i} \geq 2^{48/5}$.

We use this observation to apply Lemma \ref{lem:deconcentration}, as follows. First, we can choose $\delta > 0$ small enough so that for all $p > p_c$,
$$q(p,1-\delta) \geq 2 \beta \cdot q \Big( p,\frac{1}{2} \Big)$$
(using the lower bound in \eqref{eq:bounds_hole1}). We also introduce a modification of $(U_i)_{0 \leq i \leq k-1}$: for $i \in \{0, \ldots, k-1\}$,
$$\tilde{U}_i := \left \{
\begin{array}{ll}
U_i - \frac{1}{2} & \quad \text{if } U_i \geq 1 - \delta,\\[1mm]
U_i & \quad \text{otherwise.}
\end{array}
\right.$$
Note that in order to prove our result, we can condition on $(\tilde{U}_i)_{0 \leq i \leq k-1}$, and prove that deconcentration holds in this case.

We can take $k$ large enough so that with probability at least $\frac{\ve}{5}$, the number $l$ of indices $i \in \{0, \ldots, k-1\}$ such that $\tilde{U}_i \in \big( \frac{1}{2} - \delta, \frac{1}{2} \big)$ (i.e. $U_i \in \big( \frac{1}{2} - \delta, \frac{1}{2} \big) \cup (1 - \delta, 1)$) is at least $\delta k$: let us call $G_2$ this event, and assume that it occurs. We can list the corresponding indices as $i_1, \ldots, i_l$. We then define, for all $j \in \{1, \ldots, l\}$,
$$Y_j := \left \{
\begin{array}{ll}
1 & \quad \text{if } U_{i_j} \geq 1 - \delta \text{ \: (so that $U_{i_j} = \tilde{U}_{i_j} + \frac{1}{2}$)},\\[1mm]
0 & \quad \text{otherwise \: (in this case, $U_{i_j} = \tilde{U}_{i_j}$)}.
\end{array}
\right.$$
Since we assumed $(\tilde{U}_i)_{0 \leq i \leq k-1}$ to be given, \eqref{eq:iteration_ptilde} allows us to see $L(\tilde{p}_k)$ as a function $g(Y_1, \ldots, Y_l)$, and the $(Y_j)_{1 \leq j \leq l}$ are independent Bernoulli($\frac{1}{2}$) distributed. We are thus in a position to apply Lemma \ref{lem:deconcentration}, to the function
$$f(Y_1, \ldots, Y_l) := \ln g(Y_1, \ldots, Y_l)$$
(note that the assertion below \eqref{eq:sum_deltai} ensures: for all $j \in \{1, \ldots l\}$, $\nabla_j f \geq \ln (2^{48/5}) > 1$), and we obtain
$$\PP \big( g(Y_1, \ldots, Y_l) \in (y, \lambda y) \big) = \PP \big( f(Y_1, \ldots, Y_l) \in (\ln y, \ln y + \ln \lambda) \big) \leq \frac{c}{l^{1/2}}(\ln \lambda + 1),$$
where $c$ is a universal constant. We deduce, using $l \geq \delta k$: for all $y \in \RR$,
$$\PP \big( L(\tilde{p}_k) \in (y, \lambda y) \big) \leq \frac{c}{(\delta k)^{1/2}}(\ln \lambda + 1) + \PP(G_1^c) + \PP(G_2^c) \leq \frac{c}{(\delta k)^{1/2}}(\ln \lambda + 1) + \frac{\ve}{5} + \frac{\ve}{5},$$
which is smaller than $\ve$ for $k$ large enough. This completes the proof of Proposition \ref{prop:dec_ptilde}.
\end{proof}

By combining Proposition \ref{prop:dec_ptilde} with Lemma \ref{lem:ptilde_p*}, we can deduce immediately a deconcentration result for $L(p^{**}_k)$, which is used in the next section (and we can now forget about the chain $(\tilde{p}_i)_{0 \leq i \leq k}$).

\begin{corollary} \label{cor:dec_p*}
For all $\ve > 0$, $\lambda > 1$, and $c_2 > c_1 > 0$, there exists $k_0 \geq 1$ such that the following property holds. For all $k \geq k_0$, there exist $M_0, N_0 \geq 1$ such that: for all $N \geq N_0$, $K_0 \in [m_{k+1}(N), m_{\infty}(N) / M_0]$, and $\Lambda^{**}_0$ with $\Ball_{c_1 K_0} \subseteq \Lambda^{**}_0 \subseteq \Ball_{c_2 K_0}$,
\begin{equation} \label{eq:dec_p*}
\sup_{y > 0} \PP \big( L(p^{**}_k) \in (y, \lambda y) \big) < \ve.
\end{equation}
\end{corollary}

\section{Frozen percolation in finite boxes} \label{sec:finite_box}

\subsection{Iteration lemma}

We establish now an iteration lemma for frozen percolation, that allows us to compare the real frozen percolation process to the chain $(p^*_i)$, for which we proved deconcentration in Section \ref{sec:deconcentration_chain}.

Before stating the lemma, we need to indroduce some terminology. Let us consider $n \geq 1$, and a partition $\{1, \ldots, n\} = I \sqcup J \sqcup K$. A function $f: x=(x_1,\ldots,x_n) \in \RR^n \mapsto f(x) \in \RR$ is said to be \emph{small when $(x_j)_{j \in J}$ are small and $(x_k)_{k \in K}$ are large} if: for all $\ve > 0$, all $(x_i)_{i \in I}$, there exists $C > 1$ such that whenever all $j \in J$ satisfy $x_j < C^{-1}$ and all $k \in K$ satisfy $x_k > C$, one has $f(x) < \ve$.

\begin{lemma} \label{lem:iteration_FP}
Let $l \geq 3$, $N \geq 1$, $K \in [m_2(N), m_l(N)]$, $c_2 > c_1 > \alpha > 0$, $\eta >0$, and $\beta \in (0,\frac{1}{10})$. Further, let $\Delta$ be a simply connected $(\alpha K, \eta)$-approximable set, with $\Ball_{c_1 K} \subseteq \Delta \subseteq \Ball_{c_2 K}$. Let $0 < c_{1,\new} < c_{2,\new}$, and $K_{\new} := \nextN(K)$ (with $\nextN$ as defined in \eqref{eq:def_next}). Then there exist $\alpha_{\new} > 0$, $\eta_{\new} > 0$ (small if $N$ is large), $\beta_{\new} > 0$ (small if $\beta$ and $\eta$ are small, and $N$ is large), a simply connected stopping set $\Delta_{\new}$, and $\ve > 0$ (small if $\eta, c_{1,\new}, \beta$ are small, and $N, c_{2,\new}$ are large), such that with probability $> 1 - \ve$, each of the following properties holds.
\begin{itemize}
\item[(i)] $\Delta_{\new}$ is $(\alpha_{\new} K_{\new}, \eta_{\new})$-approximable, and
$$\Ball_{c_{1,\new} K_{\new}} \subseteq \Delta_{\new} \subseteq \Ball_{c_{2,\new} K_{\new}}.$$

\item[(ii)] For every simply connected $\Lambda$ with $(\Dint{\Delta}{\alpha K})_{(\beta)} \subseteq \Lambda \subseteq \Delta$, the first freezing event for the frozen percolation process in $\Lambda$ leaves a hole $\Lambda_{\new}$ around $0$ that satisfies
$$(\Dint{\Delta}{\alpha_{\new} K_{\new}}_{\new})_{(\beta_{\new})} \subseteq \Lambda_{\new} \subseteq \Delta_{\new}.$$

\item[(iii)] For each $\Lambda^{**}$ with $(\Dint{\Delta}{\alpha K})_{(\beta)} \subseteq \Lambda^{**} \subseteq \Delta$, if we define $p^{**}$ by
\begin{equation} \label{eq:def_p*}
\frac{L(p^{**})}{K_{\new}} = \left( \frac{|\Lambda^{**}|}{K^2} \right)^{48/5},
\end{equation}
then $\Lambda^{**}_{\new} := \hole^{(\Lambda^{**})}(p^{**})$ (recall the definition in \eqref{eq:def_hole_domain}) satisfies
$$(\Dint{\Delta}{\alpha_{\new} K_{\new}}_{\new})_{(\beta_{\new})} \subseteq \Lambda^{**}_{\new} \subseteq \Delta_{\new}.$$
\end{itemize}
\end{lemma}


\begin{proof}[Proof of Lemma \ref{lem:iteration_FP}]
(i) Let us consider $l, N, K, c_1, c_2, \alpha, \eta, \beta$ and $\Delta$ as in the statement, and recall that $K_{\textrm{new}}$ was defined by
\begin{equation} \label{eq:def_Knew}
c_{\theta} K^2 \pi_1(K_{\new}) \simeq N,
\end{equation}
where $c_{\theta}$ is the constant appearing in Proposition \ref{prop:theta_over_pi}. Let $\ve > 0$, and consider some $\delta > 0$ (it will become clear later how to take $\delta$ sufficiently small). Let us define $p^-$ and $p^+$ by
\begin{equation} \label{eq:def_p-}
|\Delta| \cdot \theta(p^-) = N (1 - \delta),
\end{equation}
and
\begin{equation} \label{eq:def_p+}
\big| (\Dint{\Delta}{\alpha K})_{(\beta)} \big| \cdot \theta(p^+) = N (1 + \delta).
\end{equation}
We first make a few observations on some of the scales involved.
\begin{itemize}
\item Since $|\Delta| \asymp K^2$ and $\theta(p^-) \asymp \pi_1(L(p^-))$, it follows from the definitions of $K_{\new}$ \eqref{eq:def_Knew} and $p^-$ \eqref{eq:def_p-} that $\pi_1(K_{\new}) \asymp \pi_1(L(p^-))$, and so
\begin{equation} \label{eq:Lp-}
L(p^-) \asymp K_{\new}.
\end{equation}

\item The assumption that $K \leq m_l(N)$ implies (using Lemma \ref{lem:nextN}) that
\begin{equation} \label{eq:Knew_K}
K_{\new} = \nextN(K) \ll K \quad \text{as $N \to \infty$}.
\end{equation}
\end{itemize}

Now we start with the proof proper. From the definitions of $p^-$ and $p^+$ (\eqref{eq:def_p-} and \eqref{eq:def_p+}), we have
\begin{equation} \label{eq:theta_p+_theta_p-_1}
\frac{\theta(p^-)}{\theta(p^+)} = \frac{1-\delta}{1+\delta} \cdot \frac{\big| (\Dint{\Delta}{\alpha K})_{(\beta)} \big|}{|\Delta|} \geq \frac{1-\delta}{1+\delta} \cdot (1 - 4 \beta) (1 - \eta),
\end{equation}
using \eqref{eq:shrinking_vol} and the $(\alpha K, \eta)$-approximability of $\Delta$. An inequality in the other direction comes from Proposition \ref{prop:theta_over_pi} and Lemma \ref{lem:Rob_ratio_lim}:
\begin{equation} \label{eq:theta_p+_theta_p-_2}
\frac{\theta(p^-)}{\theta(p^+)} \leq (1 + \delta) \frac{\pi_1(L(p^-))}{\pi_1(L(p^+))} \leq (1 + \delta)^2 \left( \frac{L(p^-)}{L(p^+)} \right)^{-5/48}
\end{equation}
if $N$ is large enough so that $p^-$ and $p^+$ are sufficiently close to $p_c$. Combining \eqref{eq:theta_p+_theta_p-_1} with \eqref{eq:theta_p+_theta_p-_2} then gives
\begin{equation} \label{eq:Lp-_Lp+}
1 \leq \frac{L(p^-)}{L(p^+)} \leq \left[ \frac{1-\delta}{(1+\delta)^3} (1 - 4 \beta) (1 - \eta) \right]^{-48/5}.
\end{equation}

For any given $\beta', \eta' > 0$, Lemmas \ref{lem:approx_H} and \ref{lem:cont_vol} imply (using \eqref{eq:Lp-_Lp+}) that if $\delta, \beta, \eta$ are sufficiently small, and $N$ is sufficiently large, there exists $\alpha' > 0$ such that: with probability $> 1 - \ve$,
\begin{equation} \label{eq:approx_hole_p-}
\hole(p^-) \text{ is } (\alpha' L(p^-), \eta') \text{-approximable}
\end{equation}
(in particular, it contains the block $b_{\alpha' L(p^-)}$), and
\begin{equation} \label{eq:inclusion_p-_p+}
(\Dint{\hole(p^-)}{\alpha' L(p^-)})_{(\beta')} \subseteq \hole(p^+).
\end{equation}

Now, we consider the set
\begin{equation} \label{eq:def_delta_new}
\Delta_{\new} := \hole^{(\Delta)}(p^-),
\end{equation}
which is a stopping set. Moreover, since (as observed in \eqref{eq:Lp-} and \eqref{eq:Knew_K}) $L(p^-) \asymp K_{\new} \ll K$,
\begin{equation} \label{eq:hole_delta}
\hole^{(\Delta)}(p^-) = \hole(p^-)
\end{equation}
with probability $> 1 - \ve$.

The a-priori bounds from Lemma \ref{lem:apriori_hole} imply the existence of $0 < c_3 < c_4$ such that: for $N$ large enough,
\begin{equation} \label{eq:a_priori_hole_p-}
\Ball_{c_3 K_{\new}} \subseteq \hole(p^-) \subseteq \Ball_{c_4 K_{\new}}
\end{equation}
with probability $> 1 - \ve$ (using \eqref{eq:Lp-}). Further, \eqref{eq:approx_hole_p-} implies that $\hole(p^-)$ is $(\alpha_{\new} K_{\new}, \eta')$-approximable, with $\alpha_{\new} = \alpha' \frac{L(p^-)}{K_{\new}}$. Also, note that $\alpha_{\new} \asymp \alpha'$ (using again \eqref{eq:Lp-}).

Hence, by \eqref{eq:approx_hole_p-}, \eqref{eq:def_delta_new}, \eqref{eq:hole_delta} and \eqref{eq:a_priori_hole_p-}, our choice of $\Delta_{\new}$ satisfies the desired properties in (i) (with probability $> 1 - 3 \ve$), if we choose $\alpha_{\new}$ as indicated, $\beta_{\new} = \beta'$, $c_{1,\new} = c_3$, $c_{2,\new} = c_4$, and $\eta_{\new} = \eta'$.

(ii) We proceed with property (ii). From now on, we take $\ve < \frac{\delta}{1 + \delta}$ (the reason will become clear soon). In the remainder of this proof, we write $\Di{\Delta}$ for $(\Dint{\Delta}{\alpha K})_{(\beta)}$. Since $\Delta$ is $(\alpha K, \eta)$-approximable, we can apply Lemma \ref{lem:largest_vol} (taking $\eta$ sufficiently small and $N$ sufficiently large, so that $\frac{L(p^-)}{\alpha K}$ is sufficiently small, as required by that lemma). This gives that, with probability $> 1 - \ve$, each of the properties (1)-(6) below occurs.
\begin{itemize}
\item[(1)] $|\lclus_{\Delta}(p^-)| < (1 + \ve) \theta(p^-) |\Delta| = (1 + \ve) (1 - \delta) N < N$, where the equality comes from the definition \eqref{eq:def_p-}, and the last inequality comes from the choice of $\ve$.

\item[(2)] Similarly, $|\lclus_{\Di{\Delta}}(p^-)| > \frac{3}{4} N$.

\item[(3)] $\lclus_{\Di{\Delta}}(p^-)$ contains a $p^-$-black circuit in $\Ann_{\frac{1}{8} \alpha K, \frac{1}{4} \alpha K}$ which is connected to $\infty$ by a $p^-$-black path.

\item[(4)] $|\lclus_{\Delta}(p^+)| < \frac{5}{4} N$.

\item[(5)] $|\lclus_{\Di{\Delta}}(p^+)| > (1 - \ve) (1 + \delta) N > N$, by using the definition \eqref{eq:def_p+}, and again the choice of $\ve$.

\item[(6)] The second-largest $p^+$-black cluster in $\Delta$ has volume $< \frac{N}{4}$.
\end{itemize}

We now claim the following (deterministic) fact: properties (1)-(6) imply together that, for every $p' \in [p^-, p^+]$ and every $\Delta'$ with $\Di{\Delta} \subseteq \Delta' \subseteq \Delta$,
\begin{itemize}
\item[(7)] $\lclus_{\Delta'}(p') \supseteq \lclus_{\Di{\Delta}}(p^-),$

\item[(8)] and the second-largest $p'$-black cluster in $\Delta'$ has volume $< \frac{3}{4} N$.
\end{itemize}
Let us first prove the claim. Since $\Delta' \supseteq \Di{\Delta}$ and $p' \geq p^-$, it is clear that either (7) holds, or: the two clusters in (7) are disjoint and (by (2)) both larger than $\frac{3}{4} N$. However, the latter leads to violation of (6) (if at time $p^+$ they are still disjoint) or violation of (4) (if by time $p^+$ they have merged). This proves (7), and practically the same argument gives (8), which proves the claim.

We now go back to the proof of Lemma \ref{lem:iteration_FP}, and we consider $\Lambda$ as in the statement of (ii). Let also $\tilde{p}$ be the first freezing time for the frozen percolation process in $\Lambda$, and let $\tilde{\cluster}$ be the corresponding cluster that freezes. From the claim above, we get
$$\tilde{\cluster} = \lclus_{\Lambda}(\tilde{p}) \supseteq \lclus_{\Di{\Delta}}(p^-).$$
In particular, $\tilde{\cluster}$ contains the $p^-$-black circuit mentioned in (3). Hence, since $\Lambda_{\new}$ is by definition the hole of the origin in $\tilde{\cluster}$,
$$\Lambda_{\new} \subseteq \hole(p^-) = \hole^{(\Delta)}(p^-) = \Delta_{\new}$$
(using \eqref{eq:hole_delta}, and then \eqref{eq:def_delta_new}), i.e. we have obtained the second inclusion in (ii).

Now we handle the first inclusion in (ii). Since the above-mentioned $p^-$-black circuit is also a $p^+$-black circuit contained in $\Cinf(p^+)$, we have that $\tilde{\cluster}$ contains a $p^+$-black circuit (around $0$) in $\Cinf(p^+)$. Hence, since $\tilde{p} < p^+$ (and $\Lambda_{\new}$ is the hole of $0$ in $\tilde{\cluster}$),
\begin{equation} \label{eq:inclusion_hole_p+}
\Lambda_{\new} \supseteq \hole(p^+) \supseteq (\Dint{\hole(p^-)}{\alpha' L(p^-)})_{(\beta')} = (\Dint{\Delta_{\new}}{\alpha_{\new} K_{\new}})_{(\beta_{\new})}
\end{equation}
(where the second inclusion comes from \eqref{eq:inclusion_p-_p+}, and we use \eqref{eq:hole_delta} and \eqref{eq:def_delta_new} for the equality, as well as the definitions of $\alpha_{\new}$, $\beta_{\new}$ and $K_{\new}$). This completes the argument for (ii).

(iii) Let us consider $\Lambda^{**}$ as in the statement. We pointed out earlier that with high probability ($> 1 - \ve$),
$$\Delta_{\new} = \hole^{(\Delta)}(p^-) = \hole(p^-)$$
(by \eqref{eq:def_delta_new} and \eqref{eq:hole_delta}), and for the same reason, we have, with high probability,
$$\hole^{(\Lambda^{**})}(p^{**}) = \hole(p^{**}).$$
By this and the second inclusion in \eqref{eq:inclusion_hole_p+}, it suffices to prove that $p^- < p^{**} < p^+$.

It follows from Proposition \ref{prop:theta_over_pi} that $\pi_1(L(p^-)) < \big(1+\frac{\delta}{2} \big) (c_{\theta})^{-1} \theta(p^-)$ for $N$ large enough. We then obtain from the definitions of $K_{\new}$ \eqref{eq:def_Knew} and $p^-$ \eqref{eq:def_p-} that
$$\frac{\pi_1(L(p^-))}{\pi_1(K_{\new})} < \left(1-\frac{\delta}{2}\right) \frac{K^2}{|\Delta|}.$$
This implies, with Lemma \ref{lem:Rob_ratio_lim}, that
\begin{equation}
\frac{L(p^-)}{K_{\new}} > \left( \left(1-\frac{\delta}{4}\right) \frac{K^2}{|\Delta|} \right)^{-48/5}
\end{equation}
for $N$ large enough, which yields
\begin{equation}
\frac{L(p^-)}{K_{\new}} > \left( \frac{|\Delta|}{K^2} \right)^{48/5} \geq \left( \frac{|\Lambda^{**}|}{K^2} \right)^{48/5} = \frac{L(p^{**})}{K_{\new}}
\end{equation}
(using \eqref{eq:def_p*}). Hence, $p^- < p^{**}$.

In a similar way, we can get from the definitions of $K_{\new}$ \eqref{eq:def_Knew} and $p^+$ \eqref{eq:def_p+}, combined with Proposition \ref{prop:theta_over_pi} and Lemma \ref{lem:Rob_ratio_lim}, that
\begin{equation}
\frac{L(p^+)}{K_{\new}} < \left( \frac{|(\Dint{\Delta}{\alpha K})_{(\beta)}|}{K^2} \right)^{48/5} \leq \left( \frac{|\Lambda^{**}|}{K^2} \right)^{48/5} = \frac{L(p^{**})}{K_{\new}},
\end{equation}
and so $p^+ > p^{**}$, which completes the proof of (iii).
\end{proof}

\subsection{Proof of Theorem \ref{thm:large_k}} \label{sec:proof_large_k}

In this section, we establish Theorem \ref{thm:large_k}. We actually prove the result below, which clearly implies Theorem \ref{thm:large_k}, and is used later (in Section \ref{sec:full_plane}) to study the full-plane process and to prove Theorem \ref{thm:full_plane}.

\begin{theorem} \label{thm:large_k_strong}
For all $\ve > 0$, there exists $\eta = \eta(\ve) > 0$ such that: for all $c_2 > c_1 > \alpha > 0$, there exists $k_0 \geq 1$ such that for all $k \geq k_0$, the following property holds. For all sufficiently large $N$, all $K \in (m_{k+2}(N), m_{k+5}(N))$, and all simply connected $(\alpha K, \eta)$-approximable sets $\Lambda$ with $\Ball_{c_1 K} \subseteq \Lambda \subseteq \Ball_{c_2 K}$, we have
$$\PP_N^{(\Lambda)}(\text{$0$ is frozen at time $1$}) < \ve.$$
\end{theorem}

\begin{proof}[Proof of Theorem \ref{thm:large_k_strong}]
Let $\Lambda$ and $K$ be as in the statement, and let $\Delta_0 = \Lambda_0 = \Lambda^{**}_0 = \Lambda$. Let us consider the various random sequences associated with the domain $\Lambda$ and the initial scale $K_0 = K$, as explained in Section \ref{sec:associated_chains}.
\begin{itemize}
\item $(\Lambda_i)_{0 \leq i \leq k}$ is the sequence of successive holes around $0$ for the frozen percolation process in $\Lambda$, with $(p_i)_{0 \leq i \leq k}$ the corresponding freezing times.

\item $(K_i)_{0 \leq i \leq k}$ is a deterministic sequence of scales.

\item $(\Lambda^{**}_i)_{0 \leq i \leq k}$ and $(p^{**}_i)_{0 \leq i \leq k}$ are two sequences, of random sets and random times, respectively, that are used to approximate the real process (i.e. the $\Lambda_i$'s and the $p_i$'s).

\item $(\Delta_i)_{0 \leq i \leq k}$ is obtained by successively applying Lemma \ref{lem:iteration_FP}. It is used (see below) to show that the $\Lambda^{**}_i$ indeed approximate the real process.
\end{itemize}
For the remainder of the proof, it is important to note that although the $\Lambda_i$'s ($i \geq 1$) are not stopping sets, the $\Delta_i$'s \emph{are}.

Lemma \ref{lem:apriori_hole} implies the existence of a constant $c > 1$ such that
\begin{equation} \label{eq:apriori_Lambda*k}
\PP \left( \Ball_{c^{-1} L(p^{**}_k)} \subseteq \Lambda^{**}_k \subseteq \Ball_{c L(p^{**}_k)} \right) > 1 - \ve.
\end{equation}
The deconcentration result for $L(p^{**}_k)$ (Corollary \ref{cor:dec_p*}) implies the following. For every $\kappa > 1$, we can find $k_0$ large enough so that: for all $k \geq k_0$, for all $N$ sufficiently large (depending on $k$), with probability $> 1 - \ve$,
\begin{equation}
L(p^{**}_k) \in \mathcal{I}_{\kappa} := \big( m_1(N), m_6(N) \big) \setminus \bigg( \bigcup_{i=1}^6 \big( \kappa^{-1} m_i(N), \kappa m_i(N) \big) \bigg)
\end{equation}
(i.e. $L(p^{**}_k)$ is between $m_1(N)$ and $m_6(N)$, but at least a factor $\kappa$ different from each of the exceptional scales $m_1(N), m_2(N), \ldots, m_6(N)$). Together with \eqref{eq:apriori_Lambda*k}, this implies (for the same $k$, and for all $N$ large enough):
\begin{equation} \label{eq:deconcentration_Lambda*k}
\PP \big( \exists L \in \mathcal{I}_{\kappa} \text{ such that } \Ball_{c^{-1} L} \subseteq \Lambda^{**}_k \subseteq \Ball_{c L} \big) > 1 - \ve.
\end{equation}

By applying Lemma \ref{lem:iteration_FP} $k$ times iteratively (which we can do since after each application, we obtain a new set $\Delta$ which is a stopping set: we can thus condition on it, and treat it as a deterministic set), we get that with high probability $\Lambda_k$ and $\Lambda^{**}_k$ are contained in, and close to, $\Delta_k$. Combining this with \eqref{eq:deconcentration_Lambda*k}, we conclude that for $\Delta_k$ the analog of \eqref{eq:deconcentration_Lambda*k} holds, with $1 - \ve$, $\kappa$ and $c$ replaced by $1 - 2 \ve$, $\frac{\kappa}{2}$ and $2 c$, respectively.

We are now in a position to use the results from \cite{BN15} about the behavior of frozen percolation in finite boxes, recalled in Section \ref{sec:exceptional_scales}. More precisely, using again that $\Delta_k$ is a stopping set, and recalling that w.h.p. $\Lambda_k$ is contained in, and close to, $\Delta_k$, we can apply Proposition \ref{prop:exceptional_scales2} (see also the reformulated version, just below it), corresponding to the case when we start between two consecutive exceptional scales $m_i$ and $m_{i+1}$ (with $1 \leq i \leq 5$) but far away from both. Indeed, $c$ is a universal constant, and we can take $\kappa$ as large as we want. This completes the proof.
\end{proof}

\section{Full-plane process} \label{sec:full_plane}

In order to obtain Theorem \ref{thm:full_plane} from Theorem \ref{thm:large_k_strong}, we have to show that, informally speaking, the hole around $0$ in the full-plane process satisfies (at some suitable random time) the conditions on the domains $\Lambda$ in Theorem \ref{thm:large_k_strong}. To do this (see Proposition \ref{prop:coupling_full_plane} below) turns out to be far from obvious: almost all of this section is dedicated to this.

\subsection{Connection with approximable domains}


The following notation is used repeatedly in this section. For $N \geq 1$ and $p \in (p_c,1]$ with $L(p) \geq \sqrt{N}$, we set $\widehat p = \widehat p (p,N) \in (p_c,1]$ to be the solution of
\begin{equation} \label{eq:p_hat}
L(p)^2 \theta(\widehat p) = N,
\end{equation}
and we extend this notation recursively by setting $\widehat p_1 = \widehat p$, and $\widehat p_k = \widehat{(\widehat p_{k-1})}$.

As we explained in the beginning of Section \ref{sec:exceptional_scales}, we expect $\widehat p$ to be roughly the time when the first frozen cluster appears for the frozen percolation process in $\Ball_{L(p)}$. Note that Lemma \ref{lem:nextN} can be rephrased in the following way (with $q_k$ as defined in \eqref{eq:def_qk}).

\begin{lemma} \label{lem:ratio_L}
There exist constants $c, \alpha > 0$ such that: for all $N \geq 1$ and $p\in (q_\infty, q_1]$,
$$\frac{L(p)}{L(\widehat p)}\geq c \left(\frac{m_\infty}{L(p)}\right)^\alpha.$$
\end{lemma}

For later use, we also note that there exist constants $\tilde{c}, \tilde{\alpha} > 0$ such that: for all $p > q_\infty$ for which $\widehat{\widehat{p}}$ is well-defined, we have
\begin{equation} \label{eq:L_phat_phathat}
\frac{L(\widehat{p})}{L(\widehat{\widehat{p}})} \leq \tilde{c} \left( \frac{L(p)}{L(\widehat{p})} \right)^{\tilde{\alpha}}.
\end{equation}
Indeed, we can write
$$\bigg( \frac{L(p)}{L(\widehat{p})} \bigg)^2 = \frac{\theta(\widehat{\widehat{p}})}{\theta(\widehat{p})} \geq c_1 \frac{\pi_1(L(\widehat{\widehat{p}}))}{\pi_1(L(\widehat{p}))} \geq c_2 \big( \pi_1(L(\widehat{\widehat{p}}), L(\widehat{p})) \big)^{-1} \geq c_3 \bigg( \frac{L(\widehat{p})}{L(\widehat{\widehat{p}})} \bigg)^{\alpha'},$$
using successively \eqref{eq:p_hat}, \eqref{eq:Kesten_theta_pi}, \eqref{eq:qmult} and \eqref{eq:1arm_bound}.

\begin{proposition} \label{prop:coupling_full_plane}
For all $\ve, \eta > 0$, there exist $c_2 > c_1 > \alpha > 0$, $M_0 > 0$ and $N_0 \geq 1$ such that: for all $N \geq N_0$ and $p \in (p_c, q_3(N))$ with $L(p) \leq m_\infty(N) / M_0$, there exists a simply connected stopping set $\Lambda^{\#}$ such that with probability at least $1 - \ve$, the following two properties hold. 
\begin{itemize}
\item[(i)] $\Lambda^{\#} = \holeT(p^{\#})$ for some $p^{\#} \in (p, \widehat{\widehat{\widehat{p}}})$.

\item[(ii)] $\Lambda^{\#}$ is $(\alpha L(p^{\#}), \eta)$-approximable, and $\Ball_{c_1 L(p^{\#})} \subseteq \Lambda^{\#} \subseteq \Ball_{c_2 L(p^{\#})}$.
\end{itemize}
\end{proposition}

Before proving this result in Section \ref{sec:proof_coupling_full_plane}, we explain how to combine it with Theorem \ref{thm:large_k_strong} and obtain Theorem \ref{thm:full_plane}.

\subsection{Proof of Theorem \ref{thm:full_plane} from Proposition \ref{prop:coupling_full_plane} and Theorem \ref{thm:large_k_strong}}

\begin{proof}[Proof of Theorem \ref{thm:full_plane}]
Let us consider some $\ve > 0$ arbitrary, and $\eta = \eta(\ve) > 0$ associated with it by Theorem \ref{thm:large_k_strong}. For this choice of $\ve$ and $\eta$, Proposition \ref{prop:coupling_full_plane} then produces $c_2 > c_1 > \alpha > 0$, $M_0 > 0$ and $N_0 \geq 1$. We know from Theorem \ref{thm:large_k_strong} that for these specific values $\alpha$, $c_1$ and $c_2$, we can find $k \geq 1$ and $N_1 \geq N_0$ large enough such that: for all $N \geq N_1$, all $K \in (m_{k+2}(N),m_{k+5}(N))$, and all simply connected $(\alpha K,\eta)$-approximable stopping sets $\Lambda$ with $\Ball_{c_1 K} \subseteq \Lambda \subseteq \Ball_{c_2 K}$, we have
\begin{equation} \label{eq:end_thm_full_plane}
 \PP_N^{(\Lambda)}(\text{$0$ is frozen at time $1$}) < \ve.
\end{equation}
In particular, for $p = q_{k+5}(N)$, we have $L(p) \leq m_\infty(N) / M_0$ for all $N \geq N_2$ (for some $N_2 \geq N_1$ large enough), so Proposition \ref{prop:coupling_full_plane} provides us with $\Lambda^{\#}$ and $p^{\#}$ which satisfy, with probability $> 1 - \ve$:
\begin{itemize}
\item[(i)] $\Lambda^{\#} = \holeT(p^{\#})$,

\item[(ii)] $K = L(p^{\#}) \in (m_{k+2}(N), m_{k+5}(N))$ (since $p^{\#} \in (p, \widehat{\widehat{\widehat{p}}})$),

\item[(iii)] and $\Lambda^{\#}$ is a simply connected $(\alpha K, \eta)$-approximable stopping set, with $\Ball_{c_1 K} \subseteq \Lambda \subseteq \Ball_{c_2 K}$.
\end{itemize}
Now, take $N \geq N_2$. For each pair $(\Lambda^{\#},K)$ satisfying (i), (ii) and (iii), we have from \eqref{eq:end_thm_full_plane} that
$$\PP_N^{(\Lambda^{\#})}(\text{$0$ is frozen at time $1$}) < \ve,$$
which we can then use for the full-plane process, since $\Lambda^{\#} = \holeT(p^{\#})$. We thus obtain
$$\PP_N^{(\TT)}(\text{$0$ is frozen at time $1$}) < 2 \ve,$$
completing the proof.
\end{proof}

\subsection{Proof of Proposition \ref{prop:coupling_full_plane}} \label{sec:proof_coupling_full_plane}

\begin{proof}[Proof of Proposition \ref{prop:coupling_full_plane}]
Let us consider some $\ve \in (0,1)$. We also consider some large constant $M > 0$, that we specify later, and $p > p_c$ such that $L(p) \leq m_\infty / M$. We use this control over $\frac{L(p)}{m_\infty}$ only via Lemma \ref{lem:ratio_L}, which implies that the ratio $L(p) / L(\widehat p)$ can be made arbitrarily large by choosing $M$ large enough.

By \eqref{eq:1arm_bound}, we can set $\mu = \mu(\ve) \in (0,1)$ small enough so that: for all $p > p_c$,
\begin{equation} \label{eq:existence_circuits}
\PP \left( E(p) \right) \geq 1 - \frac{\ve}{100}, \quad \text{where } E(p) := \big\{ \circuitevent^* (\mu L(p), L(p)) \cap \circuitevent (\mu^2 L(p), \mu L(p)) \text{ holds at time $p$} \big\}.
\end{equation}

\bigskip

\textbf{Step 1.} Let us fix some large $K \geq 1$ (we explain later how to choose it). We first prove that ``soon'' after time $p$ (we have to wait for at most one freezing), we can find a time $p^*$ at which the hole of the origin is large compared to the correlation length $L(p^*)$ at that time. Intuitively, if we take $q_1 < q_2$ so that $q_2$ is before $\widehat{q_1}$, but not ``much'' before, then a frozen circuit surrounding the origin cannot emerge close to both times $q_1$ and $q_2$. This implies that at time $q_1$ or $q_2$, the hole of the origin is large compared to the correlation length. In this step and the next one, we turn this intuition into a precise proof. 

We first introduce the outermost $p$-black circuit $\circuit$ in $\Ann_{\mu^2 L(p), \mu L(p)}$, taking $\circuit = \partial B_{\mu^2 L(p)}$ if such a circuit does not exist. Let $p'$ be the first time that a vertex on $\circuit$ freezes for the modified frozen percolation process in $\TT$ where clusters still freeze as soon as they reach volume $\geq N$, \emph{unless} they are included in $\inter(\circuit)$, in which case they keep growing as long as they do not contain a vertex of $\circuit$ (i.e. the clusters which are strictly inside $\circuit$ are allowed to grow after they reach volume $N$, until they intersect $\circuit$). Let us stress that this modified process is used only here (and not later in the proof). We use it to ensure that $p'$ has the right measurability property.

We define the (deterministic) times $p_1$ and $p_2$ by
\begin{equation} \label{eq:def_p1_p2}
L(p_1) = L(p)^{1/3} L(\widehat p)^{2/3} \quad \text{and} \quad L(p_2) = L(p)^{1/6} L(\widehat p)^{5/6}.
\end{equation}
Note that they satisfy $p < p_1 < p_2 < \widehat p$ (since $L(\widehat p) < L(p)$). In a similar way, we also introduce, for $\beta = \frac{1}{2 \tilde{\alpha}} > 0$ (where $\tilde{\alpha}$ is the universal constant from \eqref{eq:L_phat_phathat}), the times $p_3$ and $p_4$ such that
\begin{equation} \label{eq:def_p3_p4}
L(p_3) = L(\widehat p)^{1 - \beta} L(\widehat{\widehat{p}})^{\beta} \quad \text{and} \quad L(p_4) = L(\widehat p)^{1 - 2 \beta} L(\widehat{\widehat{p}})^{2 \beta},
\end{equation}
which satisfy $\widehat p < p_3 < p_4 < \widehat{\widehat p}$. For the convenience of the reader, we summarize on Figure \ref{fig:scales} the different scales that we use.

We now consider $\Lambda^* = \hole^{(\ol {\inter{\circuit}})}(p')$, where $\ol {\inter(\circuit)} = \circuit \cup \inter(\circuit)$. We distinguish two cases, depending on whether $p' \leq p_1$ or $p' > p_1$.
\begin{itemize}
\item Case a: if $p' \leq p_1$, we take $p^* = p_2$.

\item Case b: if $p' > p_1$, we take $p^* = p_4$.
\end{itemize}
In this way, we have produced a pair $(p^*, \Lambda^*)$ such that $\Lambda^*$ is a stopping set: we can thus condition on it, and treat it as a deterministic set.

\begin{figure}[t]
\begin{center}
\includegraphics[width=12cm]{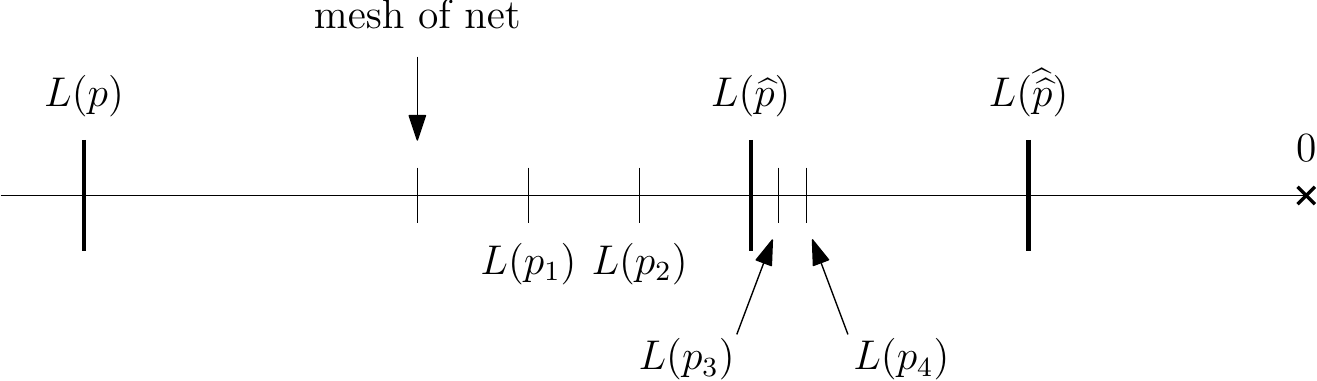}
\caption{\label{fig:scales} This figure presents, in a schematic way, the various scales involved in the proof of Proposition \ref{prop:coupling_full_plane} (Steps 1 and 2).}
\end{center}
\end{figure}

\bigskip

\textbf{Step 2.} We now prove that there exist $M$ and $N$ large enough such that with probability $\geq 1 - \frac{\ve}{10}$, the following three properties hold (recall that $M$ was introduced in the beginning of the proof, so that $L(p) \leq m_\infty / M$).
\begin{itemize}
\item[(i)] $\Ball_{K L(p^*)} \subseteq \Lambda^*$,

\item[(ii)] $\Lambda^* = \holeT(p^*)$,

\item[(iii)] and $\Lambda^* \cap \frozen(p^*) = \emptyset$ (i.e. $\Lambda^*$ does not contain any frozen cluster at time $p^*$).
\end{itemize}
Let us stress that in property (ii) above, the notation $\holeT$ refers to the original frozen percolation process (not the modified one).

First, let us note that
$$\PP_{p_2} \left( \big| \lclus_{\Ball_{L(p)}} \big| \geq N \right) = \PP_{p_2} \left( \big| \lclus_{\Ball_{L(p)}} \big| \geq x L(p)^2 \theta(p_2) \right),$$
with
$$x = \frac{N}{L(p)^2 \theta(p_2)} = \frac{\theta(\widehat p)}{\theta(p_2)} \geq c_1 \frac{\pi_1(L(\widehat p))}{\pi_1(L(p_2))} \geq c_2 \left( \frac{L(p)}{L(\widehat p)} \right)^{\alpha/6},$$
for some universal $\alpha > 0$ (using successively the definition of $\widehat p$ \eqref{eq:p_hat}, \eqref{eq:Kesten_theta_pi}, \eqref{eq:1arm_bound}, and the definition of $p_2$ \eqref{eq:def_p1_p2}). By Lemma \ref{lem:ratio_L}, this last lower bound can be made arbitrarily large by choosing $M$ large enough, so there exist constants $M_1 = M_1(\ve)$ and $N_1 = N_1(\ve)$ such that: for all $M \geq M_1$ and $N \geq N_1$, $x$ is large enough so that we can apply \eqref{eq:BCKS}, and obtain
\begin{equation} \label{eq:no_frozen}
\PP_{p_2} \left( \big| \lclus_{\Ball_{L(p)}} \big| \geq N \right) \leq c_3 e^{-c_4 \left( \frac{L(p)}{L(\widehat p)} \right)^{\alpha/6}} \leq \frac{\ve}{100}.
\end{equation}
We now assume that the event $E_1 := \{| \lclus_{\Ball_{L(p)}}(p_2) | < N \}$ occurs.

Let us assume that $E_2 := E(p)$ also holds, which has a probability at least $1 - \frac{\ve}{100}$ (using \eqref{eq:existence_circuits}). In particular, it implies that $\circuit$ exists. We claim that at time $p$, neither $\circuit$ nor anything inside it is frozen, with high probability. For that, let us set
$$E_3 := \big\{ \frozen(p) \cap \ol {\inter(\circuit)} = \emptyset \big\}.$$
Note that $\circuit$ is ``protected'' by the $p$-white circuit in $\Ann_{\mu L(p), L(p)}$ from frozen clusters outside it: we thus have
$$\PP_N(E_2 \cap E_3^c) \leq \PP_{p} \left( \big| \lclus_{\Ball_{L(p)}} \big| \geq N \right) \leq \frac{\ve}{100}$$
(using \eqref{eq:no_frozen}, since $p < p_2$).

It follows that $\PP_N(E_2 \cap E_3) \geq 1 - 2 \cdot \frac{\ve}{100}$: we now assume that this event holds, so that in particular $p' > p$, and we examine the two cases introduced earlier.
\begin{itemize}
\item \textbf{Case a:} $p < p' \leq p_1$ and $p^* = p_2$. First, we have
\begin{equation} \label{eq:K_p1_p2}
\mu^2 L(p) > L(p_1) > \mu L(p_1) > K L(p_2)
\end{equation}
for all $M \geq M_2 = M_2(\ve,K)$ and $N \geq N_2 = N_2(\ve,K)$ (using again Lemma \ref{lem:ratio_L}, and \eqref{eq:def_p1_p2}). Since we assumed that $E_1$ occurs, no cluster with volume $\geq N$ emerges before time $p_2$ in $\ol {\inter(\circuit)}$. Hence, $\circuit$ freezes at time $p'$ in the frozen percolation process (which coincides with the modified process). Moreover, let us assume that the event $E_4 := E(p_1)$ occurs (note that from \eqref{eq:existence_circuits}, $\PP(E_4) \geq 1 - \frac{\ve}{100}$). Since $p' \leq p_1$, the white circuit in $\Ann_{\mu L(p_1), L(p_1)}$ (from the definition of $E_4$) is also present at time $p'$. Hence, no vertex in $\Ball_{\mu L(p_1)}$ can freeze at time $p'$, and the freezing at time $p'$ leaves a hole $\holeT(p') \subseteq \Ball_{L(p)}$ in which no cluster with volume $\geq N$ emerges before time $p_2$. This implies that
$$\holeT(p_2) = \holeT(p') = \Lambda^* \supseteq \Ball_{\mu L(p_1)}$$
on the intersection of the events above. Using \eqref{eq:K_p1_p2}, we obtain that
$$\PP_N( \Ball_{K L(p^*)} \nsubseteq \holeT(p^*), \: p' \leq p_1) \leq 4 \cdot \frac{\ve}{100} $$
for all $M \geq \max(M_1, M_2)$ and $N \geq \max(N_1, N_2)$. We have thus checked properties (i), (ii) and (iii) in this case.

\item \textbf{Case b:} $p' > p_1$ and $p^* = p_4$. In this case, we use the intermediate scale
$$\lambda = L(p)^{1/2} L(\widehat p)^{1/2}$$
(which, intuitively, corresponds to a time strictly between $p$ and $p_1$), and the event
$$E_5 := \net_{p_1}( \lambda/4, L(p) )$$
(recall Definition \ref{def:net} for nets). We know from Lemma \ref{lem:net} that
$$\PP(E_5) \geq 1 - C_1 \bigg( \frac{L(p)}{\lambda/4} \bigg)^2 e^{- C_2 \frac{\lambda/4}{L(p_1)}} \geq 1 - C_3 \frac{L(p)}{L(\widehat p)} e^{- C_4 \big( \frac{L(p)}{L(\widehat p)} \big)^{1/6}},$$
for some universal constants $C_i > 0$ ($1 \leq i \leq 4$). Hence (using Lemma \ref{lem:ratio_L} once again), there exist $M_3 = M_3(\ve)$ and $N_3 = N_3(\ve)$ such that: for all $M \geq M_3$ and $N \geq N_3$, $\PP(E_5) \geq 1 - \frac{\ve}{100}$. In particular, it follows that with a probability $\geq 1- \frac{\ve}{100}$, there exists a $p_1$-black net inside $\ol {\inter(\circuit)}$ which is connected to $\circuit$, and which leaves holes with diameter $\leq \lambda$. Let us denote by $E_6$ the event that there exists such a net, and that in the time interval $(p_1,p_4]$ a cluster with volume $\geq N$ emerges which is not connected to $\circuit$. Because of the existence of a net at time $p_1$, any such cluster has to appear in one of the $k \leq C_5 \big( \frac{L(p)}{\lambda} \big)^2$ holes, each having a diameter $\leq \lambda$. We deduce
$$\PP(E_6) \leq C_5 \Big( \frac{L(p)}{\lambda} \Big)^2 \PP_{p_4} \left( \big| \lclus_{\Ball_{\lambda}} \big| \geq N \right),$$
which is $\leq \frac{\ve}{100}$ for all $M \geq M_4$ and $N \geq N_4$: indeed, we can proceed as for \eqref{eq:no_frozen}, as we explain now. For that, we write
$$\PP_{p_4} \left( \big| \lclus_{\Ball_{\lambda}} \big| \geq N \right) = \PP_{p_4} \left( \big| \lclus_{\Ball_{\lambda}} \big| \geq x \lambda^2 \theta(p_4) \right),$$
with
$$x = \frac{N}{\lambda^2 \theta(p_4)} = \frac{L(p)}{L(\widehat p)} \cdot \frac{\theta(\widehat p)}{\theta(p_4)},$$
and there exist universal constants $c_i > 0$ ($1 \leq i \leq 4$) such that
$$\frac{\theta(\widehat p)}{\theta(p_4)} \geq c_1 \frac{\pi_1(L(\widehat p))}{\pi_1(L(p_4))} \geq c_2 \pi_1(L(p_4), L(\widehat p)) \geq c_3 \bigg( \frac{L(p_4)}{L(\widehat{p})} \bigg)^{1/2} = c_3 \bigg( \frac{L(\widehat{\widehat{p}})}{L(\widehat{p})} \bigg)^{\beta}$$
(using \eqref{eq:Kesten_theta_pi}, \eqref{eq:qmult}, \eqref{eq:1arm_bound} and the definition of $p_4$ \eqref{eq:def_p3_p4}), which yields
$$x \geq c_4 \frac{L(p)}{L(\widehat p)} \cdot \bigg( \frac{L(p)}{L(\widehat{p})} \bigg)^{- \beta \tilde{\alpha}} = c_4 \bigg( \frac{L(p)}{L(\widehat{p})} \bigg)^{1/2}$$
(this follows from \eqref{eq:L_phat_phathat} and our particular choice of $\beta$). We are thus in a position to combine Lemmas \ref{lem:ratio_L} and \ref{lem:BCKS}.

On the other hand, with high probability, something has to freeze before time $p_3$ in $\ol {\inter(\circuit)}$. Indeed, we know from Lemma \ref{lem:largest_vol} that
$$\PP_{p_3} \left( \big| \lclus_{\Ball_{\mu^2 L(p)}} \big| \geq \Big( 1 - \frac{\ve}{100} \Big) \theta(p_3) \big| \Ball_{\mu^2 L(p)} \big| \right) \geq 1 - \frac{\ve}{100}$$
as soon as $\frac{L(p_3)}{\mu^2 L(p)}$ is small enough, and we have
$$\frac{\big( 1 - \frac{\ve}{100} \big) \theta(p_3) \big| \Ball_{\mu^2 L(p)} \big|}{N} \geq C_6 \frac{\theta(p_3) \big( \mu^2 L(p) \big)^2}{N} = C_6 \mu^4 \frac{\theta(p_3)}{\theta(\widehat p)}$$
for some universal constant $C_6 > 0$ (using the definition of $\widehat p$ \eqref{eq:p_hat}), which is $\geq 1$ for all $M \geq M_5$ and $N \geq N_5$ (thanks to Lemma \ref{lem:ratio_L} again). Hence, the only possible scenario is as follows: the connected component that contains $\circuit$ and the net at time $p_1$ freezes at time $p' \leq p_3$, and when it freezes, it leaves holes in which no other clusters with volume $\geq N$ emerge before time $p_4$. In particular, $\Lambda^* = \holeT(p') = \holeT(p^*)$. We can then conclude the claims announced in the beginning of Step 2 by using the event $E_7 = E(p_3)$, which has a probability $\geq 1 - \frac{\ve}{100}$, and ensures that $\Ball_{\mu L(p_3)} \subseteq \Lambda^*$: indeed, $\mu L(p_3) > K L(p_4)$ for all $M \geq M_6$ and $N \geq N_6$.
\end{itemize}

\bigskip

\textbf{Step 3.} We now use the big hole around $0$ at time $p^*$. We consider
\begin{itemize}
\item $\tilde{p}^+ := \inf \{t \geq p^* \: : \:$ there exists a $t$-black cluster in $\Lambda^*$ which has $\geq N$ vertices, intersects $\partial B_{K L(p^*) / 2}$, and contains a circuit surrounding $0$ that is included in $B_{K L(p^*) / 2}\}$,

\item and $\Lambda^+ := \hole^{(B_{K L(p^*) / 2})}(\tilde{p}^+)$ (i.e. $\Lambda^+$ is the hole of the origin in the cluster from the definition of $\tilde{p}^+$).
\end{itemize}
By construction, $\Lambda^+$ is a stopping set. We have to prove that it has the properties in the statement of Proposition \ref{prop:coupling_full_plane}. Throughout the proof, we use the intermediate scale
\begin{equation} \label{eq:def_gamma}
\gamma := \sqrt{K} L(p^*).
\end{equation}
If we set
$$E'_1 := \net_{p^*}( \gamma/4, K L(p^*) ),$$
Lemma \ref{lem:net} implies that 
$$\PP(E'_1) \geq 1 - C_1 K e^{- C_2 \sqrt{K}}$$
for some suitable universal constants $C_1, C_2> 0$. In particular, there exists a constant $K_1 = K_1(\ve)$ such that for all $K \geq K_1$, this lower bound is at least $1 - \frac{\ve}{100}$. We now assume that this event $E'_1$ holds.

For each $x \in \grid_K$, there exists a $p^*$-black circuit in $\Ann_{\gamma/2, \gamma}(2 \gamma x)$, where (with $\ZZ[i] = \ZZ + \ZZ i$)
$$\grid_K := \ZZ[i] \cap \Ball_{KL(p^*) / 4\gamma} = \ZZ[i] \cap \Ball_{\sqrt{K}/4}.$$
Let $\circuit^x$ denote the outermost such circuit. Note that all these circuits are connected by $p^*$-black paths (inside the net), as illustrated on Figure \ref{fig:construction_big_hole}.

\begin{figure}[t]
\begin{center}
\includegraphics[width=13cm]{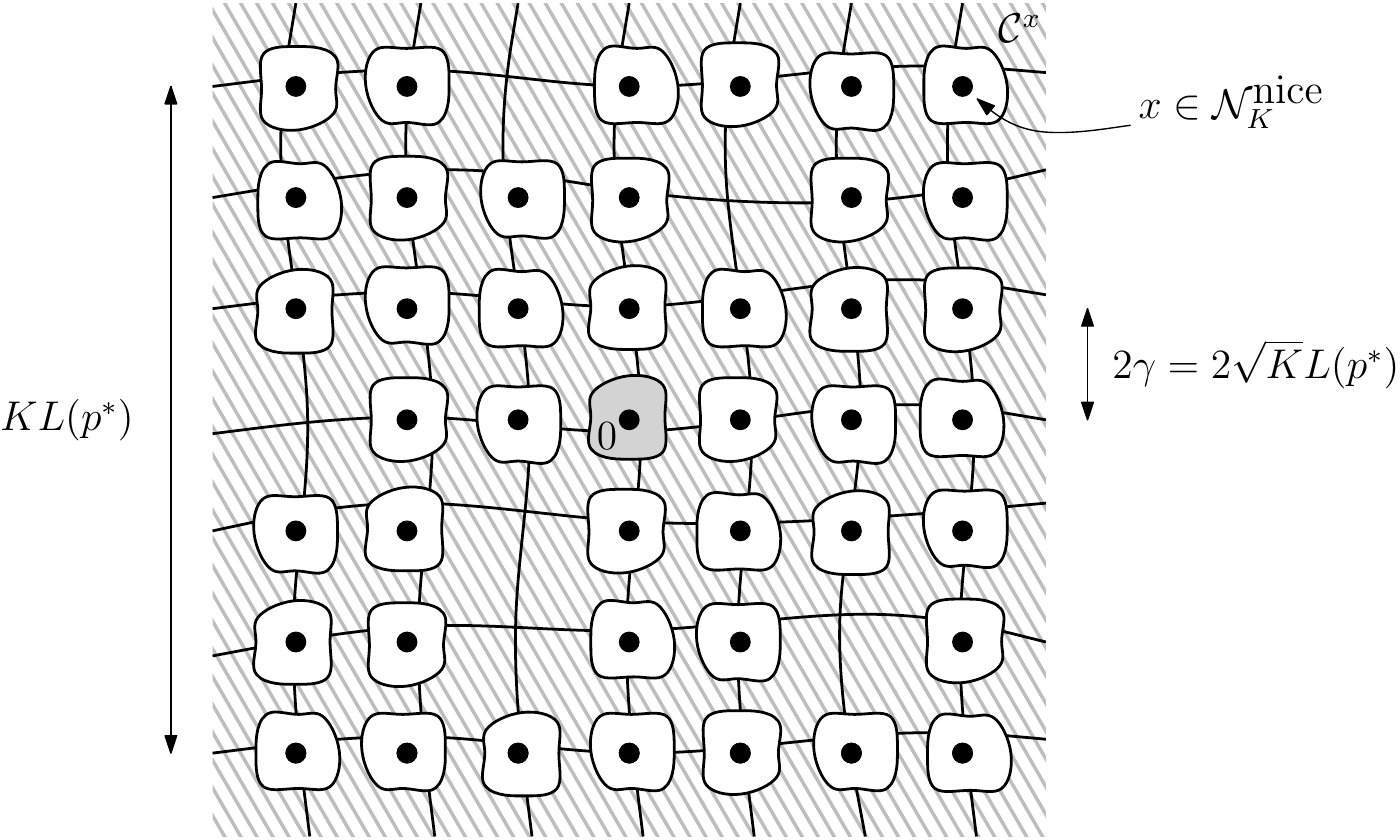}
\caption{\label{fig:construction_big_hole} This figure depicts the construction used in the proof of Proposition \ref{prop:coupling_full_plane}. We consider the independent random variables $X^x_t = X^{\circuit^x}_t$, for $x \in \gridn_K$ (i.e. such that the corresponding circuit $\circuit^x$ is nice), and we condition on the configuration outside: $Y_t$ counts the number of vertices connected to at least one $(\circuit^x)_{x \in \gridn_K}$ (including the vertices of the circuits themselves).}
\end{center}
\end{figure}

For $x \in \grid_K$ and $t \geq p^*$, we consider (recall the notation from Section \ref{sec:nice_circuits})
$$X^x_t = X^{\circuit^x}_t = \big| \{ v \in \inter(\circuit^x) \: : \: v \lra{t} \circuit^x \} \big|.$$
In the following, we define several random times in terms of $(X^x_t)_{x \in \grid_K}$, we thus restrict ourselves to the $x$ for which we have a good control on the quantiles. More precisely, we know from Lemma \ref{lem:nice_circ} that there exists a constant $C_3 > 0$ such that for each $x \in \grid_K$,
\begin{equation} \label{eq:calc_nice}
\PP \big( \circuit^x \text{ is not } (p^*, C_3) \text{-nice}, \: E'_1 \big) \leq \frac{\ve}{100}.
\end{equation}
Note that the events $\big\{ \circuit^x \text{ exists, and it is } (p^*, C_3) \text{-nice} \big\}$ are independent, for $x \in \grid_K$. We define the set
$$\gridn_K := \{ x \in \grid_K \: : \: \circuit^x \text{ is } (p^*, C_3) \text{-nice} \}.$$
Further, we write
$$E'_2 := \bigg\{ 0 \in \gridn_K, \: |\gridn_K| \geq \frac{K}{8} \bigg\} \quad \text{and} \quad E' := E'_1 \cap E'_2.$$
From Hoeffding's inequality and \eqref{eq:calc_nice}, there exists $K_2 = K_2(\ve)$ such that: for all $K \geq K_2$,
$$\PP(E'_1 \cap (E'_2)^c) \leq \frac{\ve}{100}.$$

Finally, for $t \geq p^*$, let $Y_t$ denote the number of vertices in $\Lambda^*$ which are either on one of the circuits $(\circuit^x)_{x \in \gridn_K}$, or outside these circuits and connected to at least one of them. More precisely,
$$Y_t := \bigg| \bigg\{ v \in \Lambda^* \setminus \bigcup_{x\in \gridn_K} \inter(\circuit^x) \: : \: v \lra{t} \bigcup_{x\in \gridn_K} \circuit^x \bigg\} \bigg|.$$
We set
\begin{equation} \label{eq:def_p+_new}
p^+ := \inf \bigg\{ t \geq p^* \: : \: \sum_{x \in \gridn_K} X^x_t + Y_t \geq N \bigg\}.
\end{equation}
We also define the random times
\begin{align}
\ul p^+ & := \inf \bigg\{ t \geq p^* \: : \: \ol Q_{\ve/100}(X^0_t) + \sum_{x \in \gridn_K \setminus \{0\}}  X^x_t + Y_t \geq N \bigg\} \label{eq:def_pinf+}\\[1mm]
\text{and} \quad \ol p^+ & := \inf \bigg\{ t \geq p^* \: : \: \ul Q_{\ve/100}(X^0_t) + \sum_{x \in \gridn_K \setminus \{0\}}  X^x_t + Y_t \geq N \bigg\}, \label{eq:def_psup+}
\end{align}
where we ``isolate'' $X^0_t$ by considering its quantiles (recall the notation for quantiles from Section \ref{sec:nice_circuits}). Further, let
$$\calS_0 := V(\TT) \setminus \inter(\circuit^0) \quad \text{and} \quad \calS_1 := V(\TT) \setminus \bigcup_{x\in \gridn_K} \inter(\circuit^x).$$
For later use, we note that by definition,
\begin{itemize}
\item $(X^x_t)_{x \in \gridn_K \setminus \{0\}}$ are measurable functions of $(\tau_v)_{v \in \calS_0 \setminus \calS_1}$,

\item $Y_t$ is a measurable function of $(\tau_v)_{v \in \calS_1}$,

\item and $\ul p^+$, $\ol p^+$ are measurable functions of $(\tau_v)_{v \in \calS_0}$.
\end{itemize} 
We condition on $\calS_1$ and $(\tau_v)_{v \in \calS_1}$ from now on. Under this conditioning, the function $Y_t$ becomes deterministic, while the processes $(X^x_t)_{t \geq p}$ are independent for  $x \in \gridn_K$.

\bigskip

\textbf{Step 4.} We prove that (with large probability) $p^* \leq \ul p^+ \leq p^+ \leq \ol p^+ \leq \widehat{p^*}$. For that, let us first introduce two rough bounds on $p^+$: we set
\begin{align*}
\ul p & := \inf \Bigg\{ t \geq p^* \: : \: \ol Q_{\ve/100}(X^0_t) + \ol Q_{\ve/100} \Bigg( \sum_{x \in \gridn_K \setminus \{0\}} X^x_t \Bigg) + Y_t \geq N \Bigg\},\\
\text{and} \quad \ol p & := \inf \Bigg\{ t \geq p^* \: : \: \ul Q_{\ve/100}(X^0_t) + \ul Q_{\ve/100} \Bigg( \sum_{x \in \gridn_K \setminus \{0\}} X^x_t \Bigg) + Y_t \geq N \Bigg\}.
\end{align*}
Note that it follows from the definitions that $\ul p$, $\ol p$ are measurable functions of $(\tau_v)_{v \in \calS_1}$.

It is clear that $p^* \leq \ul p$. We prove that $\ul p \leq \ul p^+ \leq p^+ \leq \ol p^+ \leq \ol p$ with probability at least $1 - \ve/5$, and we then show $\ol p \leq \widehat{p^*}$ separately. As to the first four inequalities, we only prove that $p^+ \leq \ol p^+$ with probability at least $1 - \ve/20$; the other inequalities can be established in a similar way. For that, note that if $p^+ > \ol p^+$, then
$$\ul Q_{\ve/100}(X^0_{\ol p^+}) + \sum_{x \in \gridn_K \setminus \{0\}}  X^x_{\ol p^+} + Y_{\ol p^+}\geq N > X^0_{\ol p^+} + \sum_{x \in \gridn_K \setminus \{0\}}  X^x_{\ol p^+} + Y_{\ol p^+},$$
and hence, $\ul Q_{\ve/100}(X^0_{\ol p^+}) > X^0_{\ol p^+}$. Since the process $(X^0_t)_{t \geq p^*}$ is conditionally independent of the process $(\sum_{x \in \gridn_K \setminus \{0\}}  X^x_t + Y_t)_{t \geq p^*}$, and thus of $\ol p^+$, the above has a probability at most $\ve/100$.

We now prove $\ol p \leq \widehat{p^*}$ (with large probability). For that, we define $\widehat{p}^{(K)}$ by
\begin{equation} \label{eq:def_phat_K}
K^{3/2} L(p^*)^2 \theta(\widehat{p}^{(K)}) = N.
\end{equation}
For $K \geq 1$, the monotonicity of $\theta$ implies that $\widehat{p}^{(K)} \leq \widehat{p^*}$; it is thus enough to prove that
\begin{equation} \label{eq:p_phat_K}
\text{for all } K \text{ large enough,} \quad \ol p \leq \widehat{p}^{(K)} \text{ with probability } \geq 1 - \frac{\ve}{20},
\end{equation}
which we do now (this slightly stronger result is used in the next step). Let us also note that for some constants $M_7$ and $N_7$ depending only on $K$, we have: if $N \geq N_7$ and $L(p) \leq m_\infty(N) / M_7$, then
\begin{equation} \label{eq:p_phat_K2}
\widehat{p}^{(K)} \geq p^*
\end{equation}
(indeed, for every fixed $K$, it follows from \eqref{eq:p_hat} and \eqref{eq:def_phat_K} that $\theta(\widehat{p^*}) \asymp \theta(\widehat{p}^{(K)})$, so $L(\widehat{p^*}) \asymp L(\widehat{p}^{(K)})$, by \eqref{eq:Kesten_theta_pi} and \eqref{eq:1arm_bound}, and we can use Lemma \ref{lem:ratio_L}).

Recall that $C_3$ was chosen according to Lemma \ref{lem:nice_circ}, and $\circuit^0$ is $(p^*, C_3)$-nice on $E'$. Since $\circuit^0 \subseteq \Ann_{\gamma/2, \gamma}$, with $\gamma = \sqrt{K} L(p^*)$, we obtain from Lemma \ref{lem:nice_circ2} that for some $\ul c_3 > 0$,
$$\ul Q_{\ve/100}(X^0_{\widehat{p}^{(K)}}) \geq \ul c_3 (\sqrt{K} L(p^*))^2 \theta(\widehat{p}^{(K)}),$$
and similarly,
$$\ul Q_{\ve/100} \bigg( \sum_{x \in \gridn_K \setminus \{0\}} X^x_{\widehat{p}^{(K)}} \bigg) \geq \ul c_3 |\gridn_K \setminus \{0\}| (\sqrt{K} L(p^*))^2 \theta(\widehat{p}^{(K)}).$$
Since $|\gridn_K| \geq \frac{K}{8}$ on the event $E'$, we deduce
$$\ul Q_{\ve/100} (X^0_{\widehat{p}^{(K)}}) + \ul Q_{\ve/100} \bigg( \sum_{x \in \gridn_K \setminus \{0\}} X^x_{\widehat{p}^{(K)}} \bigg) + Y_{\widehat{p}^{(K)}} \geq \frac{\ul c_3}{8} K^2 L(p^*)^2 \theta(\widehat{p}^{(K)}) \geq \frac{\ul c_3}{8} K^{1/2} N$$
(using \eqref{eq:def_phat_K}), which is $\geq N$ for all $K \geq K_3 = K_3(\ve)$. Hence, we get that for all $K \geq K_3$, $\ol p \leq \widehat{p}^{(K)}$, which completes the proof of \eqref{eq:p_phat_K}, and thus of Step 4.

\bigskip

\textbf{Step 5.} Recall the definitions of $\gamma$, $\tilde{p}^+$ and $\Lambda^+$ in the beginning of Step 3. We show that with high probability,
\begin{itemize}
\item[(i)] $\tilde{p}^+ = p^+$ (i.e. $\tilde{p}^+$ is the time when the structure consisting of the circuits $(\circuit^x)_{x \in \gridn_K}$ freezes), 

\item[(ii)] $\Lambda^+ = \holeT(p^+)$,

\item[(iii)] and $\hole^{(B_{\gamma / 8})}(\ol p^+) \subseteq \Lambda^+ \subseteq \hole^{(B_{\gamma / 8})}(\ul p^+)$.
\end{itemize}

Recall from Step 4 that for all $K$ large enough, $p^+ \leq \widehat{p}^{(K)}$ (with large probability). In a similar way as in Step 2 (Case b), we see that if, apart from the net from Step 3, no other cluster in $\Lambda^*$ intersecting $\Ball_{K L(p^*) / 2}$ reaches volume $N$ before time $\widehat{p}^{(K)}$, then $\tilde{p}^+ = p^+$. Hence,
\begin{equation} \label{eq:upper_bd_p_ptilde}
\PP_N \big( \tilde{p}^+ \neq p^+, \: p^+ \leq \widehat{p}^{(K)}, \: E' \big) \leq K \cdot \PP_{\widehat{p}^{(K)}} \Big( \big| \lclus_{\Ball_{\gamma}} \big| \geq N \Big).
\end{equation}
Using \eqref{eq:BCKS} (with $x = K^{1/2}$) and \eqref{eq:def_phat_K}, the probability in the right-hand side can be bounded as follows:
$$\PP_{\widehat{p}^{(K)}} \Big( \big| \lclus_{\Ball_{\gamma}} \big| \geq N \Big) \leq c_1 e^{-c_2 K^{1/2} \frac{\gamma^2}{L(\widehat{p}^{(K)})^2} } = c_1 e^{-c_2 K^{3/2} \frac{L(p^*)^2}{L(\widehat{p}^{(K)})^2} } \leq c_1 e^{-c_2 K^{3/2}},$$
since $L(\widehat{p}^{(K)}) \leq L(p^*)$ (from \eqref{eq:p_phat_K2}). The upper bound in \eqref{eq:upper_bd_p_ptilde} is thus $\leq \frac{\ve}{100}$ for all $K \geq K_4(\ve)$, which shows properties (i) and (ii). Since $\ul p^+ \leq p^+ \leq \ol p^+$, we also have (with large probability)
$$\hole^{(B_{2 \gamma})}(\ol p^+) \subseteq \Lambda^+ \subseteq \hole^{(B_{2 \gamma})}(\ul p^+).$$
Further, let $E'_3:= \big\{ \text{there is a } p^* \text{-black circuit } \circuit \text{ in } \Ann_{\sqrt[4]{K} L(p^*), \sqrt{K} L(p^*)} \text{ s.t. } \circuit \lra{p^*} \infty \big\}$. We have that for all $K \geq K_5 = K_5(\ve)$, $\PP(E'_3) \geq 1 - \frac{\ve}{100}$ (from \eqref{eq:exp_decay}), and $E'_3$ implies in particular that, for all $t \geq p^*$, $\hole^{(B_{2 \gamma})}(t) = \hole^{(B_{\gamma / 8})}(t)$. By using this observation at times $\ol p^+$ and $\ul p^+$, we finally get property (iii).

We are now almost in a position to conclude the proof of Proposition \ref{prop:coupling_full_plane}. Indeed, note that since $\ol p^+$ and $\ul p^+$ are measurable functions of $(\tau_v)_{v \in \TT \setminus \Ball_{\gamma/2}}$, we can apply Lemma \ref{lem:cont_vol} (and Remark \ref{rem:H_HLambda}) to $\hole^{(B_{\gamma / 8})}(\ol p^+)$ and $\hole^{(B_{\gamma / 8})}(\ul p^+)$ to deduce that $\Lambda^+$ has the desired properties, if we prove that $L(\ul p^+) / L(\ol p^+)$ can be made arbitrarily close to $1$. This will be done in the next (and last) step.

\bigskip

\textbf{Step 6.} We now fix an arbitrary $\delta > 0$, and we bound the probability of $\big\{ \frac{L(\ul p^+)}{L(\ol p^+)}>1+\delta \big\}$. For that, we first show that, for the rough lower and upper bounds $\ul p$ and $\ol p$, $L(\ul p)$ and $L(\ol p)$ are comparable. It follows from the definitions of $\ul p$ and $\ol p$ that
\begin{align}
\lim_{t \nearrow \ol p} \Bigg( \ul Q_{\ve/100} & (X^0_t) + \ul Q_{\ve/100} \bigg( \sum_{x \in \gridn_K \setminus \{0\}} X^x_t \bigg) + Y_t \Bigg) \leq N \nonumber\\
& \leq \lim_{t \searrow \ul p} \Bigg( \ol Q_{\ve/100}(X^0_t) + \ol Q_{\ve/100} \bigg( \sum_{x \in \gridn_K \setminus \{0\}} X^x_t \bigg) + Y_t \Bigg). \label{eq:pf_prop_coupling_to_MC}
\end{align}
From the same reasoning as in the end of Step 4, we obtain that in the left-hand side of \eqref{eq:pf_prop_coupling_to_MC},
$$\lim_{t \nearrow \ol p} \Bigg( \ul Q_{\ve/100}(X^0_t) + \ul Q_{\ve/100} \bigg( \sum_{x \in \gridn_K \setminus \{0\}} X^x_t \bigg) \Bigg) \geq \ul c_3 |\gridn_K| K L(p^*)^2 \theta(\ol p)$$
(using the continuity of $\theta$ at $\ol p$), and a similar \emph{upper} bound holds for the right-hand side of \eqref{eq:pf_prop_coupling_to_MC}, with $\ul c_3$ replaced by $\ol c_3$. These bounds, combined with \eqref{eq:pf_prop_coupling_to_MC}, show that $\ul c_3 \theta(\ol p) \leq \ol c_3 \theta(\ul p)$. By \eqref{eq:Kesten_theta_pi}, this shows the existence of a constant $C_4 = C_4(\ve)$ such that
\begin{equation} \label{eq:pf_prop_couplint_to_MC_1}
L(\ul p) \leq C_4L(\ol p).
\end{equation}

For each $i \in \{0, \ldots, 2 \log_{1+\delta} C_4\}$, let $t_i$ be defined by
$$L(t_i) = (1+\delta)^{-i/2}L(\ul p).$$
Note that if $\frac{L(\ul p^+)}{L(\ol p^+)} > 1 + \delta$, then there exists an $i \in \mathcal{I}_{\delta} := \{0, \ldots, 2 \log_{1+\delta} C_4 - 1\}$ for which $\ul p^+ < t_i$ and $t_{i+1} < \ol p^+$ (using the rough bound $\ul p \leq \ul p^+ \leq \ol p^+ \leq \ol p$). For this $i$, the definitions of $\ul p^+$ \eqref{eq:def_pinf+} and $\ol p^+$ \eqref{eq:def_psup+} imply
$$\ol Q_{\ve/100}(X^0_{t_i}) + \sum_{x \in \gridn_K \setminus \{0\}} X^x_{t_i} + Y_{t_i} \geq N > \ul Q_{\ve/100}(X^0_{t_{i+1}}) + \sum_{x \in \gridn_K \setminus \{0\}} X^x_{t_{i+1}} + Y_{t_{i+1}},$$
from which we deduce
\begin{align}
\sum_{x \in \gridn_K \setminus \{0\}} \big( X^x_{t_{i+1}} - X^x_{t_i} \big) & \leq \big( Y_{t_i} - Y_{t_{i+1}} \big) + \big( \ol Q_{\ve/100}(X^0_{t_i}) - \ul Q_{\ve/100}(X^0_{t_{i+1}}) \big) \nonumber\\[-3mm]
& \leq \ol Q_{\ve/100}(X^0_{\ol p}) \leq C_5 K L(p^*)^2 \theta(\ol p), \label{eq:upper_sum_nice}
\end{align}
for some $C_5 = C_5(\ve)>0$.

It turns out to be easier to work with a slightly different collection of random variables. We set
$$Z_i^x := \big| \{ v \in \Ball_{\gamma/10}(2 \gamma x) \: : \: v \lra{t_{i+1}} \circuit^x, \: v \nlra{t_{i}} \partial \Ball_{\gamma/10}(v) \} \big| \leq X^x_{t_{i+1}}- X^x_{t_i}.$$
It follows from \eqref{eq:upper_sum_nice} that it is enough to bound, for each $i \in \mathcal{I}_{\delta}$, the probability of the event that
$$\sum_{x \in \gridn_K \setminus \{0\}} Z_i^x \leq C_{5} K L(p^*)^2 \theta(\ol p).$$

For this purpose, first note that
\begin{align}
\EE \big[ Z_i^x \: | \: \circuit^x \big] & \geq \sum_{v \in \Ball_{\gamma / 10}(2 \gamma x)} \PP \big( v \lra{t_{i+1}} \circuit^x, \: v \nlra{t_i} \partial \Ball_{\gamma / 10}(v) \: | \: \circuit^x \big) \nonumber\\
& \geq \frac{1}{100} K L(p^*)^2 \PP \big( 0 \lra{t_{i+1}} \infty, \: 0 \nlra{t_i} \partial \Ball_{\gamma / 10} \big) \nonumber\\
& \geq C_6 K L(p^*)^2 \frac{|t_{i+1}-t_i|}{|t_i-p_c|} \theta(t_i) \label{eq:pf_coupling_to_MC_1}\\
& \geq C_7 K L(p^*)^2 \theta(\ol p), \label{eq:pf_coupling_to_MC_2}
\end{align}
for some suitable universal constants $C_6$ and $C_7 = C_7(\delta)$ (using Lemma \ref{lem:diff_theta_like} in \eqref{eq:pf_coupling_to_MC_1}, and Lemma \ref{lem:p-L(p)} in \eqref{eq:pf_coupling_to_MC_2}, combined with the definition of $t_i$).

Let us fix $i \in \mathcal{I}_{\delta}$, and consider $(Z^x_i)_{x \in \gridn_K \setminus \{0\}}$. Since $Z^x_i \leq \calV_{\gamma/10}(2 \gamma x)$ (recall the definition of $\calV_n$ in \eqref{eq:def_Vn}), Lemma \ref{lem:moment_bd} provides the moment bound
$$\EE \big[ (Z_i^x)^m \big] \leq \EE \big[ ( \calV_{\gamma/10}(2 \gamma x) )^m \big] \leq C_{8}^m m! (\gamma^2 \theta(t_{i+1}))^m \leq C_{8}^m m! (K L(p^*)^2 \theta(\ol p))^m$$
for some universal constant $C_{8}>0$. This shows that we can apply Bernstein's inequality (Lemma \ref{lem:Bernstein}) to the centered random variables $(Z_i^x - \EE[Z_i^x])_{x \in \gridn_K \setminus \{0\}}$, with
$$n = \big| \gridn_K \setminus \{0\} \big| \asymp K, \: M = C_{9} K L(p^*)^2 \theta(\ol p), \: \sigma_x^2 = M^2, \: \text{and } y = C_5 K L(p^*)^2 \theta(\ol p) - \EE \Bigg[ \sum_{x \in \gridn_K \setminus \{0\}} Z^x_i \Bigg],$$
for some constant $C_9 = C_9(\delta)$ large enough. Note that $|y| \asymp K M$ (since $\EE[Z^x_i] \asymp M$, by \eqref{eq:pf_coupling_to_MC_2}). We obtain that, for each $i \in \mathcal{I}_{\delta}$,
$$\PP \bigg( \sum_{x \in \gridn_K \setminus \{0\}} Z_i^x \leq C_5 K L(p^*)^2 \theta(\ol p) \bigg) \leq 2 e^{-C_{10} K}$$
for some $C_{10} = C_{10}(\delta)$. So, putting things together,
$$\PP \bigg( \frac{L(\ul p^+)}{L(\ol p^+)}>1+\delta \bigg) \leq \PP \bigg(\exists i \in \mathcal{I}_{\delta} \: : \: \sum_{x \in \gridn_K \setminus \{0\}} Z_i^x \leq C_5 K L(p^*)^2 \theta(\ol p) \bigg) \leq 4 (\log_{1+\delta} C_4) e^{-C_{10} K}$$
(using $|\mathcal{I}_{\delta}| \leq 2 \log_{1+\delta} C_4$), which is $\leq \ve/100$ for all $K\geq K_6(\ve)$. Hence, if we set $K = \max_{1 \leq i \leq 6} K_i$, and then $M = \max_{1 \leq i \leq 7} M_i(K, \ve)$ and $N = \max_{1 \leq i \leq 7} N_i(K, \ve)$, all the desired bounds hold, which completes the proof of Proposition \ref{prop:coupling_full_plane}.

\end{proof}

\subsection{Concluding remark: related processes} \label{sec:related_proc}

In this last section, we briefly and informally indicate the robustness of our methods, by considering some other interesting models for which a similar behavior as for volume-frozen percolation, and analogs of Theorems \ref{thm:full_plane} and \ref{thm:large_k}, can be expected. We discuss in particular two closely related processes (the proof of existence requires substantial work: see \cite{Duerre_EJP}). For these two processes, all vertices are initially white, and they can turn black according to some Poisson process of ``births'', with intensity $1$. We also use a second, independent, Poisson process of ``lightnings'', with a small rate $\ve_N > 0$: each vertex is hit by lightning at a rate $\ve_N$, independently of the other vertices. To fix ideas, let us take $\ve_N = N^{-\alpha}$, for some $\alpha > 0$.

We can first introduce a modified volume-frozen percolation process, where a black connected component freezes when one of its vertices is hit by lightning (so that the rate at which a cluster freezes is proportional to its volume). As a starting point, we can look for a similar separation of scales as in our previous volume-frozen percolation process. Here and further in this section, we make the usual translation $p(t) = 1 - e^{-t}$, and we define $t_c$ as the solution of $p(t_c) = p_c$. We also write $L(t)$ for $L(p(t))$, and similarly for $\theta(t)$.

Heuristically, the recursion formula \eqref{eq:p_hat} should be replaced by
$$\ve_N |\widehat t - t_c| L(t)^2 \theta(\widehat t) \asymp 1,$$
where $L(t)^2 \theta(\widehat t)$ corresponds to the volume of the ``giant'' connected component in a box with side length $L(t)$, and $\ve_N |\widehat t - t|$ is the probability for any given vertex to be hit between times $t$ and $\widehat t$, which we replace by $\ve_N |\widehat t - t_c|$ (since we look for the property $|\widehat t - t_c| \gg |t - t_c|$).

A quick computation then yields a sequence of exceptional scales
$$m^{(\alpha)}_k(N) = N^{\delta^{(\alpha)}_k + o(1)} \quad \text{as $N \to \infty$}$$
(and corresponding times $q^{(\alpha)}_k(N) = t_c + N^{-\frac{3}{4} \delta^{(\alpha)}_k + o(1)}$), where the sequence of exponents $(\delta^{(\alpha)}_k)_{k \geq 0}$ satisfies
$$\delta^{(\alpha)}_0 = 0, \quad \text{and } \: \delta^{(\alpha)}_{k+1} = \frac{\alpha}{2} + \frac{41}{96} \delta^{(\alpha)}_k \:\:\: (k \geq 0).$$
This sequence is strictly increasing, and it converges to $\delta^{(\alpha)}_{\infty} = \frac{48}{55} \alpha$. We then have a separation of scales, i.e. $L(\widehat t) \gg L(t)$, for all $t > t_c$ such that $L(t) \ll m_\infty^{(\alpha)}(N)$, as in Lemma \ref{lem:ratio_L} (where $m_\infty^{(\alpha)}(N) = N^{\delta^{(\alpha)}_\infty + o(1)}$ as $N \to \infty$). If we consider a net with mesh $\gg L(\widehat t)$ and $\ll L(t)$ at an intermediate time between $t$ and $\widehat t$ (as we did, for instance, in Step 2 of the proof of Proposition \ref{prop:coupling_full_plane}), we see that the next freezing time coincides with the freezing time of this net (w.h.p.). In particular, the next hole looks like $\hole(t^\#)$, for some random $t^\#$ such that $|t^\# - t_c|$ is comparable to $|\widehat t - t_c|$.

We can also consider the forest fire process obtained from the same Poisson processes (of births and of lightnings), where a black connected component ``burns'', i.e. all its vertices become white, when one of its vertices is hit (and may later become black again according to the Poisson process of births). During a first non-trivial stage of the process (immediately after $t_c$), the sequence of holes should be approximately the same (as $N \to \infty$) as in the previous modified frozen percolation process. Indeed, the recent work \cite{KMS_FF} for self-destructive percolation \cite{BB_SDP} indicates that it takes a positive time $\delta > 0$ for macroscopic connections outside the hole to reappear, and the next burning event occurs much before that time. In particular, this suggests the existence (hinted in \cite{BB_CMP}) of a $\delta > 0$ for which: w.h.p. (as $N \to \infty$), the origin does not burn on the time interval $[0, t_c + \delta]$.

\appendix

\section{Appendix: additional proofs}

\subsection{Proof of Lemma \ref{lem:diff_theta}} \label{sec:app_diff_theta}

\begin{proof}[Proof of Lemma \ref{lem:diff_theta}]
We consider $\kappa$, $p$ and $p'$ as in the statement, and write
$$\theta(p') - \theta(p) = \PP(\calB),$$
where $\calB := \{0 \lra{p'} \infty, \: 0 \nlra{p} \infty\}$. Let us assume that this event occurs, which implies that there exists a $p$-white circuit surrounding $0$, as well as a $p'$-black infinite path starting from $0$. We can thus introduce the closest vertex $v$ from the origin which lies on both a $p$-white circuit surrounding $0$, and a $p'$-black path from $0$ to $\infty$ (when there are multiple choices, we just pick one in some deterministic way). Note that locally around $v$, we see four disjoint arms: two $p$-white arms (coming from the $p$-white circuit), and two $p'$-black arms (from the $p'$-black path to $\infty$). 

We now distinguish two cases, depending on the distance from $v$ to the origin: we introduce the events
 $$\calB_1 := \{ d(0, v) \leq L(p) \} \quad \text{and} \quad \calB_2 := \{ d(0, v) > L(p) \}.$$

We start by bounding the probability of $\calB_1$. Let $i_{\textrm{max}} := \lceil \log_2 L(p) \rceil$: by dividing the annulus $\Ann_{1,L(p)}$ into the dyadic annuli $A_i = \Ann_{2^{i-1},2^i}$ ($1 \leq i \leq i_{\textrm{max}}$), we obtain
\begin{align} 
\PP(\calB_1) = \sum_{i=1}^{i_{\textrm{max}}} \PP (v \in A_i) & \leq |p'-p| \sum_{i=1}^{i_{\textrm{max}}} |A_i| \PP (0 \lra{p'} \partial \Ball_{2^{i-2}}) \PP(\arm_4^{p',p}(1,2^{i-1})) \PP (\partial \Ball_{2^{i+1}} \lra{p'}\infty) \nonumber\\
& \leq C_1 |p'-p| \sum_{i=1}^{i_{\textrm{max}}} 2^{2i+2} \pi_1(2^{i-2}) \pi_4(2^{i-1}) \pi_1(2^{i+1}, 4L(p)) \label{eq:pf_diff_theta_1.1}\\ 
& \leq C_2 |p'-p| \pi_1(L(p)) \sum_{i=0}^{i_{\textrm{max}}-1} 2^{2i} \pi_4(2^i) \label{eq:pf_diff_theta_1.2}\\
& \leq C_3 |p'-p| L(p)^2 \pi_4(L(p)) \theta(p) \label{eq:pf_diff_theta_1}
\end{align}
for some constants $C_j = C_j(\kappa) > 0$ ($j = 1, 2, 3$) (in \eqref{eq:pf_diff_theta_1.1} we used \eqref{eq:gps}, in \eqref{eq:pf_diff_theta_1.2} we used a combination of \eqref{eq:ext}, \eqref{eq:qmult} and \eqref{eq:Kesten_theta_pi}, while we used \eqref{eq:sum_4arm_bound} in \eqref{eq:pf_diff_theta_1}).
  
Let us turn to $\PP(\calB_2)$. If we now divide $\TT \setminus \Ball_{L(p)}$ into the annuli $A'_i = \Ann_{2^i L(p), 2^{i+1} L(p)}$ ($i \geq 0$), we obtain
\begin{align}
\PP(\calB_2) = \sum_{i \geq 0} \PP(v \in A'_i) & \leq |p'-p| \sum_{i \geq 0} |A'_i| \PP(0 \lra{p} \partial \Ball_{L(p)/2}) \PP(\arm_4^{p,p'}(1,L(p)/2)) \nonumber\\[-4mm]
& \hspace{5cm} \PP(\partial\Ball_{L(p)/2}(v)\lra{p-\text{white}} \partial \Ball_{2^i L(p)}(v)) \nonumber\\[1mm]
& \leq C_4 |p'-p| L(p)^2 \pi_4(L(p)/2) \theta(p) \sum_{i \geq 0} 2^{2i} \exp(-c_2 2^i) \label{eq:pf_diff_theta_2.1}\\
& \leq C_5 |p'-p| L(p)^2 \pi_4(L(p)) \theta(p) \sum_{i \geq 0} 2^{2i} \exp(-c_2 2^i) \label{eq:pf_diff_theta_2.2}\\
& \leq C_6 |p'-p| L(p)^2 \pi_4(L(p)) \theta(p) \label{eq:pf_diff_theta_2}
\end{align}    
for some constants $C_j = C_j(\kappa)$ ($j = 4, 5, 6$), and $c_2$ as in \eqref{eq:exp_decay} (in \eqref{eq:pf_diff_theta_2.1}, we used \eqref{eq:ext} combined with \eqref{eq:Kesten_theta_pi} and \eqref{eq:exp_decay}, while we used \eqref{eq:ext} in \eqref{eq:pf_diff_theta_2.2}). Lemma \ref{lem:diff_theta} then follows, by combining \eqref{eq:pf_diff_theta_1}, \eqref{eq:pf_diff_theta_2} and \eqref{eq:Kesten_L}.
\end{proof}

\subsection{Proof of Lemma \ref{lem:largest_vol}} \label{sec:app_BCKS}

We use the fact that $\chifin(p) := \EE_p \big[ |\cluster(0)| \: ; \: |\cluster(0)| < \infty \big]$ and $\chicov(p) := \sum_{v \in \TT} \Cov_p \big( \ind_{0 \lra{} \infty}, \ind_{v \lra{} \infty} \big)$ satisfy
\begin{equation} \label{eq:equiv_chis}
\chifin(p), \: \chicov(p) \leq c_1 L(p)^2 \theta(p)^2
\end{equation}
for all $p > p_c$ (where $c_1 > 0$ is a universal constant), which is a consequence of \eqref{eq:exp_decay} (see Section 6.4 in \cite{BCKS01}).

Let us introduce some more notation, used only in this section. For a connected subset $\Lambda$ of $\TT$, the connected components inside $\Lambda$ (at time $p$) can be listed by decreasing volume as $(\cluster^{(i)}_{\Lambda,\infty})_{i \geq 1}$ and $(\cluster^{(i)}_{\Lambda,<\infty})_{i \geq 1}$, according to whether they are included in the infinite cluster $\Cinf(p)$ or not, respectively. Clearly, $\lclus_{\Lambda}$ coincides with either $\cluster^{(1)}_{\Lambda,\infty}$ or $\cluster^{(1)}_{\Lambda,<\infty}$, so in particular $\big| \lclus_{\Lambda} \big| \leq \big| \cluster^{(1)}_{\Lambda,\infty} \big| + \big| \cluster^{(1)}_{\Lambda,<\infty} \big|$. Note also that
\begin{equation} \label{eq:sum_clusters}
\big| \Cinf \cap \Lambda \big| = \sum_{i \geq 1} \big| \cluster^{(i)}_{\Lambda,\infty} \big|.
\end{equation}

\begin{lemma} \label{lem:expectation_size}
For some universal constant $c_1 > 0$, we have
\begin{itemize}
\item[(i)] $\EE_p \Big[ \big| \cluster^{(1)}_{\Lambda,\infty} \big| \Big] \leq | \Lambda | \theta(p)$,

\item[(ii)] and $\EE_p \Big[ \big| \cluster^{(1)}_{\Lambda,<\infty} \big| \Big] \leq c_1 | \Lambda |^{1/2} L(p) \theta(p)$.
\end{itemize}
\end{lemma}

\begin{proof}[Proof of Lemma \ref{lem:expectation_size}]
(i) It follows immediately from \eqref{eq:sum_clusters} that
$$\big| \cluster^{(1)}_{\Lambda,\infty} \big| \leq \big| \Cinf \cap \Lambda \big| = \sum_{v \in \Lambda} \ind_{v \lra{} \infty},$$
and we can conclude by taking the expectation of both sides.

(ii) If we introduce $t_{\Lambda} := |\Lambda|^{1/2} L(p) \theta(p)$, we can write
\begin{align*}
\EE_p \Big[ \big| \cluster^{(1)}_{\Lambda,<\infty} \big| \Big] & \leq t_{\Lambda} + \EE_p \Big[ \big| \cluster^{(1)}_{\Lambda,<\infty} \big| \: ; \: \big| \cluster^{(1)}_{\Lambda,<\infty} \big| \geq t_{\Lambda} \Big]\\
& \leq t_{\Lambda} + \sum_{v \in \Lambda} \PP_p \Big( \big| \cluster(v) \big| = \big| \cluster^{(1)}_{\Lambda,<\infty} \big|, \: \big| \cluster(v) \big| \geq t_{\Lambda}, \: v \nlra{} \infty \Big)\\
& \leq t_{\Lambda} + |\Lambda| \PP_p \Big( \big| \cluster(0) \big| \geq t_{\Lambda}, \: 0 \nlra{} \infty \Big).
\end{align*}
We can then conclude by noting that
$$\PP_p \Big( \big| \cluster(0) \big| \geq t_{\Lambda}, \: 0 \nlra{} \infty \Big) \leq \frac{\chifin(p)}{t_{\Lambda}} \leq \frac{c_1 L(p)^2 \theta(p)^2}{|\Lambda|^{1/2} L(p) \theta(p)},$$
using successively the definition of $\chifin$, and \eqref{eq:equiv_chis}.
\end{proof}

We are now in a position to prove the main lemma.

\begin{proof}[Proof of Lemma \ref{lem:largest_vol}]
First, we observe that Lemma \ref{lem:expectation_size} implies that
$$\EE_p \Big[ \big| \cluster^{(1)}_{\Lambda,<\infty} \big| \Big] \leq c_1 (| \Lambda | \theta(p)) \frac{L(p)}{| \Lambda |^{1/2}}.$$
Further, note that for both types in the statement of the lemma, $| \Lambda | \geq n^2$ (since $\Lambda$ contains $b_n$), so
\begin{equation} \label{eq:ineq_vol_Lambda}
\frac{L(p)}{| \Lambda |^{1/2}} \leq \frac{L(p)}{n} \leq \mu.
\end{equation}
Using Markov's inequality, we obtain that
$$\PP_p \Big( \big| \cluster^{(1)}_{\Lambda,<\infty} \big| \geq \ve | \Lambda | \theta(p) \Big) \leq \frac{c_1 \mu}{\ve} \leq \frac{\ve}{10}$$
for $\mu$ small enough. We can thus restrict our attention to $\cluster^{(1)}_{\Lambda,\infty}$ and $\cluster^{(2)}_{\Lambda,\infty}$.

Let us consider $\big| \Cinf \cap \Lambda \big|$: we already noted that $\EE_p \Big[ \big| \Cinf \cap \Lambda \big| \Big] = | \Lambda | \theta(p)$, and we have
\begin{align*}
\Var_p \Big( \big| \Cinf \cap \Lambda \big| \Big) & = \sum_{v,w \in \Lambda} \Cov_p \big( \ind_{v \lra{} \infty}, \ind_{w \lra{} \infty} \big)\\
& \leq \sum_{v \in \Lambda} \sum_{w \in \TT} \Cov_p \big( \ind_{v \lra{} \infty}, \ind_{w \lra{} \infty} \big) = | \Lambda | \chicov(p) \leq c_1 (| \Lambda | \theta(p))^2 \mu^2,
\end{align*}
where the last inequality comes from \eqref{eq:equiv_chis} and \eqref{eq:ineq_vol_Lambda}. Chebyshev's inequality then implies that for $\mu$ small enough,
\begin{equation} \label{eq:vol_cinf}
\PP_p \left( \left( 1 - \frac{\ve}{10} \right) | \Lambda | \theta(p) \leq \big| \Cinf \cap \Lambda \big| \leq \left( 1 + \frac{\ve}{10} \right) | \Lambda | \theta(p) \right) \geq 1 - \frac{\ve}{10},
\end{equation}
which gives the desired upper bound for $\big| \cluster^{(1)}_{\Lambda,\infty} \big|$ (using \eqref{eq:sum_clusters}).

Now, we need to distinguish the two types for $\Lambda$. We first consider $\Lambda = (\tilde{\Lambda})_{(\beta)}$, where $\beta \in (0,\frac{1}{3})$ and $\tilde{\Lambda}$ is a connected component of $\leq C$ $n$-blocks that contains $b_n$. We consider all the horizontal rectangles of the form $[i \mu^{1/2} n, (i+2) \mu^{1/2} n] \times [j \mu^{1/2} n, (j+1) \mu^{1/2} n]$, and all the vertical rectangles of the form $[i \mu^{1/2} n, (i+1) \mu^{1/2} n] \times [j \mu^{1/2} n, (j+2) \mu^{1/2} n]$ ($i$, $j$ integers), which are entirely contained in $\Lambda$. The probability of the event that all of them have a $p$-black crossing in the long direction is at least
$$1 - c_1 (\mu^{-1/2})^2 C e^{-c_2 \mu^{1/2} n / L(p)} \geq 1 - c_1 C \mu^{-1} e^{-c_2 \mu^{-1/2}}$$
for some constants $c_1, c_2 > 0$ (using \eqref{eq:exp_decay}), which is at least $1 - \frac{\ve}{10}$ for $\mu$ small enough. Let us assume that this event indeed occurs, so that the crossings form a net that covers the sub-domain $\Lambda' = (\tilde{\Lambda})_{(\beta + 3\mu^{1/2})}$. We note that all vertices in $\Cinf \cap \Lambda'$ are connected by the net inside $\Lambda$, so that $\big| \cluster^{(1)}_{\Lambda,\infty} \big| \geq \big| \Cinf \cap \Lambda' \big|$. Moreover, for the same reason as for \eqref{eq:vol_cinf}, with probability at least $1 - \frac{\ve}{10}$,
\begin{equation} \label{eq:vol_cinf2}
\big| \Cinf \cap \Lambda' \big| \geq \left( 1 - \frac{\ve}{10} \right) | \Lambda' | \theta(p) \geq \left( 1 - \frac{\ve}{10} \right) (1 - 12 \cdot 3\mu^{1/2}) | \Lambda | \theta(p) \geq \left( 1 - \frac{\ve}{5} \right) | \Lambda | \theta(p)
\end{equation}
(for $\mu$ small enough). This gives the desired lower bound for $\big| \cluster^{(1)}_{\Lambda,\infty} \big|$, and we can then get an upper bound on $\big| \cluster^{(2)}_{\Lambda,\infty} \big|$ from \eqref{eq:sum_clusters}:
$$\big| \cluster^{(2)}_{\Lambda,\infty} \big| \leq \big| \Cinf \cap \Lambda \big| - \big| \cluster^{(1)}_{\Lambda,\infty} \big| \leq \left( 1 + \frac{\ve}{10} \right) | \Lambda | \theta(p) - \left( 1 - \frac{\ve}{5} \right) | \Lambda | \theta(p) = \frac{3 \ve}{10} | \Lambda | \theta(p),$$
with probability at least $1 - \frac{\ve}{5}$ (using \eqref{eq:vol_cinf} and \eqref{eq:vol_cinf2}). Finally, the net provides a circuit as desired, which is connected to $\infty$ with high probability (using once again \eqref{eq:exp_decay}).

In the case when $\Lambda$ is an $(n,\frac{\ve}{2})$-approximable set with $\Ball_n \subseteq \Lambda \subseteq \Ball_{C n}$, we proceed in the same way, by introducing $\Lambda' = (\Dint{\Lambda}{n})_{(3 \mu^{1/2})}$ (note that $\Dint{\Lambda}{n}$ consists of at most $C^2$ $n$-blocks).
\end{proof}

\subsection{Proof of Lemma \ref{lem:coupling}} \label{sec:app_coupling}

\begin{proof}[Proof of Lemma \ref{lem:coupling}]
Let us denote $x_i = \frac{r_i}{1-r_i}$. For notational convenience, we identify $\omega \in \Omega_n$ with the subset $\{ i\in\{1,\ldots,n\}   \: : \: \omega_i=1 \}$.
For every $S \subseteq \{1,\ldots,n\}$,
$$\PP(\omega = S) = \prod_{i \in S} r_i \cdot \prod_{i \in S^c} (1-r_i) = \prod_{i=1}^n (1-r_i) \cdot \sigma_S,$$
with $\sigma_S := \prod_{i \in S} x_i$. Hence, we want to ensure that for every $S$ with $|S|=i$,
$$\PP(\tilde{\omega}_{[i]} = S) = \frac{\sigma_S}{\Sigma_i}, \quad \text{with } \: \Sigma_i = \sum_{\substack{S \subseteq \{1,\ldots,n\}\\ |S|=i}} \sigma_S$$
(where $\Sigma_0 = 1$ by convention). We claim that the desired coupling can be obtained with the following transition probabilities: for every $S$ with $|S| = i <n$, every $j \in S^c$,
$$p_{S,S\cup\{j\}} = \frac{x_j}{\Sigma_{i+1}} \sum_{\substack{T: j \notin T\\ |T| = i}} \frac{\sigma_T}{i+1-|S \cap T|} = \frac{1}{\Sigma_{i+1}} \sum_{\substack{T: j \in T\\ |T| = i+1}} \frac{\sigma_T}{i+1-|S \cap T|}.$$
Since the summand in the last expression is $\leq \sigma_T$, it is clear that $p_{S,S\cup\{j\}} \in [0,1]$. One also has
\begin{align*}
\sum_{j \in S^c} p_{S,S\cup\{j\}} & = \frac{1}{\Sigma_{i+1}} \sum_{j \in S^c} \sum_{\substack{T: j \in T\\ |T| = i+1}} \frac{\sigma_T}{i+1-|S \cap T|}\\
& = \frac{1}{\Sigma_{i+1}} \sum_{T: |T| = i+1} \sum_{j \in S^c \cap T} \frac{\sigma_T}{i+1-|S \cap T|} = \frac{\Sigma_{i+1}}{\Sigma_{i+1}} = 1,
\end{align*}
as desired. Finally, it only remains to be checked that for every $0 \leq i \leq n$, we obtain the right distribution for $\tilde{\omega}_{[i]}$. We proceed by induction over $i$: this clearly holds for $i=0$, and let us assume that it holds for some $0 \leq i <n$. Then for every $T \subseteq \{1,\ldots,n\}$ with $|T| = i+1$,
\begin{align*}
\PP(\tilde{\omega}_{[i+1]} = T) & = \sum_{j \in T} \PP(\tilde{\omega}_{[i]} = T \setminus \{j\}) p_{T \setminus \{j\},T}\\
& = \sum_{j \in T} \frac{\sigma_{T \setminus \{j\}}}{\Sigma_i} \frac{x_j}{\Sigma_{i+1}} \sum_{\substack{U: j \notin U\\ |U| = i}} \frac{\sigma_U}{i+1-|(T \setminus \{j\}) \cap U|},
\end{align*}
using the induction hypothesis. Since $\sigma_{T \setminus \{j\}} x_j = \sigma_T$, we obtain
\begin{align*}
\PP(\tilde{\omega}_{[i+1]} = T) & = \frac{\sigma_T}{\Sigma_{i+1} \Sigma_i} \sum_{U : |U|=i} \sum_{j \in T \cap U^c} \frac{\sigma_U}{i+1-|T\cap U|} = \frac{\sigma_T}{\Sigma_{i+1} \Sigma_i} \Sigma_i = \frac{\sigma_T}{\Sigma_{i+1}},
\end{align*}
which completes the proof of Lemma \ref{lem:coupling}.
\end{proof}

\subsection*{Acknowledgements}

We thank G\'abor Pete for his help with Proposition \ref{prop:theta_over_pi}. We also thank him and Christophe Garban for an early access to an expanded version of \cite{GPS13b}. Part of this work was done at NYU Abu Dhabi, the Institut Henri Poincar\'e in Paris, and the Isaac Newton Institute in Cambridge, and we thank these institutions for their support and their hospitality.

\bibliographystyle{plain}
\bibliography{frozen_perc_volume}

\end{document}